\documentclass[12pt]{amsart}

\usepackage{amssymb, amsthm, amsmath, amsfonts, amsxtra, mathrsfs, mathtools,bm,stmaryrd, colonequals}
\usepackage{bbm}
\usepackage[backref=page]{hyperref}
\hypersetup{
	colorlinks   = true,          
	urlcolor     = violet,          
	linkcolor    = violet,          
	citecolor   = blue}             
\usepackage{graphicx}
\usepackage{subcaption}
\usepackage{xcolor}
\usepackage{tikz-cd}
\usepackage{enumitem}
\usepackage{url}
\usepackage{float}
\setlist[description]{leftmargin=\parindent,labelindent=\parindent}
\usetikzlibrary{matrix,arrows,decorations.pathmorphing}

\usepackage{physics}
\usepackage{amsmath}
\usepackage{tikz}
\usepackage{mathdots}
\usepackage{yhmath}
\usepackage{cancel}
\usepackage{color}
\usepackage{siunitx}
\usepackage{array}
\usepackage{multirow}
\usepackage{amssymb}
\usepackage{gensymb}
\usepackage{tabularx}
\usepackage{extarrows}
\usepackage{booktabs}
\usetikzlibrary{fadings}
\usetikzlibrary{patterns}
\usetikzlibrary{shadows.blur}
\usetikzlibrary{shapes}

\usepackage[letterpaper, top=3cm, bottom=3cm,left=2.7cm, right=2.7cm, heightrounded,bindingoffset=0mm, marginparwidth=22mm]{geometry}
\theoremstyle{plain}
\newtheorem{theorem}{Theorem}[subsection]
\newtheorem{proposition}[theorem]{Proposition}
\newtheorem{lemma}[theorem]{Lemma}
\newtheorem{corollary}[theorem]{Corollary}

\newtheorem{fact}[theorem]{Fact}
\newtheorem{mainthm}{Theorem}

\theoremstyle{definition}

\newtheorem{definition}[theorem]{Definition}
\newtheorem{construction}[theorem]{Construction}

\newtheorem{observation}[theorem]{Observation}

\theoremstyle{remark}
\newtheorem{remark}[theorem]{Remark}
\newtheorem{question}[theorem]{Question}

\makeatletter
\def\@tocline#1#2#3#4#5#6#7{\relax
	\ifnum #1>\c@tocdepth 
	\else
	\par \addpenalty\@secpenalty\addvspace{#2}%
	\begingroup \hyphenpenalty\@M
	\@ifempty{#4}{%
		\@tempdima\csname r@tocindent\number#1\endcsname\relax
	}{%
		\@tempdima#4\relax
	}%
	\parindent\z@ \leftskip#3\relax \advance\leftskip\@tempdima\relax
	\rightskip\@pnumwidth plus4em \parfillskip-\@pnumwidth
	#5\leavevmode\hskip-\@tempdima
	\ifcase #1
	\or\or \hskip 1em \or \hskip 2em \else \hskip 3em \fi%
	#6\nobreak\relax
	\dotfill\hbox to\@pnumwidth{\@tocpagenum{#7}}\par
	\nobreak
	\endgroup
	\fi}
\makeatother

\def\B{\mathcal{B}}
\newcommand{\bl}[2]{\mathrm{Bl}_{#1}{#2}}

\def\A{\mathcal{A}}
\def\AA{\mathbb{A}}
\def\aut{\lambda^{\dagger}}
\def\Mbar{\overline{\mathcal{M}}}

\newcommand{\T}[1]{\widetilde{#1}}

\def\D{\delta}
\def\F{\mathcal{F}}
\def\Domega{\delta_{\omega}}

\def\PP{\mathbb{P}}
\def\pt{\mathrm{pt}}
\def\CC{\mathbb{C}}
\def\Z{\mathbb{Z}}
\def\N{\mathbb{N}}
\def\R{\mathbb{R}}
\def\Rgeq{\mathbb{R}_{\geq 0}}
\def\Gm{\mathbb{G}_\mathrm{m}}
\def\Ga{\mathbb{G}_\mathrm{a}}

\def\S{\Sigma}
\def\Sn{\Sigma^n}

\def\T{\mathtt{T}}

\def\tropmod{\Pi_n(\Sigma)}
\def\tropfam{\Pi_n^+(\Sigma)}

\def\cone{\mathrm{cone}}
\def\gridexp{\mathcal{E}_n(X|D)}
\def\gridexpfam{\mathcal{E}_n^+(X|D)}
\def\fam{\mathfrak{X}}

\def\X{\mathcal{X}}
\def\logProd{(X|D)^{[n]}}
\newcommand{\logProdexample}[3]{(#1|#2)^{[#3]}}
\def\Conf{\mathrm{Conf}_n(X)}
\def\FM{\mathrm{FM}_n(X)}
\newcommand{\FMexample}[2]{\mathrm{FM}_#1(#2)}
\def\logConf{\mathrm{Conf}_n(X \backslash D)}

\def\logFM{\mathrm{FM}_n(X|D)}
\def\logFMfam{\mathrm{FM}_n(X|D)^+}
\def\piFM{\pi_{\mathrm{FM}}}
\newcommand{\logFMexample}[3]{\mathrm{FM}_#1(#2 | #3)}
\def\oneton{\{1,\ldots,n\}}
\newcommand{\rthOrthant}[1][r]{\mathbb{R}_{\geq 0}^{#1}}
\newcommand{\cal}[1]{\mathcal{#1}}
\def\tropmoddegen{\Pi_n(\Delta)}
\def\tropfamdegen{\Pi_n^+(\Delta)}
\def\tropmodpoly{\mathsf{P}_n(\Delta)}
\def\tropfampoly{\mathsf{P}_n^+(\Delta)}
\def\tropmoddegenFM{\Pi_n^{\text{FM}}(\Delta)}
\def\tropfamdegenFM{\Pi_n^{\text{FM},+}(\Delta)}
\def\simpexp{\mathcal{E}_n(\Delta)}
\def\simpexpfam{\mathcal{E}_n^+(\Delta)}
\def\simpexpfamW{\mathcal{E}_n^+(W/B)}
\def\pointsdegen{W_{\omega}^{[n]}}
\def\pointsdegenfam{\mathfrak{W}}
\def\FMdegen{\mathrm{FM}_n(W/B)}
\def\FMdegenfam{\mathrm{FM}_n(W/B)^+}

\DeclareMathOperator{\Spec}{Spec}
\DeclareMathOperator{\trop}{trop}
\DeclareMathOperator{\val}{val}

\newcommand{\pullback}[1]{\arrow[#1, phantom, "\square" , color=black]}
\newcommand{\fspullback}[1]{\arrow[#1, phantom, "\square_{\mathrm{fs}}" , color=black]}

\title{Logarithmic Fulton--MacPherson configuration spaces}
\author{Siao Chi Mok}

\begin{document}
	
	\begin{abstract}
		Using techniques in logarithmic geometry, we construct a logarithmic analogue of the Fulton--MacPherson configuration spaces. We similarly construct a logarithmically smooth degeneration of the Fulton--MacPherson configuration spaces. Both constructions parametrise point configurations on certain target degenerations arising from both logarithmic geometry and the original Fulton--MacPherson construction. The degeneration satisfies a ``degeneration formula" -- each irreducible component of its special fibre can be described as a proper birational modification of a product of logarithmic Fulton--MacPherson configuration spaces.
	\end{abstract}

	\maketitle

	\tableofcontents
	
\section{Introduction}
Let $X$ be a smooth projective variety over an algebraically closed field $k$. In the seminal paper \cite{FM}, Fulton and MacPherson constructed a simple normal crossings compactification $\FM$ of the ordered configuration space of $n$ distinct points $\Conf$. The compactification $\FM$ is constructed via an iterated blow-up of $X^n$ along diagonals to separate the points. The boundary $\FM \backslash \Conf$ parametrises point configurations on \emph{Fulton--MacPherson (FM) degenerations}; when points collide with each other at $x \in X$, in the limit a $\PP(T_x X \oplus \mathbbm{1})$-bubble is attached to $X$ and the colliding points instead land in the bubble, on the smooth locus. This procedure is repeated until all points have been separated. 

Among numerous applications, the Fulton--MacPherson space $\FM$ was used by Kim, Kresch and Oh in \cite{stable_ramified} to build a moduli space of \emph{stable ramified maps}, i.e. maps from curves with targets some Fulton--MacPherson degenerations of $X$, and to define and study \emph{unramified Gromov--Witten theory}. It was conjectured by Pandharipande \cite[Conjecture 5.2.1]{stable_ramified} that when $X$ is Fano, these unramified Gromov--Witten invariants are equal to Gopakumar--Vafa invariants, and this conjecture was later proved by Nesterov in \cite{uGWequalsGV} by obtaining a wall-crossing formula between Gromov--Witten invariants and unramified Gromov--Witten invariants.

Two questions about Fulton--MacPherson spaces motivate this present work. 
\begin{question} \label{question:logFM}
	A compactification $X$ of a non-proper variety $U$ with simple normal crossings boundary $D \colonequals X \backslash U$ is called a \emph{simple normal crossings compactification}. Given a non-proper variety $U$ with a simple normal crossings compactification $(X,D)$, how can we construct a simple normal crossings compactification of $\mathrm{Conf}_n(U)$?
\end{question}
\begin{question} \label{question:FMdegen}
	Given a degeneration $\omega\colon W \to B$ of $X$ over a curve $B$, where the special fibre is a simple normal crossings divisor, how does the space $\FM$ degenerate?
\end{question}

In characteristic zero, it follows from results of \cite{hironaka} and \cite{nagata} that every variety $U$ admits a simple normal crossings compactification. A degeneration as described in Question~\ref{question:FMdegen} is called a \emph{simple normal crossings} or \emph{semistable} degeneration.

Note that there exist quasi-projective varieties $U$ without a compactification by a smooth divisor. A result by Danilov and Thuillier \cite{danilov, htpy-dualcomplex} says that the homotopy type of the dual complex of the boundary divisor (the \emph{boundary complex}) is independent of the choice of compactification. Since the boundary complex of a compactification with smooth boundary is a point, any quasi-projective variety $U$ admitting a compactification with non-contractible boundary complex cannot have a compactification with a smooth boundary. An example is $(\CC^*)^2$ with boundary complex $\Delta(\PP^2 \backslash (\CC^*)^2)$ homotopic to $S^1$. Thus, it is desirable to work with the full generality of simple normal crossings compactifications.

With the initial compactification $(X,D)$ of $U$ as input, our goal is to construct a simple normal crossings compactification $\logFM$ of $\logConf$.

Our plan is to proceed in two steps:
\begin{enumerate}
	\item Construct a moduli space $\logProd$ of points on certain degenerations of $X$ along $D$ (\emph{grid expansions}), to separate the points from the boundary;
	\item Iteratively blow-up $\logProd$ to separate the points from each other, in the style of \cite{FM}. The resulting scheme is $\logFM$, which parametrises points on Fulton--MacPherson degenerations of the grid expansions (\emph{FM grid expansions}).
\end{enumerate}

In the special case when $D$ is smooth, Question~\ref{question:logFM} has been answered via a similar approach by Kim and Sato \cite{KS}. The general case requires us to use recent ideas from logarithmic and tropical geometry.

Given a degeneration $\omega\colon W \to B$ where its special fibre has smooth singular locus $D$, Abramovich and Fantechi \cite{abram_fantechi} constructed a moduli space of points on the fibres of \emph{expanded degenerations}. Building on this, Routis \cite{weightedFM} constructed a degeneration of $\FM$, answering Question~\ref{question:FMdegen} when $D$ is smooth.

\subsection{Main results}
We answer Question~\ref{question:logFM} by first constructing a moduli space $\logProd$ of \emph{stable $n$-pointed grid expansions} in Section~\ref{section:logProd}, then iteratively blow-up $\logProd$ to construct the desired compactification $\logFM$. The compactification $\logFM$ comes with a family $\piFM\colon \logFMfam \to \logFM$. In Section~\ref{section:logFM} we prove:

\begin{mainthm}[Theorem~\ref{thm:logFM_fibres}]
    Let $(X,D)$ be a simple normal crossings pair.
    \begin{enumerate}
    	\item The scheme $\logFM$ is a simple normal crossings compactification of $\logConf$. 
    	\item The morphism $\piFM$ is a family of stable $n$-pointed FM grid expansions.
    	\item The scheme $\logFM$ admits a stratification by combinatorial types.
    \end{enumerate} 
\end{mainthm}

We call $\logFM$ the \emph{logarithmic Fulton--MacPherson configuration space} of $(X,D)$. The \emph{combinatorial type} of an \emph{FM grid expansion} records the isomorphism type of the underlying scheme and for each marked point, the label of the component containing it.

These results generalise the constructions by Kim--Sato \cite{KS}, and Abramovich--Fantechi \cite{abram_fantechi}. While these constructions all use the notion of \emph{expansions}, stemming from the work of Gieseker and Jun Li \cite{gieseker, junli}, differences remain. In constructing $\logProd$, Kim--Sato constructed $\logProd$ using the theory of wonderful compactifications, whereas both Abramovich--Fantechi and this paper use an appropriate \emph{stack of expanded degenerations}. 

The stack of expansions also appears in the logarithmic curve counting theories of Maulik and Ranganathan in \cite{logGWexp, MR20, log_enum}. When the simple normal crossings divisor $D$ has multiple components, in these papers there are usually many choices of expansions depending on the tropical types of the curves and their maps to $(X,D)$, therefore there is no canonical stack of expansions. However, in this paper a \emph{canonical} choice of expansions can be made, and universal weak semistable reduction \cite[Theorem 0.3]{semistable} \cite[Section 3.2]{uni_semistable} enables a neat construction of the stack of expansions. We call these canonically chosen expansions \emph{grid expansions}, as the components of these expansions are $(\PP^1)^k$-bundles arranged in a grid. These grid expansions are constructed from a canonical choice of \emph{polyhedral subdivisions} of the tropicalisation of $(X,D)$.

The construction of $\logFM$ has an interesting feature: when $\logFM$ is given the divisorial logarithmic structure with respect to its boundary divisors, the tropicalisation of $\logFM$ is, by construction, the moduli space of \emph{planted forests} satisfying some conditions. See Figures~\ref{fig:forest_div_exp} and \ref{fig:forest_div_fm} for examples of planted forests. In particular, when the boundary divisor $D$ is empty, the tropicalisation is the space of \emph{rooted trees}, as described in \cite[Section 2]{FM}.

Given a simple normal crossings degeneration $\omega \colon W\to B$ of $X$, we answer Question~\ref{question:FMdegen} by first constructing a moduli space $\pointsdegen$ of \emph{stable $n$-pointed expanded degenerations} in Section~\ref{section:points_degen}, then iteratively blow-up $\pointsdegen$ to construct the desired degeneration $\FMdegen \to B$ of $\FM$ in Section~\ref{section:degeneration}. This degeneration comes with a family $\piFM^W\colon \FMdegenfam \to \FMdegen$. We prove the following:

{\begin{mainthm}[Theorem~\ref{thm:FMdegen}]\hfill
	\begin{enumerate}
		\item The morphism $\FMdegen \to B$ is a proper, flat, logarithmically smooth degeneration of $\FM$ with reduced fibres. 
		\item The morphism $\piFM^W\colon \FMdegenfam \to \FMdegen$ is a flat family of stable FM expanded degenerations with $n$ sections $\mathbf{s}_1, \mathbf{s}_2, \ldots, \mathbf{s}_n$, where the sections are supported on the smooth locus of each fibre.
	\end{enumerate}
\end{mainthm}}

Like the construction in \cite[Section 8]{log_enum}, the degeneration $\FMdegen\to B$ comes with a \emph{degeneration formula}. A \emph{rigid combinatorial type} $\rho$ is a combinatorial type of FM expansions whose tropical moduli is a vertex. Let $\cal{S}_{\rho,1}$ be the corresponding $n$-marked simplex-lattice subdivision, and $Y_{\rho}$ be the corresponding expansion.

We prove in Section~\ref{subsection:degen_formula} the following:

\begin{mainthm}[Corollary \ref{thm:degen_formula}] \label{thm:degen_formula_intro} \hfill
	\begin{enumerate}
		\item The irreducible components of the special fibre of $\FMdegen \to B$ biject with rigid combinatorial types $\rho$.
		\item The irreducible component $\FMdegen(\rho)$ is a (proper birational) modification of a product of logarithmic FM spaces of components of $Y_{\rho}$ relative to their boundaries.
	\end{enumerate}
	
\end{mainthm}

Expansions are constructed using methods from toric geometry, which is defined for schemes over any base ring. Therefore, the constructions of $\logFM$ and $\FMdegen$ extend to a scheme $X$ smooth over any base ring and $D$ a reduced, simple normal crossings divisor, and any simple normal crossings degeneration of $X$. Many results persist with minor adjustments, such as replacing ``points" by ``geometric points". For simplicity, we work with a smooth projective variety $X$ over an algebraically closed field, and $D$ a simple normal crossings divisor on $X$.

\subsection{Potential applications}
An application is to use a version of $\logFM$, constructed over the stack of expansions from \cite{MR20}, to define and study the \emph{logarithmic unramified Gromov--Witten theory} of a pair $(X,D)$ -- this is work in progress \cite{expram}. The idea behind the degeneration package of $\FM$ developed in this paper is expected to lead to a degeneration formula for unramified Gromov--Witten theory of Kim--Kresch--Oh \cite{stable_ramified}. We hope that this will lead to new applications. 

Another application pertains to the intersection theory of classical and logarithmic Fulton--MacPherson spaces. For some suitably defined tautological classes $\psi_{i,j}$ (\emph{higher $\psi$-classes}) and for any polynomial $p(x_{i,j})$, one can consider an integral of the form \[N_n^X \colonequals \int_{[\FM]} p(\psi_{i,j}),\] and a generating series \[Z(X,q) \colonequals 1+ \sum_{n\geq 1} N_n^X \frac{q^n}{n!}.
\]
Similarly to \cite[Section 0.6]{cobordism}, the degeneration formula (Theorem~\ref{thm:degen_formula_intro}) suggests that $Z(X,q)$ is a cobordism invariant, thus like in \cite[Conjecture 1]{cobordism} there is a universal polynomial $Q_{\dim(X),n,p(x_{i,j})}$ such that \[N_n^X = Q_{d,n,p(x_{i,j})}(c_1(X), c_2(X), \ldots, c_{\mathrm{top}}(X)),\] where $c_k(X) = c_k(T_X)$ is the $k$th Chern number of $X$. The existence of this universal polynomial can also be deduced from the explicit presentation of the cohomology of $\FM$ \cite[Theorem 6]{FM}. The logarithmic analogue will be an interesting problem to study, and similar formulae for the classical Fulton--MacPherson space have been obtained by Nesterov in \cite{hilb_fm}.

\subsection{Proximate moduli}
As noted before, the works of Kim--Sato \cite{KS}, Abramovich--Fantechi \cite{abram_fantechi} and Routis \cite{weightedFM} are important precursors of the present work. 

\subsubsection{Fulton--MacPherson space for relative curves}
Consider the forgetful morphism \[\Mbar_{g,n+d}^{\log} \to \Mbar_{g,d}^{\log}\] between moduli spaces of stable log curves, which is semistable. The geometric fibre over the point $(C, p_{n+1}, \ldots, p_{n+d}) \in \Mbar_{g,d}^{\log}$ is then a logarithmically smooth compactification of $\mathrm{Conf}_n(C^{\text{sm}}\backslash \{p_{n+1}, \ldots, p_{n+d}\})$. When $C$ is smooth, this recovers $\logFMexample{n}{C}{p_{n+1} + \cdots + p_{n+d}}$. As a generalisation, Bejleri and Landesman constructed normal crossings compactifications of the Hurwitz spaces of $G$-covers of $C$ and of the universal curve $\mathcal{C} \to \Mbar_{g,d}^{\log}$, see \cite[Theorem B.1.2, Theorem B.3.2, Remark B.3.4]{landesman}.
\subsubsection{Weighted Fulton--MacPherson spaces}
In \cite{weightedFM}, Routis constructs a space which parametrises stable weighted configurations of points, generalising both the classical Fulton--MacPherson spaces and the Hassett spaces of weighted pointed rational curves. In the setting of \cite{abram_fantechi}, a degeneration of weighted Fulton--MacPherson spaces has also been constructed in \cite{weightedFM}. We expect weighted analogues of our constructions to exist.

\subsubsection{Toric configuration spaces}
The work of Nabijou \cite{toric_conf} studies compactifications of point configurations on an algebraic torus $(\mathbb{C}^*)^r$. The combinatorial techniques in both \cite{toric_conf} and this paper are similar, and many approaches in this paper draw inspiration from it. A key difference is that the points are allowed to collide arbitrarily in the toric configuration spaces, but the points are separated in $\logFM$.

Nonetheless, we can still compare $\logProd$ with the toric configuration spaces. When $(X,D)$ is a toric pair, the space $\logProd$ is an example of the toric configuration spaces. An interesting example can be found in Section~\ref{section:losev_manin}. However, we construct $\logProd$ by fixing a choice of subdivisions, whereas Nabijou studies the collection of all subdivisions with some interesting modular interpretations in combinatorics. There also exist toric configuration spaces which cannot be globally expressed in the form $\logProd$ for any toric variety $X$. 

\subsubsection{Degenerations of Hilbert scheme of points}
Given any semistable degeneration $W\to B$ of a surface $X$, the work of Tschanz \cite{calla} gives an explicit local construction of a semistable degeneration of the Hilbert schemes of points of $X$. This construction considers 0-dimensional subschemes on a fixed choice of expansions of the special fibre $W_0$, which parallels the degeneration construction in this paper. A key difference lies in the choice of expansions; the choice in \cite{calla} makes the degenerations smooth but not symmetric, whereas the degenerations in this paper can be singular but are symmetric. 

Like the degeneration construction in this work, the construction in \cite{calla} is local. A recent follow-up by Shafi and Tschanz \cite{qaasim_calla} describes a strategy to globalise the local construction in \cite{calla} to a given arbitrary type III degeneration of a K3 surface, giving a good type III degeneration of $\mathrm{Hilb}^m(\mathrm{K3})$. These are the first explicit examples of a good type III degeneration of hyperk\"ahler varieties with dimension greater than two.


\subsubsection{Pointed trees of projective spaces and polypermutohedral varieties}
The works \cite{trees_projective} and \cite{polymatroids} feature moduli which utilise classical, weighted and logarithmic Fulton--MacPherson spaces. The work \cite{polymatroids} compactifies the space of distinct weighted points on a flag via weighted Fulton--MacPherson spaces, which recovers the construction in \cite{trees_projective}. In addition, the work \cite{polymatroids} also compactifies the space of not necessarily distinct points on a flag via the construction of Kim--Sato \cite{KS}, denoted $\logProd$ in this paper. Both constructions admit meaningful interpretations in combinatorics, and it would be interesting to explore possible interactions between our works.

In the construction of $\logFM$, there is a significant increase in complexity compared to Kim--Sato when the divisor $D$ is not smooth, so it is not clear how to construct $\logFM$ using wonderful compactifications. Meanwhile, we note that the stack of expansions in \cite{MR20} is only defined up to certain birational transformations. In comparison, this paper benefits from a canonical choice of expansions, namely \emph{grid expansions}, enabling a more detailed and granular description of the stack of grid expansions. The key insight in this paper is the systematic use of tropical and logarithmic geometry, especially Artin fans to construct the stack of grid expansions.

\subsection{Acknowledgements}
I am greatly indebted to my supervisor Dhruv Ranganathan for his invaluable input and relentless support. I would like to extend my gratitude to Dan Abramovich and Barbara Fantechi for their interest in this topic and inspiring discussions about Fulton--MacPherson degenerations. I have greatly benefited from informative conversations with Samir Canning, Renzo Cavalieri, Melody Chan, Robert Crumplin, Patrick Kennedy-Hunt, Navid Nabijou, Denis Nesterov, Qaasim Shafi, Bernd Siebert, Terry Song, Calla Tschanz, Ajith Urundolil Kumaran and Jonathan Wise. This work was supported by the Cambridge Trust, DPMMS and Newnham College Scholarship. \\

\noindent\textbf{Terminologies.} The \textit{boundary} of a compactification $\overline{\mathcal{M}}$ of a space $\mathcal{M}$ is the complement $\overline{\mathcal{M}}\backslash \mathcal{M}$. A compactification is \textit{simple normal crossings} when the boundary is a simple normal crossing divisor. The pair $(X,D)$ refers to a smooth projective variety $X$ and a simple normal crossings divisor $D = D_1+ \cdots + D_r$, such that any intersection stratum $\cap_{i \in I} D_i$ is connected. We only consider \emph{fine and saturated} \emph{Zariski} logarithmic structures. The word \emph{cone} will refer to a strongly convex rational polyhedral cone. 

\section{Moduli space of points on expansions} \label{section:logProd}
The main goal of this section is to construct a modular compactification $\logProd$ of $(X \backslash D)^n$, which parametrises points on \emph{grid expansions} of $X$.

\subsection{Preliminaries on tropicalisation and logarithmic modifications} \label{subsection:logmod}

In this section, we review some techniques used in constructing logarithmic moduli spaces. We refer the reader to \cite[Section 6]{CCUW}, \cite[Section 1]{MR20}, \cite[Section~1]{log_enum}, \cite[Section 4]{functrop} and \cite{Ogus} for standard definitions and results in logarithmic geometry.

\begin{definition}[Tropicalisation of a pair]
    Let $(X,D)$ be a smooth projective variety $X$ equipped with a simple normal crossings divisor $D$ with irreducible components $D_1, \dots, D_r$ such that any intersection $\cap_{i \in I} D_i$ is connected. Then the \textit{tropicalisation} $\Sigma_X \subset \R_{\geq 0}^r$ is a rational polyhedral cone complex consisting of 1-dimensional cones $e_i$, which are the rays of $\R_{\geq 0}^r$, and $k$-dimensional cones spanned by $e_{i_1},\ldots,e_{i_k}$ whenever the intersection of $D_{i_1},\dots, D_{i_k}$ is non-empty. 
\end{definition}

\begin{definition}[Artin fan of a pair]
   Given a simple normal crossings pair $(X,D)$ with tropicalisation $\Sigma_X \subset \R_{\geq 0}^r$, its \textit{Artin fan} $\A_X$ is defined as the open substack of the toric quotient stack $[\mathbb{A}^r/\mathbb{G}_m^r]$ formed by taking the substack corresponding to the cones in $\R_{\geq 0}^r$ that are present in $\Sigma_X$.
\end{definition}

Sections 2 and 6 of \cite{CCUW} introduced the 2-category of cone stacks and the 2-category of Artin fans, and constructed a functor $a^*$ from the 2-category of cone stacks to the 2-category of Artin fans. Every cone complex is naturally a cone stack.

\begin{theorem}[\cite{CCUW} Theorem 6.11]\label{thm:CCUW}
    The functor $a^*$ is an equivalence of 2-categories.
\end{theorem}

There is a canonical morphism $X \to \mathcal{A}_X \hookrightarrow [\AA^r/\Gm^r]$, given by line bundle and section pairs $(L_i, s_i)$ corresponding to the effective Cartier divisor $D = D_1 + \cdots + D_r$. 

Concretely, since $D$ is a simple normal crossings divisor, at each point $p$, there exist local coordinates $x_1, \ldots, x_d$ such that $D$ is of the form $x_{i_1} \ldots x_{i_k} = 0$, and each $x_{i_j}$ can be taken to be the stalk of $s_j$ at $p$. Therefore, there exists an open neighbourhood $U_p$ of $p$ such that $U_p$ is canonically isomorphic to an open neighbourhood of $(\underline{0}, \underline{1}) \in \AA^k \times \Gm^{d-k}$. We call $U_p$ a \emph{toric chart} at $p$. This gives a map \[U_p \xrightarrow{\text{\'etale}} \AA^k \times \Gm^{d-k} \to [\AA^k \times \Gm^{d-k}/\Gm^k] \cong [\AA^k/\Gm^k] \hookrightarrow \mathcal{A}_X, \] where the last map corresponds via Theorem~\ref{thm:CCUW} to a face inclusion $\sigma \hookrightarrow \Sigma_X$, where $\sigma$ is the cone corresponding to the intersection $D_{i_1} \cap \dots \cap D_{i_k}$ containing $p$. We note that this map glues to a global map $X \to \A_X$, and we note that it is flat.

In the language of cone complexes, this map assigns the strata of $X$ to the corresponding cones. Here, a stratum of $X$ is of the form \[D_I^\circ \colonequals \{x \in X \;|\; x \in D_i \text{ if and only if } i \in I\}.\]

\begin{definition} 
A (conical) \emph{subdivision} $\widetilde{\Sigma}$ of a cone complex $\Sigma$ is a morphism $\widetilde{\Sigma} \to \Sigma$ that is bijective on the supports, such that the lattice points of each cone $\tau \in \Tilde{\Sigma}$ are exactly the intersection of the lattice points of $\Sigma$ with the image of $\tau$.
\end{definition}

\noindent \textbf{Functoriality.} Given two simple normal crossings pairs $(X,D)$ and $(Y,E)$ and a map of pairs $f \colon (X,D) \to (Y,E)$, satisfying $f^{-1}(E) \subset D$, then functoriality holds for tropicalisations and Artin fans: there are induced maps $\Sigma_X \to \Sigma_Y$ and $\A_X \to A_Y$. 

Functoriality holds for maps between logarithmic schemes $X \to Y$ when $X$ is logarithmically smooth of characteristic zero \cite[Proposition~2.8]{decomp}, and when $X, Y$ are \emph{atomic}, i.e. admits a global toric chart \cite[Lemma~A.18]{lifting_rational}, but does not hold in general \cite{skeletons-fans}.

\noindent \textbf{Logarithmic modification.} A subdivision $\widetilde{\Sigma_X} \to \Sigma_X$ corresponds to a modification of Artin fans $\widetilde{\A_X} \to \A_X$ via Theorem~\ref{thm:CCUW}, and by pulling back the modification along the flat morphism $X \to \A_X$ we obtain a modification $\widetilde{X} \to X$ (see \cite[Section~1.4]{log_enum}). 
\begin{equation*}
	\begin{tikzcd}
		\widetilde{X} \arrow[r] \arrow[d] \arrow[dr, phantom, "\square"] & \widetilde{\A_X} \arrow[d]\\
		X \arrow[r] & \A_X
	\end{tikzcd} 
\end{equation*}

More concretely, the subdivision of cone complexes $\widetilde{\Sigma_X} \to \Sigma_X$ induces toric modifications of toric charts of $X$ of the form $\AA^k \times \Gm^{d-k}$, and one can check that these local modifications glue to a global modification of $X$, as the subdivision of cone complexes is compatible with face morphisms of cones.

\subsection{Tropicalisation map and transversality} \label{section:tropicalisation}

In this subsection, we give a description of the boundary points of $\logProd$ in terms of tropical geometry. When a family of $n$ points in $X$ approaches $D$, tropical geometry suggests the following roadmap: 

\begin{enumerate} 
    \item Suppose the Cartier divisor $D = D_1 + \dots + D_r$ is locally described by functions $f_1, \dots, f_r$. Let $\Spec R$ be a discrete valuation ring with valuation $\val$. For a one-parameter family $p\colon \Spec R \to X^n$, given by $(p_1, p_2, \dots, p_n)$, define the \textit{$i$th tropicalisation map} as:
    \begin{align*}
        \trop_i\colon  \;X(R) &\to \Sigma_X \subseteq \rthOrthant \\
        p_i &\mapsto (\val(p_i^*f_1),\dots,\val(p_i^*f_r)).
    \end{align*}
    This map records the order of vanishing of the functions $f_j$ along $p_i$. Taken for all $i$, we have a \textit{tropicalisation map}, written as $\trop\colon X^n(R) \to \Sigma_X^n$. The image $\trop(p)\in \Sigma_X^n$ describes $n$ points on $\Sigma_X$. 
    \item We seek a polyhedral subdivision $\Sigma'_X$ of $\Sigma_X$ such that the $n$ points are vertices on $\Sigma'_X$.
    \item The polyhedral subdivision $\Sigma'_X$ corresponds to an \emph{expansion} of $X$ (defined in Definition~\ref{def:grid_exp}), where a \textit{transversality} condition is satisfied -- the strict transform of $p$ intersects the expanded $X$ in the smooth locus. The new limit of $p$ on the expansion will be a typical element in our moduli problem, which will be described in Section~\ref{section:logProd_moduli}. \label{transversality}
\end{enumerate}

\begin{figure}[!htb]
	\centering
	\def\svgwidth{.5\textwidth}
	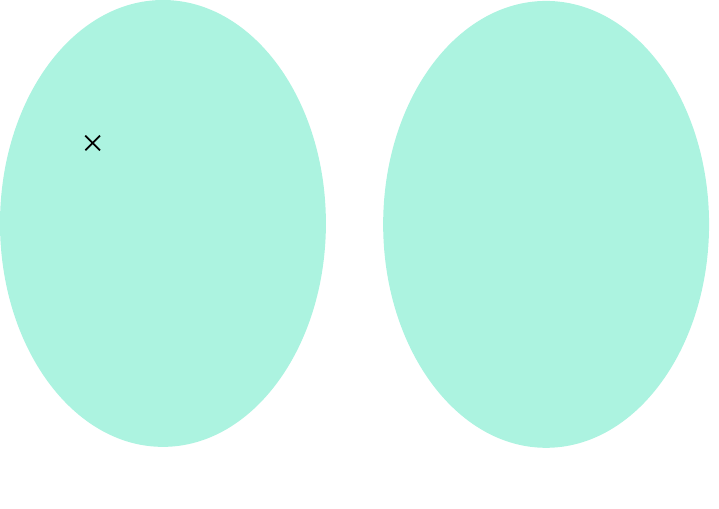
	\caption{A family of 2 points approaching the smooth locus of an expansion.}
	\label{fig:gridexp}
\end{figure}

In subsequent subsections, we globalise the above construction to form a moduli space. The \emph{transversality} mentioned in Step~(\ref{transversality}) is due to a result of Tevelev \cite{tevelev} for toric varieties $X$ and Ulirsch \cite[Theorem~1.2]{ulirsch} for general snc pairs $(X,D)$. 

This can also be understood in terms of blowing up in the special fibre of the trivial family over $\Spec k[[t]]$: if the power series $p_i^*f_j$ has order of vanishing $a$, then blowing up the family in suitable ideals and taking the strict transform of $p_i$ factors out the $t^a$ term, resulting in a non-vanishing power series.

\subsection{Marked grid subdivision of a cone complex} \label{section:gridsub} We proceed with the above plan.

\begin{definition}[Marked polyhedral subdivisions of $\Rgeq$]
    An $n$-marked polyhedral subdivision $(\cal{P}, m)$ of $\Rgeq$ is a polyhedral subdivision $\mathcal{P}$ with a function $m\colon  \oneton \to V(\cal{P})$, where $V(\cal{P})$ is the set of vertices in $\cal{P}$, and $m$ is surjective away from the origin in $V(\cal{P})$.
\end{definition}

\begin{definition}[Marked grid subdivisions of $\S_X$]
    A \emph{grid subdivision} $\cal{P}$ of $\Sigma_X$ is the intersection of $\Sigma_X$ with a product of polyhedral complexes $\{\cal{P}_i\}_{1\leq i \leq r}$, \[
    \cal{P} = \Sigma_X \cap \prod_{1\leq i \leq r} \cal{P}_i,
    \] where each $\cal{P}_i \to \S_{(i)}$ is a polyhedral subdivision of the ray $\S_{(i)}\cong \Rgeq$ with marking function $m_i\colon \oneton \to V(\cal{P}_i)$. 
    
    An \emph{$n$-marked grid subdivision} $(\cal{P}, m)$ of $\S_X$ is a grid subdivision $\mathcal{P}$ of $\S_X$ with a function $m \colon \oneton \to V(\cal{P})$, where $V(\cal{P})$ is the set of vertices in $\cal{P}$. The marking function $m$ must satisfy $\mathrm{pr}_i\circ m = m_i$ for all $i = 1,2, \dots, m$, where $\mathrm{pr}_i$ is the projection map from $\S_X$ to the $i$th ray $\S_{(i)}$ of $\S_X$.
\end{definition}

\begin{figure}[!htb]
	\centering
	\begin{minipage}{.35\textwidth}
		\centering
		\def\svgwidth{.9\textwidth}
\begingroup%
  \makeatletter%
  \providecommand\color[2][]{%
    \errmessage{(Inkscape) Color is used for the text in Inkscape, but the package 'color.sty' is not loaded}%
    \renewcommand\color[2][]{}%
  }%
  \providecommand\transparent[1]{%
    \errmessage{(Inkscape) Transparency is used (non-zero) for the text in Inkscape, but the package 'transparent.sty' is not loaded}%
    \renewcommand\transparent[1]{}%
  }%
  \providecommand\rotatebox[2]{#2}%
  \newcommand*\fsize{\dimexpr\f@size pt\relax}%
  \newcommand*\lineheight[1]{\fontsize{\fsize}{#1\fsize}\selectfont}%
  \ifx\svgwidth\undefined%
    \setlength{\unitlength}{266.62923017bp}%
    \ifx\svgscale\undefined%
      \relax%
    \else%
      \setlength{\unitlength}{\unitlength * \real{\svgscale}}%
    \fi%
  \else%
    \setlength{\unitlength}{\svgwidth}%
  \fi%
  \global\let\svgwidth\undefined%
  \global\let\svgscale\undefined%
  \makeatother%
  \begin{picture}(1,0.88681055)%
    \lineheight{1}%
    \setlength\tabcolsep{0pt}%
    \put(0,0){\includegraphics[width=\unitlength,page=1]{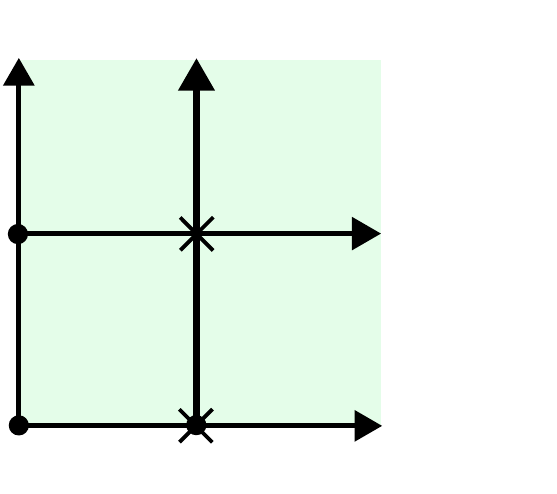}}%
    \put(0.37815607,0.16129614){\color[rgb]{0,0,0}\makebox(0,0)[lt]{\lineheight{1.25}\smash{\begin{tabular}[t]{l}$u_3$\end{tabular}}}}%
    \put(0.35977134,0.52359003){\color[rgb]{0,0,0}\makebox(0,0)[lt]{\lineheight{1.25}\smash{\begin{tabular}[t]{l}$u_1, u_2$\end{tabular}}}}%
    \put(0.70746905,0.09271098){\color[rgb]{0,0,0}\makebox(0,0)[lt]{\lineheight{1.25}\smash{\begin{tabular}[t]{l}$D_1$\end{tabular}}}}%
    \put(-0.00481101,0.81591939){\color[rgb]{0,0,0}\makebox(0,0)[lt]{\lineheight{1.25}\smash{\begin{tabular}[t]{l}$D_2$\end{tabular}}}}%
    \put(0.0029328,0.01775534){\color[rgb]{0,0,0}\makebox(0,0)[lt]{\lineheight{1.25}\smash{\begin{tabular}[t]{l}$X\backslash D$\end{tabular}}}}%
  \end{picture}%
\endgroup%

	\end{minipage}%
	\begin{minipage}{0.65\textwidth}
		\caption{An $n$-marked grid subdivision.}
		\label{fig:gridsub}
	\end{minipage}
\end{figure}

Any tuple of $n$ points on $\S_X$ canonically defines an $n$-marked grid subdivision. The projection to each ray $\mathrm{pr}_i\colon  \S_X \to \S_{(i)}$ sends $n$ points on $\S_X$ to points on $\S_{(i)}$, and this determines a canonical polyhedral subdivision $\cal{P}_i$ of the ray $\S_{(i)}$. We define the associated grid subdivision to be \[\cal{P} = \Sigma_X \cap\prod_{1\leq i \leq n} \cal{P}_i,\] and the $n$ points on $\S_X$ are vertices of $\cal{P}$ and defines an $n$-marking $m$ on $\cal{P}$. We call this assignment $b$.

\begin{lemma} \label{lemma:moduli_trop_points}
    The function $b$ from the set $\S_X^n$ of $n$-tuples of points on $\S_X$ to the set of $n$-marked grid subdivisions of $\S_X$ is a bijection.
\end{lemma}
\begin{proof}
    We associate an $n$-marked grid subdivision $(\cal{P},m)$ of $\S_X$ with $(u_1,u_2,\dots,u_n)\in \Sigma_X^n$ where $u_i = m(i)$. It is immediate that $b$ is bijective.
\end{proof}

\begin{definition}[Marked grid expansions] \label{def:grid_exp}
	Given an $n$-marked grid subdivision $\cal{P} \to \S_X$, the cone over $\cal{P} \times \{1\}$ defines a subdivision of $\S_X \times \Rgeq$. The \emph{$n$-marked grid expansion} $X_{\cal{P}}$ over $\Spec k$ is the special fibre of the corresponding modification of $
	X \times \AA^1$.
\end{definition}

Geometrically, a grid expansion is a (reducible) variety with irreducible components $X$ and chains of $(\PP^1)^k$-bundles over intersections $D_{i_1}\cap D_{i_2} \cap \dots \cap D_{i_k}$. See Figures~\ref{fig:gridsub-fam} and \ref{fig:marked-gridexp} for a family of $n$-marked grid expansions corresponding to the grid subdivision in Figure~\ref{fig:gridsub}.

\begin{figure}[!htb]
	\centering
	\begin{minipage}{.5\textwidth}
		\centering
		\def\svgwidth{\textwidth}
\begingroup%
  \makeatletter%
  \providecommand\color[2][]{%
    \errmessage{(Inkscape) Color is used for the text in Inkscape, but the package 'color.sty' is not loaded}%
    \renewcommand\color[2][]{}%
  }%
  \providecommand\transparent[1]{%
    \errmessage{(Inkscape) Transparency is used (non-zero) for the text in Inkscape, but the package 'transparent.sty' is not loaded}%
    \renewcommand\transparent[1]{}%
  }%
  \providecommand\rotatebox[2]{#2}%
  \newcommand*\fsize{\dimexpr\f@size pt\relax}%
  \newcommand*\lineheight[1]{\fontsize{\fsize}{#1\fsize}\selectfont}%
  \ifx\svgwidth\undefined%
    \setlength{\unitlength}{476.49210292bp}%
    \ifx\svgscale\undefined%
      \relax%
    \else%
      \setlength{\unitlength}{\unitlength * \real{\svgscale}}%
    \fi%
  \else%
    \setlength{\unitlength}{\svgwidth}%
  \fi%
  \global\let\svgwidth\undefined%
  \global\let\svgscale\undefined%
  \makeatother%
  \begin{picture}(1,0.46703947)%
    \lineheight{1}%
    \setlength\tabcolsep{0pt}%
    \put(0,0){\includegraphics[width=\unitlength,page=1]{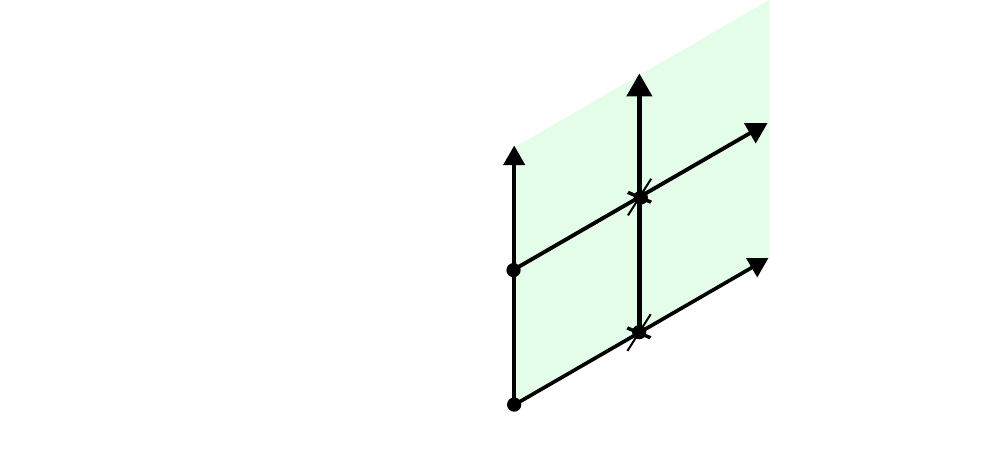}}%
    \put(0.64803745,0.15925533){\color[rgb]{0,0,0}\makebox(0,0)[lt]{\lineheight{1.25}\smash{\begin{tabular}[t]{l}$u_3$\end{tabular}}}}%
    \put(0.57389088,0.29297698){\color[rgb]{0,0,0}\makebox(0,0)[lt]{\lineheight{1.25}\smash{\begin{tabular}[t]{l}$u_1, u_2$\end{tabular}}}}%
    \put(0,0){\includegraphics[width=\unitlength,page=2]{gridsub-fam.pdf}}%
    \put(0.51115019,0.00365403){\color[rgb]{0,0,0}\makebox(0,0)[lt]{\lineheight{0}\smash{\begin{tabular}[t]{l}$1$\end{tabular}}}}%
    \put(0.05754877,0.00396756){\color[rgb]{0,0,0}\makebox(0,0)[lt]{\lineheight{0}\smash{\begin{tabular}[t]{l}$0$\end{tabular}}}}%
    \put(0.92232159,0.04863481){\color[rgb]{0,0,0}\makebox(0,0)[lt]{\lineheight{0}\smash{\begin{tabular}[t]{l}$\Rgeq$\end{tabular}}}}%
    \put(0.17437037,0.32486355){\color[rgb]{0,0,0}\makebox(0,0)[lt]{\lineheight{0}\smash{\begin{tabular}[t]{l}$\Sigma_X$\end{tabular}}}}%
  \end{picture}%
\endgroup%

		\caption{A family of $n$-marked grid subdivisions.}
		\label{fig:gridsub-fam}
	\end{minipage}\hfill
	\begin{minipage}{0.5\textwidth}
		\centering
		\def\svgwidth{.8\textwidth}
		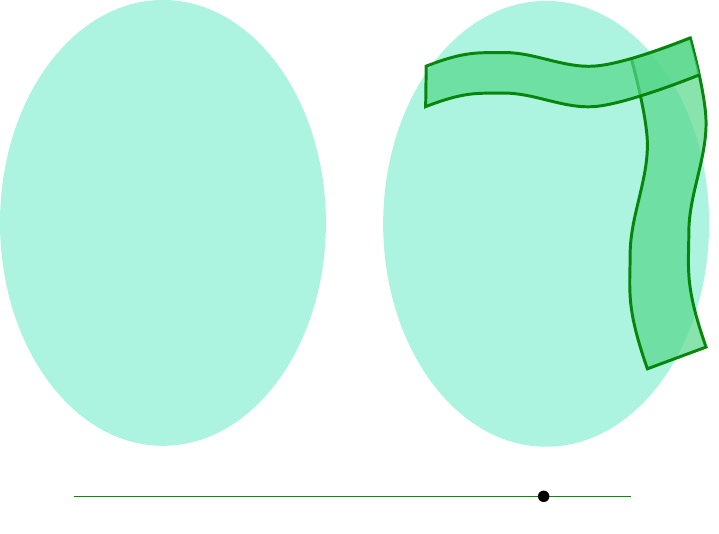
		\caption{A family of $n$-marked grid expansions over $\AA^1$.}
		\label{fig:marked-gridexp}
	\end{minipage}
\end{figure}

\begin{remark} \label{rem:fam-exp}
	There is a more general notion of \emph{expansions} as defined in \cite[Definition 2.2.3]{MR20}. In order to define families of grid expansions, the usual convention is to define families of grid expansions over logarithmic schemes by a logarithmic analogue of Remark~\ref{rem:gridsub-fam}. It will follow that the Artin fan $\gridexp$ of the tropical moduli space of subdivisions $\tropmod$ is the \emph{logarithmic stack} of $n$-marked grid expansions, which turns out to be representable by an algebraic stack $\gridexp$ over schemes.
	
	To date, there is no intrinsic definition of an expansion over a scheme $S$. Instead, we will take the Artin fan $\gridexp$ of the tropical moduli space of subdivisions $\tropmod$, equip it with a universal family $\X$ of grid expansions, and define families of grid expansions over a scheme $S$ as pullbacks of $\X \to \gridexp$ along some morphism $S \to \gridexp$. This definition is made possible by the fact that $\gridexp$ is an algebraic stack, in contrast to the logarithmic stack of expansions in \cite{MR20} that is often not representable by a stack over schemes. 
\end{remark}

\subsubsection{Combinatorial types}

Given a subdivision $\cal{P} \to \Sigma_X$, we consider its \emph{face poset}, denoted $\cal{F}_{\cal{P}}$, and the induced order preserving map $\cal{F}_{\cal{P}} \to \cal{F}_{\Sigma_X}$. The marking function $m$ on $V(\cal{P})$ induces a marking function \[m\colon \oneton \to \cal{F}_{\cal{P}}.\] 
\begin{definition}
	Given an $n$-marked grid subdivision $(\cal{P}, m)$ of $\Sigma_X$, its \emph{combinatorial type} is the pair $(\cal{F}_{\cal{P}} \to \cal{F}_{\Sigma_X}, m)$.
\end{definition}
\begin{remark}\label{rem:combtype}
	Consider an $n$-tuple $(u_1, u_2, \ldots, u_n)$ of points in $\S_X$. Let $u_i^{(1)}, u_i^{(2)}, \ldots, u_i^{(r)}$ be the coordinates of $u_i$, and let $T_j$ denote the total pre-order of $0, u_1^{(j)}, u_2^{(j)}, \ldots, u_n^{(j)}$. Then the combinatorial type of $(u_1, u_2, \ldots, u_n)$ is equivalent to the data $(T_1, T_2, \ldots, T_r)$.
\end{remark}

\subsection{Tropical moduli problem and semistable reduction} \label{section:tropmoduli_grid} For compactness of notation, replace $\Sigma_X$ by $\Sigma$. As suggested by the tropicalisation map, we consider the moduli problem of $n$-marked grid subdivisions of $\S$, and construct the tropical moduli space $\tropmod$ of $n$-marked grid subdivisions.

We have the following diagram of cone complexes, where $s_1, \ldots, s_n$ are the sections:
\begin{equation}\label{cd:logprodnaive}
	\begin{tikzcd}
		\Sigma \times \Sigma^{n} \arrow[dd, "p"']     \arrow[r]        & \S          \\
		\\
		\Sigma^n \arrow[uu, "{s_1, \ldots, s_n}"', bend right] 
	\end{tikzcd}
\end{equation}
The map $p$ is induced by the projection onto the second factor, and $s_i$ is the $i$th section induced by the map $\mathrm{pr}_i \times \mathrm{id}$, where $\mathrm{pr}_i\colon  \Sigma^n \to \Sigma$ is the projection onto the $i$th factor.

\begin{construction}\label{prop:tropmoduli_grid} 
    We construct a diagram of cone complexes
\begin{equation}\label{cd:subdivision}
    \begin{tikzcd}
        \tropfam \arrow[dd, "p"'] \arrow[r] & \S                       \\
        \\
        \tropmod \arrow[uu, "{s_1, \ldots, s_n}"', bend right]
    \end{tikzcd}
\end{equation} 
where each term in the diagram is obtained by a subdivision of the corresponding term in Diagram (\ref{cd:logprodnaive}). This diagram satisfies:
\begin{enumerate}
    \item \textbf{\emph{Transversality.}} The image of each $s_i$ is a union of cones in $\tropfam$. 
    \item \textbf{\emph{Combinatorial flatness.}} Every cone of $\tropfam$ surjects under $p$ onto a cone of $\tropmod$.
    \item \textbf{\emph{Reducedness.}} The image under $p$ of the lattice of every cone in $\tropfam$ is equal to the lattice of the image cone in $\tropmod$.
    
\end{enumerate}
\end{construction}
\begin{proof}[Steps]
    We carry out weak semistable reduction following \cite[Sections 4 \& 5]{semistable}. 
    \begin{enumerate}[wide, labelindent=0pt]
        \item \textbf{Subdivision of $\S \times \Sn$ for transversality.} By Lemma~\ref{lemma:moduli_trop_points}, we adopt the perspective that $\Sn$ is the underlying set of $n$-tuples of points on $\S$, or equivalently, the set of $n$-marked grid subdivisions. Similarly, $\S\times \Sn$ is the set of $n+1$-tuples of points on $\S$, or equivalently, the set of $n+1$-marked grid subdivisions. 
		The image of $s_i$ is the set 
        \begin{align*}
            s_i(\Sn)  &=  \{(x, u_1, u_2, \ldots, u_n)\in \S \times \Sn \; | \; x = u_i  \}.
        \end{align*} 
        The goal is to subdivide $\S\times \Sn$ such that $s_i(\Sn)$ is a union of cones in the subdivision.

		By Remark~\ref{rem:combtype}, each combinatorial type $\tau^+ = (\cal{F}^+ \to \F_{\Sigma}, m^+)$ is equivalent to a collection of total pre-orders $T_j^+$ of the $j$th-coordinates \[0, x^{(j)}, u_1^{(j)}, u_2^{(j)}, \ldots, u_n^{(j)}.\] These total pre-orders $(T_1^+, T_2^+, \ldots, T_r^+)$ give rise to an intersection of half-spaces, and due to the linear independence of the coordinates on $\Sigma$, this intersection of half-spaces is a rational polyhedral cone, which we denote by $\sigma_{\tau^+}$. The lattice of $\sigma_{\tau^+}$ is the groupification of the intersection of the lattice of $\S \times \Sn$, namely $(\Z^r)^{n+1}$, with the image of $\sigma_{\tau^+}$. The collection of cones $\sigma_{\tau^+}$ for all combinatorial types $\tau^+$ of elements in $\S\times \Sn$ form a cone complex $\tropfam$, which is a subdivision of $\S\times \Sn$ as $\tropfam$ has the same underlying support and lattice as $\S\times \Sn$. 

        We note that $s_i(\Sn)$ is the union of cones $\sigma_{\tau^+}$ with combinatorial types $\tau^+ = (\cal{F}^+ \to \F_{\Sigma}, m^+)$ such that the equalities $x^{(j)} = u_i^{(j)}$ holds for all $j$.

		\smallskip
        \item \textbf{Subdivision of $\Sn$ to make each $s_i$ a morphism of cone complexes.} 

        As described in the subdivision of $\S \times \Sn$, we subdivide $\Sn$ by considering the cones $\mu_{\tau}$ indexed by combinatorial types $\tau = (\cal{F}\to \F_{\Sigma},m)$, determined by total pre-orders of coordinates $(T_1, T_2, \ldots, T_r)$. The lattice of $\mu_{\tau}$ is the groupification of the intersection of the lattice of $\Sn$, namely $(\Z^r)^{n}$, with the image of $\mu_{\tau}$.
        
        The collection of cones $\mu_{\tau}$ for all combinatorial types $\tau$ of elements in $\Sn$ form a cone complex $\tropmod$, which is a subdivision of $\Sn$, as $\tropmod$ has the same underlying support and lattice as $\Sn$. 
        
        To check $s_i$ is a morphism of cone complexes, for any given combinatorial type $\tau = (\cal{F}\to \F_{\Sigma},m)$, one needs to check $s_i$ sends the cone $\mu_{\tau}\subset \tropmod$ to some cone in $\tropfam$. 
        Fixing the combinatorial type $\tau= (\cal{F}\to \F_{\Sigma},m)$ of an $n$-marked grid subdivision $\cal{P}$, with associated total pre-orders $(T_1, T_2, \ldots, T_r)$, let $\tau_i$ be the combinatorial type of an $n+1$-marked grid subdivision given by $(\cal{F}\to \F_{\Sigma},m_i)$, where the marking function is \[m_i \colon \{0, 1, \ldots, n\} \to V(\cal{P}), \; \; m_i(0) = m(i), \; m_i(j) = m(j) \text{ for } j \neq i.
        \]This combinatorial type is precisely determined by the total pre-orders $(T_1^+, T_2^+, \ldots, T_r^+)$ where for each $j$, $T_j^+$ is the total pre-order $T_j$ with the additional equality $x^{(j)} = u_i^{(j)}$.  

		So, $s_i(\mu_{\tau})$ is equal to the cone $\sigma_{\tau_i}$, and $s_i$ is a morphism of cone complexes.
        
        \smallskip
        \item \textbf{Combinatorial reducedness.} 
        Recall that the underlying lattice of $\tropmod$ is $(\Z^r)^n \subset (\R^r)^n$. Namely, a point $(u_1, u_2, \ldots, u_n)$ in $\tropmod$ is integral if and only if each $u_i\in \Rgeq^r$ has integral coordinates. We take a similar lattice $(\Z^r)^{n+1}$ on the cone complex $\tropfam$. 
        
        We first show that $p$ is combinatorially reduced. For an arbitrary cone $\sigma$ in $\tropfam$ and its image cone $\mu = p(\sigma)$ in $\tropmod$, above each integral point $(u_1, u_2, \ldots, u_n)$ in $\mu$, the preimage \[p^{-1}(u_1, u_2, \ldots, u_n) \cap \sigma\] is a polyhedron $P$ in the polyhedral complex $p^{-1}(u_1, u_2, \ldots, u_n)$. 
        
         Each vertex $v$ of $P$ is the intersection of $r$ coordinate hyperplanes, where each coordinate hyperplane is of the form $x_j=0$ or of the form $x_j = u_i^{(j)}$, and since $(u_1, u_2, \ldots, u_n)$ is integral, it follows that $v$ is also integral. This proves that the image of the lattice of $\sigma$ is the lattice of $\mu$.
        
        \item \textbf{Combinatorial flatness.}
        It remains to show that every cone of $\tropfam$ surjects under $p$ onto a cone of $\tropmod$. 
        
        Given a combinatorial type $\tau^+ = (\F^+ \to \F_{\S}, m^+)$ in $\tropfam$ described by total pre-orders $(T_1^+, T_2^+, \ldots, T_r^+)$, and its associated cone $\sigma_{\tau^+}$, an element of $\sigma_{\tau^+}$ is of the form $(\cal{P}^+, m^+)$ where $\F_{\cal{P}^+} = \F^+$. Identifying $(\cal{P}^+, m^+)$ with $(x,u_1,u_2,\ldots,u_n)$, where $x = m^+(0)$ and $u_j = m^+(j)$, we recall \[p(x,u_1,u_2,\ldots,u_n) = (u_1, u_2, \ldots, u_n) \] 
        where $(u_1, u_2, \ldots, u_n)$ is identified with $(\cal{P}, m)$. Its combinatorial type $\tau := (\F \to \F_{\S}, m)$ is then given by the total pre-orders $(T_1, T_2, \ldots, T_r)$, where each $T_j$ is obtained by forgetting the coordinate $x^{(j)}$ in $T_j^+$. It is straightforward to verify that $p(\sigma_{\tau^+}) = \mu_{\tau}$.
    \end{enumerate}
\end{proof}

For simplicity of notation, we will write $\tau$ to denote a combinatorial type and its cone.

\begin{observation} \label{obs:cone}
	Fix a combinatorial type $\tau$, and let \[
	\cal{P} = \Sigma_X \cap \prod_{1\leq i \leq r} \cal{P}_i,
	\] be any $n$-marked grid subdivision of type $\tau$. Each $\cal{P}_i$ corresponds to the ray $\Sigma_{(i)}$ subdivided by $n$ not necessarily distinct vertices into bounded line segments and an unbounded ray. The relative positions of the vertices correspond to the total pre-order $T_i$. The (non-zero) lengths $l_{i,j}$ of all the bounded line segments of $\cal{P}_i$, taken for all $i$, are linearly independent and exactly determine the subdivision $\cal{P}$ of the given type $\tau$. The cone corresponding to $\tau$ is the cone $\cone(l_{i,j})_{i,j}$ freely generated by the length parameters, and setting a length parameter to zero corresponds to generalisation of combinatorial types. Since the generators form a subset of a $\Z$-basis of the lattice $(\Z^r)^n$, we conclude that the cone $\tau$ is smooth \cite{cls}.
\end{observation}

\begin{figure}[!htb]
	\centering
	\def\svgwidth{.3\textwidth}
\begingroup%
  \makeatletter%
  \providecommand\color[2][]{%
    \errmessage{(Inkscape) Color is used for the text in Inkscape, but the package 'color.sty' is not loaded}%
    \renewcommand\color[2][]{}%
  }%
  \providecommand\transparent[1]{%
    \errmessage{(Inkscape) Transparency is used (non-zero) for the text in Inkscape, but the package 'transparent.sty' is not loaded}%
    \renewcommand\transparent[1]{}%
  }%
  \providecommand\rotatebox[2]{#2}%
  \newcommand*\fsize{\dimexpr\f@size pt\relax}%
  \newcommand*\lineheight[1]{\fontsize{\fsize}{#1\fsize}\selectfont}%
  \ifx\svgwidth\undefined%
    \setlength{\unitlength}{318.14412275bp}%
    \ifx\svgscale\undefined%
      \relax%
    \else%
      \setlength{\unitlength}{\unitlength * \real{\svgscale}}%
    \fi%
  \else%
    \setlength{\unitlength}{\svgwidth}%
  \fi%
  \global\let\svgwidth\undefined%
  \global\let\svgscale\undefined%
  \makeatother%
  \begin{picture}(1,0.75747239)%
    \lineheight{1}%
    \setlength\tabcolsep{0pt}%
    \put(0,0){\includegraphics[width=\unitlength,page=1]{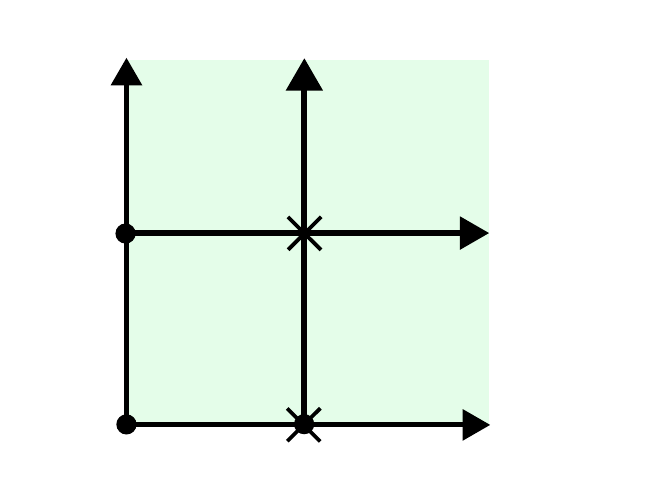}}%
    \put(0.47884691,0.14943554){\color[rgb]{0,0,0}\makebox(0,0)[lt]{\lineheight{1.25}\smash{\begin{tabular}[t]{l}$u_3$\end{tabular}}}}%
    \put(0.46343909,0.45306567){\color[rgb]{0,0,0}\makebox(0,0)[lt]{\lineheight{1.25}\smash{\begin{tabular}[t]{l}$u_1, u_2$\end{tabular}}}}%
    \put(0.7548365,0.09195591){\color[rgb]{0,0,0}\makebox(0,0)[lt]{\lineheight{1.25}\smash{\begin{tabular}[t]{l}$D_1$\end{tabular}}}}%
    \put(0.15789105,0.69806015){\color[rgb]{0,0,0}\makebox(0,0)[lt]{\lineheight{1.25}\smash{\begin{tabular}[t]{l}$D_2$\end{tabular}}}}%
    \put(0.00017563,0.02215535){\color[rgb]{0,0,0}\makebox(0,0)[lt]{\lineheight{1.25}\smash{\begin{tabular}[t]{l}$X\backslash D$\end{tabular}}}}%
    \put(0,0){\includegraphics[width=\unitlength,page=2]{gridsub-cone.pdf}}%
    \put(0.02924219,0.26881323){\color[rgb]{0,0,0}\makebox(0,0)[lt]{\lineheight{0}\smash{\begin{tabular}[t]{l}$l_{2,1}$\end{tabular}}}}%
    \put(0.26173574,0.00749861){\color[rgb]{0,0,0}\makebox(0,0)[lt]{\lineheight{0}\smash{\begin{tabular}[t]{l}$l_{1,1}$\end{tabular}}}}%
  \end{picture}%
\endgroup%

	\caption{Combinatorial type with cone $\Rgeq^2$ spanned by two length parameters.}
	\label{fig:gridsub-cone}
\end{figure}

\begin{proposition} \label{prop:simplicial}
	The cone complexes $\tropmod$ and $\tropfam$ are smooth.
\end{proposition}
\begin{proof}
	We proved the statement for $\tropmod$ in Observation~\ref{obs:cone}, and an analogous argument proves the statement for $\tropfam$.
\end{proof}

\begin{remark}\label{rem:gridsub-fam}
	Let $\mathfrak{S}$ be an arbitrary rational polyhedral cone complex. One can define a \emph{family of $n$-marked grid subdivisions} of $\Sigma$ over $\mathfrak{S}$ as a family with $n$ sections 
	\begin{equation*}
		\begin{tikzcd}
			\widetilde{\Sigma \times \mathfrak{S}} \arrow[r] & \mathfrak{S} \arrow[l, bend left, "{s_1, \ldots, s_n}"],
		\end{tikzcd}
	\end{equation*}
	such that $\widetilde{\Sigma \times \mathfrak{S}}$ is a subdivision of $\Sigma \times \mathfrak{S}$, the family satisfies combinatorial flatness and reducedness (see Proposition~\ref{prop:tropmoduli_grid}), and the preimage over each point on $\mathfrak{S}$ is an $n$-marked grid subdivision whose marking is induced by the $n$ sections. Then by the canonicity of grid subdivisions, one observe that any such family must be pulled back from the tropical family $\tropfam \to \tropmod$ along a morphism $\mathfrak{S} \to \tropmod$. This justifies the implicit claim that $\tropmod$ is the \emph{tropical moduli space} of $n$-marked grid subdivisions.
\end{remark}

\subsection{Stable $n$-pointed grid expansions} \label{section:logProd_moduli}
We formulate the moduli problem for the compactification $\logProd$. Let $\gridexp$ and $\gridexpfam$ denote the Artin fan corresponding to $\tropmod$ and $\tropfam$. 

\begin{proposition} \label{prop:grid-exp}
	The stack $\gridexp$ is equipped with a flat and representable family of grid expansions $\eta\colon \X \to \gridexp$, which is logarithmically smooth and has reduced fibres. 
\end{proposition}
\begin{proof}
	Proposition~\ref{prop:tropmoduli_grid} gives the flat family $\gridexpfam \to \gridexp$ with reduced fibres. The family $\eta\colon \X \to \gridexp$ is obtained by the following pullback. Note that along a strict morphism, the pullback in algebraic stacks coincides with the pullback in fine and saturated logarithmic stacks.

	\begin{equation*}\label{cd:artin_subdivision}
		\begin{tikzcd}
			\mathcal{X} \arrow[d, "\text{strict, flat, smooth}"'] \arrow[r] \pullback{rd} & X \times \gridexp \arrow[r] \arrow[d] \pullback{rd} & X \arrow[d, "\text{strict, flat, smooth}"] \\
			\gridexpfam \arrow[d, "\text{flat}"'] \arrow[r] & \A_X \times \gridexp \arrow[r] & \mathcal{A}_X \\
			\gridexp \arrow[u, "{e_1, e_2,\ldots,e_n}"', bend right] &                             
		\end{tikzcd}
	\end{equation*}
	
	Since $X \to \A_X$ is flat and smooth, so is the base change $\X \to \gridexpfam$. Logarithmic smoothness of $\X \to \gridexp$ is equivalent to smoothness of $\X \to \gridexpfam$, which holds. By Proposition~\ref{prop:tropmoduli_grid} and the flatness criterion for morphisms of toric varieties \cite[Theorem II.4.2]{tsuji}, $\gridexpfam \to \gridexp$ is flat. Hence $\X \to \gridexp$ is flat by composition. We note that the geometric fibres of $\X \to \gridexp$ are grid expansions.
\end{proof}

As described in Remark~\ref{rem:fam-exp}, we let $\gridexp$ be the stack of $n$-marked grid expansions of $(X,D)$ with universal family $\eta \colon \X \to \gridexp$, and we define families of $n$-marked grid expansions below.

\begin{definition}
	A \emph{family of $n$-marked grid expansions} over a scheme $S$ is a flat family $\X_S \to S$ with reduced fibres, such that there exists a morphism $f \colon S \to \gridexp$ and $\X_S$ is the pullback of $\X \to \gridexp$ to $S$ along $f$.
\end{definition}


\begin{definition}\label{cor:gridexp_over_point}
	An \emph{$n$-pointed grid expansion} is $(X_{\cal{P}}, p_1, p_2, \dots p_n)$, where $X_{\cal{P}}$ is an $n$-marked grid expansion induced by some grid subdivision $(\cal{P}, m)$ (see Definition~\ref{def:grid_exp}), and $p_1, p_2, \dots, p_n$ are points on the smooth locus of $X_{\cal{P}}$. It is \emph{stable} when $p_i$ lies on the irreducible component with marking $m(i)$.
\end{definition}

\begin{definition} \label{lem:gridexppullback}
	An \emph{$n$-pointed grid expansion} over a scheme $S$ is the data \[(\cal{X}_S\to S, s_1, s_2, \dots s_n),\] where $\X_S \to S$ is a family of $n$-marked grid expansions pulled back along some morphism $f \colon S\to \gridexp$, and $s_1,s_2, \dots, s_n$ are sections. It is \emph{stable} when the sections $s_i$ satisfy that $e_i \circ f$ and $\hat{\eta}\circ \hat{f} \circ s_i$ (2-) commute. We have the following (2-)commutative diagram:
	\begin{equation}
		\begin{tikzcd}
			\X_S \arrow[d, "f^*\eta"] \arrow[r, "\hat{f}"] \arrow[dr, phantom, "\square" , color=black] & \X \arrow[d, "\eta"] \arrow[r, "\hat{\eta}"] & \gridexpfam \arrow[dl]\\
			S \arrow[u, "{s_i}", bend left] \arrow[r, "f"] & \gridexp \arrow[ur, "{e_i}"', bend right]. &
		\end{tikzcd}
	\end{equation}
	
\end{definition}


In Observation~\ref{obs:stability}, we explain why stability guarantees the absence of non-trivial automorphisms. Before proceeding to define morphisms between $n$-pointed grid expansions, we need the following lemma from \cite[Lemma 2.9]{rubbertori}:

\begin{lemma} \label{lem:2-iso}
	A 2-isomorphism $\lambda$ between two morphisms $f_1\colon  S \to \gridexp$ and $f_2\colon  S \to \gridexp$ induces a unique isomorphism $\aut\colon  f_1^*\X \to f_2^*\X$. The maps fit into the 2-commutative diagram:
	\begin{equation*}
		\begin{tikzcd}
			& f_2^*\mathcal{X} \arrow[rd] \arrow[d, "\eta_2"]& \\                               
			& S \arrow[rd, "f_2", ""{name=U,inner sep=1pt,below}]   & \mathcal{X} \arrow[d] \\
			f_1^*\mathcal{X} \arrow[rru] \arrow[d, "\eta_1"] \arrow[ruu, " \lambda^{\dagger}"]  &           & \mathcal{E}_n(X|D)  \\
			S \arrow[rru, "f_1"', ""{name=D,inner sep=1pt}] \arrow[ruu, "\mathrm{id}"']  \arrow[Rightarrow, from=D, to=U, "\lambda"]&                                                                          &                      
		\end{tikzcd}
	\end{equation*}
\end{lemma}

\begin{remark} \label{rem:rubber_action}
	Both $f_1$ and $f_2$ necessarily factor through some $\B T_{\tau}$, where $T_{\tau}$ is the torus associated to the cone $\tau$ of $\tropmod$, and $\B T_{\tau}$ is the classifying stack. The isomorphism $\aut\colon  f_1^*\X \to f_2^*\X$ is known as the \emph{rubber action} of $T_{\tau}$ on the $n$-marked grid expansion $f_1^*\X \cong f_2^*\X$ associated to the combinatorial type $\tau$. We will describe the rubber action in more detail in the following subsection.
\end{remark}

\begin{definition} \label{lem:pullback_mor}
	A \emph{morphism} between stable $n$-pointed grid expansions \[(f_1\colon S_1\to \gridexp, s_{1,1}, s_{1,2}, \dots, s_{1,n}) \to (f_2\colon  S_2 \to \gridexp, s_{2,1}, s_{2,2}, \dots, s_{2,n})\] is the data $(\psi,\lambda)$:
	\begin{enumerate}
		\item A morphism $\psi\colon S_1 \to S_2$ fitting into the 2-commutative diagrams
		\begin{equation*} 
			\begin{tikzcd}
				& f_2^*\mathcal{X} \arrow[rd] \arrow[d]& \\                               
				& S_2 \arrow[rd, "f_2", ""{name=U,inner sep=1pt,below}] \arrow[u, "{s_{2,1}, s_{2,2}, \dots, s_{2,n}}"', bend right]    & \mathcal{X} \arrow[d] \\
				f_1^*\mathcal{X} \arrow[rru] \arrow[d] \arrow[ruu, "\Psi \circ \lambda^{\dagger}"]  &           & \mathcal{E}_n(X|D)  \\
				S_1 \arrow[rru, "f_1"', ""{name=D,inner sep=1pt}] \arrow[ruu, "\psi"'] \arrow[u, "{s_{1,1}, s_{1,2}, \dots, s_{1,n}}", bend left] \arrow[Rightarrow, from=D, to=U, "\lambda"]&                                                                          &                     
			\end{tikzcd}
		\end{equation*}
		\item 2-isomorphisms $\lambda\colon  f_1 \Rightarrow f_2 \circ \psi$ and equalities \[(\Psi \circ \lambda^\dagger) \circ s_{1,i} = s_{2,i} \circ \psi,\] where $\Psi\colon \psi^*f_2^*\X \to f_2^*\X$ is the canonical morphism.
		
	\end{enumerate}
	
	A morphism $(\psi, \lambda)$ is an \emph{isomorphism} if $\psi$ is an isomorphism.
\end{definition}

\subsection{Construction of moduli space $\logProd$} 
We construct the compactification $\logProd$ via the tropical moduli space $\tropmod$. We prove that $\logProd$ represents the moduli problem of stable $n$-pointed grid expansions and has desired properties.

The subdivision $\tropmod \to \S^n$ induces a modification of Artin fans $\psi\colon \gridexp \to \A_X^n$, hence a logarithmic modification of $X^n$ (see Section~\ref{subsection:logmod}):

\begin{equation}\label{cd:logprod}
    \begin{tikzcd}
	\logProd \arrow[d] \arrow[r, "\phi"] \arrow[dr, phantom, "\square", color=black] & {\gridexp}\arrow[d, "\psi"] \\
	{X^n} \arrow[r, "\text{flat}"] & {\mathcal{A}_X^n}.
    \end{tikzcd}
\end{equation}

We note that $\logProd$ has Artin fan $\gridexp$, and is representable by a logarithmic scheme as it is a modification of $X^n$.


These pullbacks fit together in a bigger 2-commutative diagram below, and we pullback $\eta\colon \X \to \gridexp$ along $\phi\colon \logProd \to \gridexp$. As $\eta$ is representable, we obtain a flat and logarithmically smooth universal family $\pi\colon  \fam \to \logProd$. In addition, for each $i$, the morphism \[\logProd \xrightarrow{e_i \circ \phi} \gridexpfam \to \A_X\] 2-commutes with \[\logProd \to X^n \xrightarrow{\mathrm{pr_i}} X \to \A_X,\] and this follows from the construction of $\gridexpfam$ and $\gridexp$ together with the sections $e_i$. Therefore, we obtain morphisms $S_i\colon  \logProd \to \X$. The fact that $e_i$ are sections of $\mathbf{p}$ guarantees that the morphisms $S_i$ induce universal sections $\mathbf{s}_i\colon  \logProd \to \fam$ such that $e_i \circ \phi$ and $\hat{\eta}\circ \hat{\phi}\circ \mathbf{s}_i$ 2-commute.

\begin{equation}\label{cd:family}
	\begin{tikzcd}
		\fam \arrow[d, "\pi", "\text{flat}"'] \arrow[r,"\hat{\phi}"] \arrow[dr, phantom, "\square" , color=black] & \X \arrow[rd,"\hat{\eta}"] \arrow[r] \arrow[d, "\eta", "\text{flat}"'] \arrow[drr, phantom, "\square" , color=black]                               & X \arrow[rd, "\text{strict}"]            &               \\
		\logProd \arrow[u, "\mathbf{s}_i", bend left = 60] \arrow[r, "\phi", "\text{strict}"'] \arrow[d]\arrow[dr, phantom, "\square", color=black]        & \mathcal{E}_n(X|D) \arrow[r, "e_i"', bend right=15]  \arrow[d , "\psi"] & \mathcal{E}_n^+(X|D) \arrow[l,"\mathbf{p}"'] \arrow[r] & \mathcal{A}_X \\
		X^n \arrow[r, "\text{strict}"]                      & \mathcal{A}_X^n \arrow[rru, "{\mathrm{pr}_i}"', bend right=10]                               &                                          &              
	\end{tikzcd}
\end{equation}

\begin{theorem} \label{thm:logprod}
    Given a simple normal crossings pair $(X,D)$, the scheme $\logProd$ is a simple normal crossings compactification of $(X\backslash D)^n$. It represents the moduli stack of stable $n$-pointed grid expansions.
\end{theorem}
\begin{proof}
    The first statement of the theorem is a direct consequence of the construction of $\logProd$ as a logarithmic modification of $X^n$. As $\tropmod$ is a complete subdivision of $\S_X^n$, therefore $\psi\colon  \gridexp \to \A_X^n$ is proper and so is the pullback $\logProd\to X^n$. The properness of $\logProd$ over $\Spec k$ follows. The scheme $\logProd$ is smooth over its Artin fan $\gridexp$ as smoothness is preserved under pullbacks, and $\gridexp$ is smooth as the cone complex $\tropmod$ is smooth by Proposition~\ref{prop:simplicial}. Therefore, the scheme $\logProd$ is smooth with a simple normal crossings boundary divisor.
     
    The second statement follows directly from construction. A stable $n$-pointed grid expansion $(f\colon  S \to \gridexp, s_1, s_2, \dots s_n\colon S \to f^*\X)$ induces a morphism $s\colon  S \to X^n$ given by the compositions $S \xrightarrow{s_i} f^* \X \to X$, and one can check  (see Diagram~\ref{cd:family}) that the maps $f$ and $s$ commute with the bottom right corner of Diagram~\ref{cd:logprod}. Therefore, the data $(f\colon  S \to \gridexp, s_1, s_2, \dots s_n\colon S \to f^*\X)$ induces a morphism $F\colon  S\to \logProd$. 
	
	It remains to check that two stable $n$-pointed grid expansions induce the same morphism $F\colon  S \to \logProd$ if and only if they are isomorphic, but this follows directly from the universal property of the 2-fibre product in Diagram~\ref{cd:logprod}.
\end{proof}

\begin{remark}
	By Proposition~\ref{prop:simplicial}, the universal family $\fam$ is also smooth. 
\end{remark}

\begin{remark}
	There is an equivalent formulation of the above result (see \cite[Section 1.5]{abram_fantechi}) in terms of the stack $\gridexp$ of grid expansions of $X$, the universal family $\eta\colon \X \to \gridexp$, and the stack $(\X|\cal{D})^{[n]}$ of stable $n$-pointed grid expansions which is defined to be the open locus of the fibre product $\X_{\eta}^{(n)}$ with no automorphisms. The above lemmas then prove that $\logProd$ represents $(\X|\cal{D})^{[n]}$.
\end{remark}
\begin{remark}
	We observe that the tropical family of $n$ points $\tropfam$ is the same as the tropical moduli space of $n+1$ points. Consequently, we have $\fam = \logProdexample{X}{D}{n+1}$.
\end{remark}

\subsubsection{Recursive description of diagonals}

Consider \emph{diagonals} \[\Delta_I = \{(x_1, x_2, \dots, x_n) \in X^n \; | \; x_i = x_j \; \forall \; i, j \in I\} \subset X^n,\] and let $\delta_I$ denote the strict transform of $\Delta_I$ under the modification $\logProd \to X^n$. We note that $\Delta_I \cong X^{n-|I|+1}$; to remember the indexing set $I \subset \oneton$ where points are identified, we introduce the set notation $X^{[n]/I} \cong X^{n-|I|+1}$. Let $\Delta_I$ be endowed the logarithmic structure on $X^{[n]/I}$, so it has Artin fan $\A_X^{[n]/I}$ and tropicalisation $\Sigma_X^{[n]/I}$. The embedding $\Delta_I \hookrightarrow X^n$ is thus a morphism of logarithmic schemes.

\begin{proposition} \label{prop:diagonal}
	The strict transform $\delta_I$ of the diagonal $\Delta_I$ is the locus of configurations in $\logProd$ where points indexed by $I$ coincide. 
\end{proposition}
In other words, the strict transform $\delta_I$ is again a diagonal in $\logProd$, and there is an identification $\delta_I \cong (X|D)^{[n]/I}$. This result will be useful in proving that $\logFM$ is a simple normal crossings compactification, see Theorem~\ref{thm:logFM_fibres}. 
\begin{proof}
	Proposition~\ref{prop:tropmoduli_grid} gives a subdivision of cone complexes $\Pi_{[n]/I}(\Sigma_X) \to \Sigma_X^{[n]/I}$. One can verify that there is a Cartesian diagram of cone complexes
	\begin{equation*}
		\begin{tikzcd}
			\Pi_{[n]/I}(\Sigma_X) \arrow[r, hook] \arrow[d] \pullback{rd} & \tropmod \arrow[d]\\
			\Sigma_X^{[n]/I} \arrow[r, hook] & \Sigma_X^n,
		\end{tikzcd}
	\end{equation*} 
	and a corresponding Cartesian diagram of Artin fans 
	\begin{equation*}
		\begin{tikzcd}
			\cal{E}_{[n]/I}(X|D) \arrow[r] \arrow[d] \fspullback{rd} & \gridexp \arrow[d]\\
			\A_X^{[n]/I} \arrow[r] & \A_X^n.
		\end{tikzcd}
	\end{equation*} 
	Since both $\logProd$ and $(X|D)^{[n]/I}$ are modifications of $X^n$ and $X^{[n]/I}$ obtained by pulling back the modifications $\gridexp \to \A_X^n$ and $\cal{E}_{[n]/I}(X|D) \to \A_X^{[n]/I}$ respectively, there is a Cartesian diagram in the category of fine and saturated logarithmic schemes
	\begin{equation*}
		\begin{tikzcd}
			(X|D)^{[n]/I} \arrow[r] \arrow[d] \fspullback{rd} & \logProd \arrow[d]\\
			X^{[n]/I} \arrow[r] & X^n.
		\end{tikzcd}
	\end{equation*} 
	Since $\Delta_I \cong X^{[n]/I}$, and $\logProd \to X^n$ is a log blow-up by Proposition~\ref{prop:blow-up}, by Lemma~\ref{lem:fs-pullback} we conclude that $(X|D)^{[n]/I}$ must be the strict transform of $\Delta_I$, hence equal to $\delta_I$.
\end{proof}

\subsection{Description of boundary strata and the rubber action} \label{subsection:b_strata}
We describe the boundary strata in more detail, and provide a description for the objects represented by the points in the boundary strata. Many results in this section follow almost verbatim from Sections 1 and 2 of \cite{rubbertori}.

The cone of a combinatorial type $\tau$ defines a substack $\B T_{\tau} \hookrightarrow \gridexp$. The pullback 
\begin{equation}
	\begin{tikzcd}
		V_{\tau}\arrow[d] \arrow[r, hook]\arrow[dr, phantom, "\square"] & \logProd \arrow[d,"\phi"]\\
		\B T_{\tau} \arrow[r, hook] & \gridexp
	\end{tikzcd}
\end{equation}
defines a stratum $V_{\tau}$ of $\logProd$. 

\begin{theorem}\label{thm:logprod_strat}
    The boundary of $\logProd$ admits a stratification by $V_{\tau}$. 
\end{theorem}
\begin{proof}
	Recall that $\logProd$ has $\gridexp$ as its Artin fan; in other words, the logarithmic structure on $\logProd$ is given by the cone complex $\tropmod$, whose cones are precisely given by the combinatorial types. The logarithmic structure determines a stratification by the strata $V_{\tau}$.
\end{proof} 
\begin{figure}
    \centering
	\tikzset{every picture/.style={line width=0.75pt}} 
	
	\begin{tikzpicture}[x=0.75pt,y=0.75pt,yscale=-1,xscale=1]
		
		\draw  (56,664.09) -- (156,664.09)(66,574.09) -- (66,674.09) (149,659.09) -- (156,664.09) -- (149,669.09) (61,581.09) -- (66,574.09) -- (71,581.09)  ;
		\draw   (62.89,661.37) -- (68.35,666.97)(68.42,661.44) -- (62.82,666.9) ;
		\draw   (102.89,661.37) -- (108.35,666.97)(108.42,661.44) -- (102.82,666.9) ;
		\draw    (106.45,647) -- (106.45,663.82) ;
		\draw    (106.45,577) -- (106.45,663.82) ;
		\draw  [draw opacity=0][fill={rgb, 255:red, 80; green, 227; blue, 194 }  ,fill opacity=0.31 ] (66,573.35) -- (155.58,573.35) -- (155.58,664.09) -- (66,664.09) -- cycle ;
		
		\draw (98,667) node [anchor=north west][inner sep=0.75pt]   [align=left] {$\displaystyle I$};
		\draw (49,666) node [anchor=north west][inner sep=0.75pt]   [align=left] {$\displaystyle I^{c}$};
		\draw (166,654.09) node [anchor=north west][inner sep=0.75pt]   [align=left] {$\displaystyle D_{k}$};
		\draw (-8,549.09) node [anchor=north west][inner sep=0.75pt]   [align=left] {Other divisors};

	\end{tikzpicture}

    \caption{A combinatorial type $\tau$ whose stratum closure $\overline{V_{\tau}}$ is a divisor.}
    \label{fig:comb_type_divisor}
\end{figure}


We give an explicit calculation of the codimension of boundary strata.
\begin{proposition}[Codimension of boundary strata]
    Given a combinatorial type $\tau=(\F \to \F_{\S},m)$ of an $n$-pointed grid expansion, consider any $n$-pointed grid expansion $(\cal{P},m)$ with type $\tau$. Writing \[\cal{P} = \prod_{1\leq i \leq n} \cal{P}_i,\] we let $w_i$ be the number of vertices on $\cal{P}_i$ apart from the origin. Then the codimension of the stratum $V_{\tau}$ is \[\sum_{1\leq i\leq n}w_i.\]
\end{proposition}
\begin{proof}
    This follows from counting the dimension of the cone $\tau$ via the description in terms of the total pre-orders of coordinates.
\end{proof}

Here are some examples of stable $n$-pointed grid expansions with their combinatorial types, where $(X,D) = (\PP^2, D_1 + D_2)$, $D_1$ and $D_2$ being two coordinate lines on $\PP^2$:
\begin{figure}
\centering
\tikzset{every picture/.style={line width=0.75pt}} 

\begin{tikzpicture}[x=0.75pt,y=0.75pt,yscale=-1,xscale=1]
	
	\draw   (253,480.09) .. controls (253,459.06) and (284.21,442) .. (322.7,442) .. controls (361.19,442) and (392.4,459.06) .. (392.4,480.09) .. controls (392.4,501.13) and (361.19,518.19) .. (322.7,518.19) .. controls (284.21,518.19) and (253,501.13) .. (253,480.09) -- cycle ;
	\draw [color={rgb, 255:red, 0; green, 0; blue, 0 }  ,draw opacity=1 ]   (280.4,464.19) .. controls (328.4,455.19) and (323.4,478.19) .. (366.4,463.19) ;
	\draw    (360.33,453) .. controls (348.92,500.28) and (372.92,468.28) .. (360.33,511.19) ;
	\draw  (41,521) -- (141,521)(51,431) -- (51,531) (134,516) -- (141,521) -- (134,526) (46,438) -- (51,431) -- (56,438)  ;
	\draw   (47.89,518.28) -- (53.35,523.88)(53.42,518.35) -- (47.82,523.81) ;
	\draw   (312.89,486.28) -- (318.35,491.88)(318.42,486.35) -- (312.82,491.81) ;
	\draw   (342.89,480.62) -- (348.35,486.21)(348.42,480.69) -- (342.82,486.14) ;
	\draw  [draw opacity=0][fill={rgb, 255:red, 80; green, 227; blue, 194 }  ,fill opacity=0.31 ] (51,430.25) -- (140.58,430.25) -- (140.58,521) -- (51,521) -- cycle ;
	
	\draw (343,515) node [anchor=north west][inner sep=0.75pt]  [font=\small]  {$D_{1}$};
	\draw (259,465) node [anchor=north west][inner sep=0.75pt]  [font=\small]  {$D_{2}$};
	\draw (246,425) node [anchor=north west][inner sep=0.75pt]    {$X$};
	\draw (40,523) node [anchor=north west][inner sep=0.75pt]  [font=\small]  {$1,2$};
	\draw (309,492) node [anchor=north west][inner sep=0.75pt]  [font=\small]  {$1$};
	\draw (336.67,485) node [anchor=north west][inner sep=0.75pt]  [font=\small]  {$2$};

\end{tikzpicture}

\caption{Open stratum $(\PP^2 \backslash D_1\cup D_2)^2 \subset \logProdexample{\PP^2}{D_1+D_2}{2}$} \label{fig:M1}
\end{figure}

\begin{figure}
    \centering

    \tikzset{every picture/.style={line width=0.75pt}} 
    
    \begin{tikzpicture}[x=0.75pt,y=0.75pt,yscale=-1,xscale=1]
    	
    	\draw  [draw opacity=0][fill={rgb, 255:red, 80; green, 227; blue, 194 }  ,fill opacity=0.31 ] (51,257.25) -- (140.58,257.25) -- (140.58,348) -- (51,348) -- cycle ;
    	\draw  [color={rgb, 255:red, 0; green, 0; blue, 0 }  ,draw opacity=1 ][fill={rgb, 255:red, 126; green, 211; blue, 33 }  ,fill opacity=0.3 ] (382.93,277) .. controls (381.09,277) and (379.6,283.51) .. (379.6,291.55) .. controls (379.6,299.58) and (381.09,306.09) .. (382.93,306.09) .. controls (384.76,306.09) and (386.25,312.61) .. (386.25,320.64) .. controls (386.25,328.67) and (384.76,335.19) .. (382.93,335.19) -- (356.33,335.19) .. controls (358.16,335.19) and (359.65,328.67) .. (359.65,320.64) .. controls (359.65,312.61) and (358.16,306.09) .. (356.33,306.09) .. controls (354.49,306.09) and (353,299.58) .. (353,291.55) .. controls (353,283.51) and (354.49,277) .. (356.33,277) -- cycle ;
    	\draw   (252,304.09) .. controls (252,283.06) and (283.21,266) .. (321.7,266) .. controls (360.19,266) and (391.4,283.06) .. (391.4,304.09) .. controls (391.4,325.13) and (360.19,342.19) .. (321.7,342.19) .. controls (283.21,342.19) and (252,325.13) .. (252,304.09) -- cycle ;
    	\draw [color={rgb, 255:red, 0; green, 0; blue, 0 }  ,draw opacity=1 ]   (279.4,288.19) .. controls (327.4,279.19) and (322.4,302.19) .. (365.4,287.19) ;
    	\draw   (312.89,306.28) -- (318.35,311.88)(318.42,306.35) -- (312.82,311.81) ;
    	\draw  (41,348) -- (141,348)(51,258) -- (51,358) (134,343) -- (141,348) -- (134,353) (46,265) -- (51,258) -- (56,265)  ;
    	\draw   (47.89,345.28) -- (53.35,350.88)(53.42,345.35) -- (47.82,350.81) ;
    	\draw   (87.89,345.28) -- (93.35,350.88)(93.42,345.35) -- (87.82,350.81) ;
    	\draw   (367.89,301.62) -- (373.35,307.21)(373.42,301.69) -- (367.82,307.14) ;
    	\draw    (90.4,347.91) -- (90.33,258) ;
    	\draw [shift={(90.33,256)}, rotate = 89.95] [color={rgb, 255:red, 0; green, 0; blue, 0 }  ][line width=0.75]    (10.93,-3.29) .. controls (6.95,-1.4) and (3.31,-0.3) .. (0,0) .. controls (3.31,0.3) and (6.95,1.4) .. (10.93,3.29)   ;
    	
    	\draw (309,312) node [anchor=north west][inner sep=0.75pt]  [font=\small]  {$1$};
    	\draw (342,339) node [anchor=north west][inner sep=0.75pt]  [font=\small]  {$D_{1}$};
    	\draw (258,289) node [anchor=north west][inner sep=0.75pt]  [font=\small]  {$D_{2}$};
    	\draw (245,249) node [anchor=north west][inner sep=0.75pt]    {$X$};
    	\draw (39,350) node [anchor=north west][inner sep=0.75pt]  [font=\small]  {$1$};
    	\draw (361.67,306) node [anchor=north west][inner sep=0.75pt]  [font=\small]  {$2$};
    	\draw (378.26,338.61) node [anchor=north west][inner sep=0.75pt]  [font=\small]  {$N_{D_{1}} X$};
    	\draw (84.33,350.33) node [anchor=north west][inner sep=0.75pt]  [font=\small]  {$2$};

    \end{tikzpicture}
    \caption{The boundary divisor with $p_2$ on an expansion component over $D_1$}
    \label{fig:enter-label}
\end{figure}

\subsubsection{Rubber action on static target.} \label{subsection:rubbertori}
Recall that we denote both a combinatorial type and its corresponding cone by $\tau$.

Fix a cone $\tau$ of $\tropmod$, and consider the corresponding restriction of the universal grid subdivision $\tropfam \to \tropmod$ to a subdivision $(\Sigma \times \tau)^{\dagger} \to \tau$. The cone $\tau$ induces the substack \[\B T_{\tau} \subset \A_{\tau}\subset \gridexp,\] where $\A_{\tau} = [U_{\tau}/T_{\tau}]$ is an open substack of $\gridexp$, the variety $U_{\tau}$ is the affine toric variety corresponding to the cone $\tau$, and $\B T_{\tau}$ is the quotient stack of the torus fixed point by $T_{\tau}$. 

\begin{proposition}\label{prop:family_over_torusstack}
	The restriction of the universal family $\X \to \gridexp$ to the closed substack $\B T_{\tau}$ can be written as \[\X_{\tau} = [Y_{\tau}/T_{\tau}] \to \B T_{\tau},\] where $Y_{\tau}$ is a grid expansion with combinatorial type $\tau$.
\end{proposition}
\begin{proof}
	Let $X_{\tau}$ be the logarithmic modification of $X \times U_{\tau}$ induced by $(\Sigma \times \tau)^{\dagger}$.
	The discussion in \cite[Section 2.1.3]{rubbertori} shows that the restriction of the universal family to $\A_{\tau}$ is $[X_{\tau}/T_{\tau}] \to \A_{\tau}$. The restriction to the closed substack $\B T_{\tau}$ is then $\X_{\tau} = [Y_{\tau}/T_{\tau}] \to \B T_{\tau}$, where $Y_{\tau}$ is the fibre of $X_{\tau} \to U_{\tau}$ over the torus fixed point. 
	We identify $\tau$ as the cone $\Rgeq^k$ for some $k\geq 1$ and consider the toric morphism $\AA^1 \to U_{\tau}$ given by the diagonal morphism $\Rgeq \to \tau$. Due to a result in \cite[Lemma 2.2.6]{uni_semistable}, the scheme-theoretic pullback 
	\begin{equation*}
		\begin{tikzcd}
			X_{\tau,\AA^1}\arrow[d] \arrow[r, hook]\arrow[dr, phantom, "\square"] & X_{\tau} \arrow[d,"\phi"]\\
			\AA^1 \arrow[r, hook] & U_{\tau}
		\end{tikzcd}
	\end{equation*}
	coincides with the fs-pullback, as the morphism $X_{\tau} \to U_{\tau}$ is weakly semistable. It follows that $Y_{\tau}$ is a grid expansion with combinatorial type $\tau$.

\end{proof}


As promised in Remark~\ref{rem:rubber_action}, we describe the \emph{rubber action} $\aut\colon  f_1^*\X \to f_2^*\X$ induced by a 2-isomorphism $\lambda$ between two morphisms factoring through some substack $\B T_{\tau}$, where $T_{\tau}$ is known as the \emph{rubber torus}. In this case, $f_1^*\X$ and $f_2^*\X$ are isomorphic and have combinatorial type $\tau$.

The 2-isomorphisms $\lambda$ between two morphisms $f_1\colon  S \to \gridexp$ and $f_2\colon  S\to \gridexp$ factoring through $\B T_{\tau}$ correspond to the elements $\lambda$ of the \emph{rubber torus} $T_{\tau}$, and the induced isomorphisms $\aut$ are called the \emph{rubber actions} of the \emph{rubber torus} $T_{\tau}$. They arise from the following 2-commutative diagram:
	\begin{equation*}
		\begin{tikzcd}
			& f_2^*\mathcal{X} \arrow[rd] \arrow[d, "\eta_2"]& & \\                               
			& S \arrow[rd, "f_2", ""{name=U,inner sep=1pt,below}] & \left[Y_{\tau}/T_{\tau}\right] \arrow[r] \arrow[d] \pullback{rd}  & \mathcal{X} \arrow[d]  \\
			f_1^*\mathcal{X} \arrow[rru] \arrow[d, "\eta_1"] \arrow[ruu, " \lambda^{\dagger}"]  &  & \B T_{\tau} \arrow[r]    & \mathcal{E}_n(X|D)  \\
			S \arrow[rru, "f_1"', ""{name=D,inner sep=1pt}] \arrow[ruu, "\mathrm{id}"']  \arrow[Rightarrow, from=D, to=U, "\lambda"]&                                                                          &    &                  
		\end{tikzcd}.
	\end{equation*}
	On the other hand, over every closed point of the family $\Spec k \to S$, the fibre of the family $f_1^*\X \to S$ is an $n$-marked grid expansion of the given combinatorial type $\tau = (\F \to \F_{\S}, m)$, given by an $n$-marked grid subdivision $\cal{P}$. The position of each vertex $v$ of $\cal{P}$ is determined by a \emph{tropical position map} $\varphi_v\colon  \tau \to \sigma_v$ with underlying lattice map $\varphi_v\colon  N_{\tau} \to N_{\sigma_v}$, where $\sigma_v$ is the minimal cone in $\Sigma$ containing $v$. Recalling that the tori corresponding to the cones $\tau$ and $\sigma_v$ are $T_{\tau} = N_{\tau} \otimes \Gm$ and $T_{\sigma_v} = N_{\sigma_v} \otimes \Gm$ respectively, the rubber action is then \[\varphi_v \otimes \Gm\colon  T_{\tau} \to T_{\sigma_v}\] acting on the component $Y_v$ corresponding to the vertex $v$. We refer the reader to \cite[Section 1]{rubbertori} for more details.

\begin{figure}
	\centering

	\tikzset{every picture/.style={line width=0.75pt}} 
	
	\begin{tikzpicture}[x=0.75pt,y=0.75pt,yscale=-1,xscale=1]
		
		\draw  [draw opacity=0][fill={rgb, 255:red, 80; green, 227; blue, 194 }  ,fill opacity=0.31 ] (51,257.25) -- (140.58,257.25) -- (140.58,348) -- (51,348) -- cycle ;
		\draw  [color={rgb, 255:red, 0; green, 0; blue, 0 }  ,draw opacity=1 ][fill={rgb, 255:red, 126; green, 211; blue, 33 }  ,fill opacity=0.3 ] (382.93,277) .. controls (381.09,277) and (379.6,283.51) .. (379.6,291.55) .. controls (379.6,299.58) and (381.09,306.09) .. (382.93,306.09) .. controls (384.76,306.09) and (386.25,312.61) .. (386.25,320.64) .. controls (386.25,328.67) and (384.76,335.19) .. (382.93,335.19) -- (356.33,335.19) .. controls (358.16,335.19) and (359.65,328.67) .. (359.65,320.64) .. controls (359.65,312.61) and (358.16,306.09) .. (356.33,306.09) .. controls (354.49,306.09) and (353,299.58) .. (353,291.55) .. controls (353,283.51) and (354.49,277) .. (356.33,277) -- cycle ;
		\draw   (252,304.09) .. controls (252,283.06) and (283.21,266) .. (321.7,266) .. controls (360.19,266) and (391.4,283.06) .. (391.4,304.09) .. controls (391.4,325.13) and (360.19,342.19) .. (321.7,342.19) .. controls (283.21,342.19) and (252,325.13) .. (252,304.09) -- cycle ;
		\draw [color={rgb, 255:red, 0; green, 0; blue, 0 }  ,draw opacity=1 ]   (279.4,288.19) .. controls (327.4,279.19) and (322.4,302.19) .. (365.4,287.19) ;
		\draw   (312.89,306.28) -- (318.35,311.88)(318.42,306.35) -- (312.82,311.81) ;
		\draw  (41,348) -- (141,348)(51,258) -- (51,358) (134,343) -- (141,348) -- (134,353) (46,265) -- (51,258) -- (56,265)  ;
		\draw   (47.89,345.28) -- (53.35,350.88)(53.42,345.35) -- (47.82,350.81) ;
		\draw   (87.89,345.28) -- (93.35,350.88)(93.42,345.35) -- (87.82,350.81) ;
		\draw   (367.89,301.62) -- (373.35,307.21)(373.42,301.69) -- (367.82,307.14) ;
		\draw    (357,268.2) -- (381.92,268.2) ;
		\draw [shift={(384.92,268.2)}, rotate = 180] [fill={rgb, 255:red, 0; green, 0; blue, 0 }  ][line width=0.08]  [draw opacity=0] (10.72,-5.15) -- (0,0) -- (10.72,5.15) -- (7.12,0) -- cycle    ;
		\draw [shift={(354,268.2)}, rotate = 0] [fill={rgb, 255:red, 0; green, 0; blue, 0 }  ][line width=0.08]  [draw opacity=0] (10.72,-5.15) -- (0,0) -- (10.72,5.15) -- (7.12,0) -- cycle    ;
		\draw    (78.4,375) -- (103.32,375) ;
		\draw [shift={(106.32,375)}, rotate = 180] [fill={rgb, 255:red, 0; green, 0; blue, 0 }  ][line width=0.08]  [draw opacity=0] (10.72,-5.15) -- (0,0) -- (10.72,5.15) -- (7.12,0) -- cycle    ;
		\draw [shift={(75.4,375)}, rotate = 0] [fill={rgb, 255:red, 0; green, 0; blue, 0 }  ][line width=0.08]  [draw opacity=0] (10.72,-5.15) -- (0,0) -- (10.72,5.15) -- (7.12,0) -- cycle    ;
		\draw    (90.4,347.91) -- (90.33,258) ;
		\draw [shift={(90.33,256)}, rotate = 89.95] [color={rgb, 255:red, 0; green, 0; blue, 0 }  ][line width=0.75]    (10.93,-3.29) .. controls (6.95,-1.4) and (3.31,-0.3) .. (0,0) .. controls (3.31,0.3) and (6.95,1.4) .. (10.93,3.29)   ;
		
		\draw (309,312) node [anchor=north west][inner sep=0.75pt]  [font=\small]  {$1$};
		\draw (342,339) node [anchor=north west][inner sep=0.75pt]  [font=\small]  {$D_{1}$};
		\draw (258,289) node [anchor=north west][inner sep=0.75pt]  [font=\small]  {$D_{2}$};
		\draw (245,249) node [anchor=north west][inner sep=0.75pt]    {$X$};
		\draw (39,350) node [anchor=north west][inner sep=0.75pt]  [font=\small]  {$1$};
		\draw (361.67,306) node [anchor=north west][inner sep=0.75pt]  [font=\small]  {$2$};
		\draw (378.26,338.61) node [anchor=north west][inner sep=0.75pt]  [font=\small]  {$N_{D_{1}} X$};
		\draw (84.33,350.33) node [anchor=north west][inner sep=0.75pt]  [font=\small]  {$2$};

	\end{tikzpicture}
	\caption{An example of a rubber action where the positioning of $u_2$ on ray $D_1$ corresponds to a $\Gm$ action by scaling on the fibres of the $\PP^1$-bundle}
	\label{fig:rubber}
\end{figure}

\subsubsection{Boundary strata as moduli spaces of points on expansions of fixed combinatorial types}
We let $\fam_{\tau}$ be the restriction of the universal target $\fam \to \logProd$ to a stratum $V_{\tau}$. A consequence of the construction of $\logProd$ is that $V_{\tau}$ is the moduli space of isomorphism classes of stable $n$-pointed grid expansions of fixed combinatorial type $\tau$, with universal family $\fam_{\tau} \to V_{\tau}$ and $n$ sections $\mathbf{s}_i$. 

\begin{observation}\label{obs:stability}
	There cannot be any non-trivial automorphisms of a stable $n$-pointed grid expansion $(X_{\cal{P}}, p_1, p_2, \dots, p_n)$. Let the combinatorial type be $\tau$. The primitive elements of cone $\tau$ are given by the coordinates $u_i^{(j)}$ of $\Sigma^n$, possibly zero or non-distinct in $\tau$. For each non-zero $u_i^{(j)}$, its image under \[\varphi_{m(i)}\colon\tau \to \sigma_{m(i)}\] is then non-zero. This means that $u_i^{(j)} \otimes \Gm$ acts non-trivially on the fibres of the $(\PP^1)^k$-bundle $Y_{m(i)}$ which contain the points $m^{-1}(m(i))$, and the action on the points is free since the points lie in the torus $\Gm^k$ of the $(\PP^1)^k$-fibres.
\end{observation}

\subsection{Blow-up description of $\logProd$}
Let $(X,D)$ be a simple normal crossings pair where $D=D_1 + D_2 + \dots + D_r$. In this subsection, we prove that $\logProd$ can be written as an iterated blow-up. 
\begin{definition}[Dominant transform {\cite[Definition 2.7]{wonderful}}]
	Let $Z$ be a non-singular subscheme of a non-singular scheme $Y$, and let $\pi \colon \bl{Z}{Y} \to Y$ be the blow-up of $Y$ along $Z$. The \emph{dominant transform} of a subscheme $V$ of $Y$, denoted by $\widetilde{V}$, is either
	\begin{itemize}
		\item the strict transform of $V$ if $V \not\subset Z$, or
		\item the total transform $\pi^{-1}(V)$ otherwise.
	\end{itemize}
\end{definition}

Write $Y_0 \colonequals X^n$. For each $k = 1,2, \dots, r$, we consider loci of $n$-tuples 
\[D_{k,I} \colonequals \{(x_1, x_2, \dots, x_n) \in X^n\; |\; x_i \in D_k \;\; \forall i \in I\}.
\]
We denote $Y_k$ as the iterated blow-up of $Y_{k-1}$ in the sequence of centres \begin{equation}\label{seq:logProd}
	\widetilde{D}_{k,\{1\}};\widetilde{D}_{k,\{1,2\}}, \widetilde{D}_{k,\{2\}}; \widetilde{D}_{k,\{1,2,3\}}, \widetilde{D}_{k,\{1,3\}}, \widetilde{D}_{k,\{2,3\}}, \widetilde{D}_{k,\{3\}}; \dots; \widetilde{D}_{k,\{1,2,\dots,n\}}, \dots, \widetilde{D}_{k,\{n\}}, \tag{$*$}
\end{equation}
where $\widetilde{D}_{k,I}$ denotes the dominant transform of $D_{k,I}$ under the existing iterated blow-up. By \cite[Theorem 1]{KS}, the space $Y_1$ coincides with the moduli space $\logProdexample{X}{D_1}{n}$. 

\begin{proposition}\label{prop:blow-up}
	\hfill
	\begin{enumerate}
		\item There exists a Cartesian diagram in the fine and saturated logarithmic category (henceforth called fs-pullback diagram)
		\begin{equation*}
			\begin{tikzcd}
				Y_r \arrow[r] \arrow[d] \fspullback{rd} & \prod_{i=1}^{r} \logProdexample{X}{D_i}{n} \arrow[d]\\
				X^n \arrow[r, "\Delta", hook] & (X^n)^r.
			\end{tikzcd}
		\end{equation*}
		\item The moduli space $\logProd$ coincides with $Y_r$, and $\logProd \to X^n$ is a \emph{log blow-up} (see below).
	\end{enumerate}
\end{proposition}
We prove this proposition in several steps. 

Let $X$ be a logarithmically smooth scheme. Consider a logarithmic modification $\widetilde{X} \to X$ induced by a \emph{star subdivision} $\widetilde{\Sigma_X} \to \Sigma_X$ \cite[Definition~3.3.13]{cls}: this subdivision corresponds to toric blow-ups on each toric chart. This is an example of a log blow-up introduced in \cite[Section~1.3.3]{kato}: log blow-ups are induced by subdivisions of $\Sigma_X$ into domains of linearity of a piecewise linear function $\phi$ on $\Sigma_X$ \cite{molcho-pand-schmitt}. 
\begin{remark}
	A log blow-up is a blow-up as the morphism $X \to \A_X$ is strict and flat, and fs-pullbacks along a strict morphism coincide with scheme-theoretic pullbacks.
\end{remark}
We state a lemma where the distinction between a scheme-theoretic pullback and an fs-pullback is crucial.
\begin{lemma}[Log blow-ups are stable under base change] \label{lem:fs-pullback}
	Given a morphism of logarithmic schemes $f \colon Y \to X$, the fs-pullback along $f$ of a logarithmic blow-up $\widetilde{X} \to X$ coincides with the strict transform $\widetilde{Y}$.
	\begin{equation*}
		\begin{tikzcd}
			\widetilde{Y} \arrow[r] \arrow[d] \fspullback{rd} & \widetilde{X}\arrow[d]\\
			Y \arrow[r, "f"] & X
		\end{tikzcd}
	\end{equation*}
\end{lemma}
\begin{proof}
	This follows from Theorem~4.7 and Corollary~4.8 of \cite{logblowup}. We give another argument in the language of Artin fans. 
	
	This question is \'etale local, so we pass to open neighbourhoods of $X$ and $Y$, and assume that $X$ and $Y$ are \emph{atomic}, i.e. they admit a global toric chart \cite[Definition~2.2.2.2]{logPic}. Under this assumption, functoriality of Artin fans holds by \cite[Lemma~A.18]{lifting_rational} and we have maps $\A_Y \to \A_X$ and $\Sigma_f \colon \Sigma_Y \to \Sigma_X$. We observe that $\widetilde{\Sigma_Y} = \widetilde{\Sigma_X} \times_{\Sigma_X} \Sigma_Y$ is a subdivision of $\Sigma_Y$  inducing a logarithmic modification $\widetilde{Y} \to Y$, i.e.~$\widetilde{Y} = Y \times_{\A_Y} \widetilde{\A_Y}$. This subdivision is induced by the pullback of $\phi$ along $\Sigma_f$, which is a piecewise linear function on $\Sigma_Y$. Therefore, we observe that $\widetilde{Y}$ is the strict transform of $f$ under the log blow-up $\widetilde{X} \to X$. 

	We thus have a commutative diagram in the fs logarithmic category
	\begin{equation*}
		\begin{tikzcd}
			\widetilde{Y} \arrow[r] \arrow[d] & \widetilde{X} \arrow[d] \arrow[r] \fspullback{rd}& \widetilde{\A_X} \arrow[d]\\
			Y \arrow[r, "f"] & X \arrow[r] & \A_X,
		\end{tikzcd}
	\end{equation*}
	where the outer square and the right square are Cartesian by construction, so the left square is also Cartesian.
\end{proof}
We now introduce the notion of \emph{transversality} and a technical result about blow-ups.
\begin{definition}[Transversality]
	Let $Y$ and $Z$ be smooth subvarieties of a smooth variety $X$. Then $Y$ and $Z$ \emph{intersect transversely} if, at every point $p\in X$, 
	\begin{enumerate}
		\item there exists a regular sequence $f_1, f_2, \ldots, f_d \in \mathcal{O}_{X,p}$, where $d = \dim X$, such that $\mathfrak{m}_p = (f_1, f_2, \ldots, f_d)$, and
		\item there exist integers $0 \leq k_1 \leq k_2 \leq d$ such that at $p$, the ideal sheaves are $(\mathcal{I}_Y)_p = (f_1, f_2, \ldots, f_{k_1})$ and $(\mathcal{I}_Z)_p = (f_{k_1+1}, f_{k_1+2}, \ldots, f_{k_2})$.
	\end{enumerate}
\end{definition}

\begin{lemma} \label{lem:strict_transverse}
	Let $Y$ and $Z$ be smooth subvarieties of $X$ intersecting transversely, and consider the blow-up $\bl{Z}{X} \to X$. The strict transform of $Y$ is equal to the total transform.
\end{lemma}
\begin{proof}
	The regular sequence $f_1, f_2, \ldots, f_d$ gives a flat morphism $\Spec \mathcal{O}_{X,p} \to \AA^d$. Consider subvarieties $V = \mathbb{V}(x_1, x_2, \ldots, x_{k_1})$ and $W = \mathbb{V}(x_{k_1+1}, x_{k_1+2}, \ldots, x_{k_2})$ of $\AA^d$, and it is straightforward to verify that under the blow-up $\bl{W}{\AA^d} \to \AA^d$, the strict transform of $V$ coincides with the total transform. We conclude by flat base change.
\end{proof}

\begin{lemma}\label{lem:inductive}
	Consider the iterated blow-up $\phi_{k-1}\colon Y_{k-1} \to X^n$. For each $k = 1, 2, \dots, r$, there exists an fs-pullback diagram
	\begin{equation*}
		\begin{tikzcd}
			Y_k \arrow[r] \arrow[d] \fspullback{rd} & \logProdexample{X}{D_k}{n} \arrow[d]\\
			Y_{k-1} \arrow[r] & X^n.
		\end{tikzcd}
	\end{equation*}
\end{lemma}
\begin{proof}
	We note that the logarithmic modification $\logProdexample{X}{D_k}{n} \to X^n$ is induced by the subdivision $\Pi_n(\Sigma_{(k)}) \to \Sigma_{(k)}^n$ where $\Sigma_{(k)}$ is the tropicalisation of the smooth divisor pair $(X,D_k)$. We can verify that this subdivision is the star subdivision corresponding to the iterated blow-up of $X^n$ in the sequence of strata (\ref{seq:logProd}) given in \cite[Theorem 1]{KS}, hence $\logProdexample{X}{D_k}{n} \to X^n$ is a log blow-up.

	Let $Y'_k$ denote the fs-pullback.
	Lemma~\ref{lem:fs-pullback} implies that $Y'_k$ is the iterated blow-up of $Y_{k-1}$ in the sequence of centres where every term $\widetilde{D}_{k,I}$ in the sequence (\ref{seq:logProd}) is replaced by the total transform $\phi_{k-1}^{-1}(\widetilde{D}_{k,I})$. The blow-up centres of $\phi_{k-1}$ are the dominant transforms of $D_{i,J}$ where $1 \leq i \leq k-1$, but as the intersections of divisors $D_i$ and $D_k$ are transverse, therefore the intersections of $D_{i,J}$ and $D_{k,I}$ are transverse, and so are the intersections of their dominant transforms \cite[Lemma~2.9(v)]{wonderful}. Therefore, by Lemma~\ref{lem:strict_transverse}, under the blow-up $\phi_{k-1} \to X^n$, the total transform $\phi_{k-1}^{-1}(\widetilde{D}_{k,I})$ is equal to the strict transform $\widetilde{D}_{k,I}$. We conclude that $Y'_k = Y_k$.
\end{proof}
\begin{proof}[Proof of Proposition~\ref{prop:blow-up}] 
	\hfill
	\begin{enumerate}
		\item We prove by induction that for all $k = 1, 2, \dots, r$, there exists a fs-pullback diagram
		\begin{equation*}
			\begin{tikzcd}
				Y_k \arrow[r] \arrow[d] \fspullback{rd} & \logProdexample{X}{D_1}{n} \times \dots\times \logProdexample{X}{D_{k-1}}{n} \times \logProdexample{X}{D_k}{n} \times X^n \times \dots X^n \arrow[d]\\
				X^n \arrow[r, "\Delta", hook] & (X^n)^r.
			\end{tikzcd}
		\end{equation*}
		But both the base case $k=1$ and the inductive step follow from Lemma~\ref{lem:inductive}.
		\item This follows from Lemma~\ref{lem:prod_of_smooth_pairs}, and the fact that log blow-ups are stable under base change (Lemma~\ref{lem:fs-pullback}).
	\end{enumerate}
	
\end{proof}

\begin{lemma} \label{lem:prod_of_smooth_pairs}
	There exists a Cartesian diagram in the category of schemes and in the fine and saturated logarithmic category
	\begin{equation}\label{cd:prod_of_smooth}
		\begin{tikzcd}
			\logProd \arrow[r] \arrow[d] \pullback{rd} & \prod_{i=1}^{r} \logProdexample{X}{D_i}{n} \arrow[d]\\
			X^n \arrow[r, "\Delta", hook] & (X^n)^r.
		\end{tikzcd}
	\end{equation}
\end{lemma}
\begin{proof}
	Let $\Sigma_i = \Rgeq$ be the tropicalisation of the smooth pair $(X,D_i)$, with Artin fan $\A_i = [\AA^1 / \Gm]$. Observe that we have open embeddings
	\begin{equation*}
		\begin{tikzcd}
			\gridexp \arrow[r,hook] \arrow[d] \pullback{rd} &\prod_{i=1}^{r}\cal{E}_n(X|D_i) \arrow[d]\\
			\A_X^n \arrow[r, hook] & \prod_{i=1}^{r} \A_i^n.
		\end{tikzcd}
	\end{equation*}
	The strict morphism $X^n \to \A_X^n \hookrightarrow\prod_{i=1}^{r} \A_i^n$ factors through $X^n \xhookrightarrow{\Delta} (X^n)^r$, giving a commutative diagram
	\begin{equation*}
		\begin{tikzcd}
			\logProd \arrow[r] \arrow[d] & \prod_{i=1}^{r}\logProdexample{X}{D_i}{n} \arrow[r] \arrow[d] \pullback{rd} & \prod_{i=1}^{r}\cal{E}_n(X|D_i) \arrow[d]\\
			X^n \arrow[r, "\Delta", hook] & (X^n)^r \arrow[r, "\text{strict}"] & \prod_{i=1}^{r} \A_i^n
		\end{tikzcd}
	\end{equation*}
	where both the outer and the right squares are Cartesian, thus we obtain the desired Cartesian diagram.
\end{proof}

\subsection{Examples} 
\subsubsection{Cases where $n$ is small}
  \begin{itemize}
  	\item $n=1$: For any simple normal crossings pair $(X,D)$, we have $\logProdexample{X}{D}{1} = X$.
  	\item $n=2$: For a smooth pair $(X,D)$, we have $\logProdexample{X}{D}{2} = \bl{D_{1,2}}{X^2}$, where $D_{1,2}$ is the locus where both points are on $D$.
  \end{itemize}
  
\subsubsection{Losev--Manin moduli spaces  and the permutohedral variety  }\label{section:losev_manin}
When $X = \PP^1$ and $D = 0 + \infty$, the construction of $\logProd$ recovers the Losev--Manin moduli space of $n+1$ points \cite{LM}, the $n$-dimensional permutohedral variety and the dimension one toric configuration space of $n$ points (see \cite[Section 3]{toric_conf}).

\section{Logarithmic Fulton--MacPherson compactification} \label{section:logFM}
In this section, we will construct a space $\logFM$ parametrising Fulton--MacPherson-type degenerations of the stable $n$-pointed grid expansions, separating the $n$ points. 

\subsection{Stable $n$-pointed FM grid expansions}
\begin{definition}[{cf. \cite[p. 194]{FM}}] \label{def:FMgridexp}
	An \emph{FM grid expansion} is an iterative construction $X_{\cal{P}}^{\text{FM}} \colonequals X_{\cal{P}}^{(N)}$ with finite steps, where $X_{\cal{P}}^{(0)}$ is some grid expansion $X_{\cal{P}}$ and at each step, \[X_{\cal{P}}^{(k)} \colonequals \bl{x}{X_{\cal{P}}^{(k-1)}} \coprod_{\PP(T_x X_{\cal{P}}^{(k-1)})} \PP(T_x X_{\cal{P}}^{(k-1)} \oplus \mathbbm{1}),\] where $x$ is a point on the smooth locus of $X_{\cal{P}}^{(k-1)}$. 
	
	An \emph{$n$-pointed FM grid expansion} is the data $(X_{\cal{P}}^{\text{FM}}, p_1, p_2, \ldots, p_n)$ where $p_1, p_2, \dots, p_n$ are distinct points on the smooth locus of $X_{\cal{P}}^{\text{FM}}$. 
	
	It is said to be \emph{stable} if
	\begin{itemize}
		\item $(X_{\cal{P}}^{\text{FM}}, p_1, p_2, \ldots, p_n)$ maps to a stable $n$-pointed grid expansion $(X_{\cal{P}}^{(0)},p_1', p_2', \ldots, p_n')$,
		\item every component of $X_{\cal{P}}^{\text{FM}}$, apart from those mapping birationally onto a component of the grid expansion $X_{\cal{P}}$, contains at least three \emph{markings}, consisting of marked points or intersections with other components along exceptional divisors. 
	\end{itemize}
\end{definition}

See Figures~\ref{fig:3-FMexp} and \ref{fig:4-FMexp} for examples. We recover the definition in \cite[p. 194]{FM} if we take the trivial grid expansion $X_{\cal{P}} = X$.

\begin{definition}\label{def:forest}
	Given a stable $n$-pointed FM grid expansion $(X_{\mathcal{P}}^{\text{FM}}, p_1, p_2, \ldots, p_n)$, let the vertex set of $\mathcal{P}$ be $V(\mathcal{P}) = \{v_1, v_2, \ldots,  v_k\}$.
	The \emph{combinatorial type} of $(X_{\mathcal{P}}^{\text{FM}}, p_1, p_2, \ldots, p_n)$ is a tuple $(\F \to \F_{\S}, m, \T_{v_1}, \dots, \T_{v_k})$, where $(\F \to \F_{\S}, m)$ is the combinatorial type of $X_{\mathcal{P}}$, and each $\T_{v_i}$ is a rooted tree defined as follows:
	\begin{itemize}
		\item \textit{Vertices:} The tree $\T_{v_i}$ has a vertex for each component of $X_{\mathcal{P}}^{\text{FM}}$ which contracts to the component of $X_{\mathcal{P}}$ corresponding to $v_i$. The \emph{root vertex} of $\T_{v_i}$ is $v_i$.
		\item \textit{Edges:} Two vertices share an edge if the corresponding components intersect.
		\item \textit{Legs:} For each
		 vertex $v$ of $\T_{v_i}$, attach a leg to $v$ for each marked point $p_j$ contained in the component of $X_{\mathcal{P}}^{\text{FM}}$ corresponding to $v$.
	\end{itemize}
	
\end{definition}

Following \cite[Section 2.5]{stable_ramified}, we consider the automorphisms of FM grid expansions $\psi\colon  X_{\cal{P}}^{(N)} \to X_{\cal{P}}^{(N)}$ without points. On each component corresponding to a root vertex in its combinatorial type, which we call a \emph{root component}, its automorphisms $\psi$ are lifts of the rubber actions. Each component whose vertex is non-root and of valence 1, is called an \emph{end component} and is of the form $\PP(T_x X_{\cal{P}}^{(k)} \oplus \mathbbm{1})$. Its automorphism group fixing the hyperplane at infinity is $\Ga^d \rtimes \Gm$, where $d$ is the dimension of $X$ and $\Ga$ is the additive group. Each component whose vertex is non-root and of valence 2 is called a \emph{ruled component}, which is a ruled variety of the form $\bl{x}{\PP(T_y X_{\cal{P}}^{(k-1)} \oplus \mathbbm{1})}$. Its automorphism group fixing the intersections with the other two components is $\Gm$, which acts by scaling the family of lines.

\begin{definition}
	An \emph{isomorphism of stable $n$-pointed FM grid expansions} \[\psi\colon  (X_{\cal{P}}^{\text{FM}}, p_1, p_2, \dots, p_n) \to (X_{\cal{P}}^{\text{FM}}, p_1', p_2', \dots, p_n')\] is an automorphism $\psi$ of FM grid expansions sending $p_i$ to $p_i'$.
\end{definition}

We see that the stability condition guarantees that there are no automorphisms of stable $n$-pointed FM grid expansions.

\begin{remark}
	The stability condition for $n$-pointed FM grid expansions can be restated in terms of trees -- an $n$-pointed grid expansion is stable if and only if every rooted tree in the combinatorial type is at least 3-valent, except at the root vertices.
\end{remark}

We define a tropical analogue of stable $n$-pointed FM grid expansions, and its combinatorial type.
\begin{definition}
	A \emph{planted forest} is a tuple $(\cal{P}, m, \T^{\mathrm{m}}_{v_1}, \ldots, \T^{\mathrm{m}}_{v_k})$. The pair $(\cal{P}, m)$ is an $n$-marked grid subdivision, and $v_1, v_2, \ldots, v_k$ are not necessarily distinct vertices of $\cal{P}$. Each $\T^{\mathrm{m}}_{v_i}$ is a rooted metric tree with legs and root $v_i$, such that the legs on $\T^{\mathrm{m}}_{v_i}$ are in bijection with markings in $m(v_i)$, and every vertex apart from the root is at least 3-valent.
\end{definition}
\begin{definition}
	The \emph{combinatorial type} of a planted forest $(\cal{P}, m, \T^{\mathrm{m}}_{v_1}, \ldots, \T^{\mathrm{m}}_{v_k})$ is $(\F \to \F_{\S}, m, \T_{v_1}, \ldots, \T_{v_k})$, where each $\T_{v_i}$ is obtained by forgetting the metric on $\T^{\mathrm{m}}_{v_i}$.
\end{definition}

\begin{remark} \label{rem:FMexp}
	We can extend all definitions in this subsection to other types of expansions. This will be useful in Section~\ref{section:degeneration}.
\end{remark}
%

Below are some examples of planted forests:

\begin{figure}[H]
	\begin{minipage}{.5\textwidth}
		\centering
		\tikzset{every picture/.style={line width=0.75pt}} 
		
		\tikzset{every picture/.style={line width=0.75pt}} 

		\tikzset{every picture/.style={line width=0.75pt}} 
		
		\begin{tikzpicture}[x=0.75pt,y=0.75pt,yscale=-1,xscale=1]
			
			\draw  (419,660.09) -- (519,660.09)(480.96,570.09) -- (423.23,670.09) (514.89,655.09) -- (519,660.09) -- (509.11,665.09) (471.92,577.09) -- (480.96,570.09) -- (481.92,577.09)  ;
			\draw   (425.89,657.37) -- (431.35,662.97)(431.42,657.44) -- (425.82,662.9) ;
			\draw   (465.89,657.37) -- (471.35,662.97)(471.42,657.44) -- (465.82,662.9) ;
			\draw    (521.99,572.62) -- (469.45,659.82) ;
			\draw [shift={(523.02,570.91)}, rotate = 121.07] [color={rgb, 255:red, 0; green, 0; blue, 0 }  ][line width=0.75]    (10.93,-3.29) .. controls (6.95,-1.4) and (3.31,-0.3) .. (0,0) .. controls (3.31,0.3) and (6.95,1.4) .. (10.93,3.29)   ;
			\draw [color={rgb, 255:red, 0; green, 100; blue, 0 }  ,draw opacity=1 ][line width=1.5]    (459.45,647.91) -- (469.45,659.82) ;
			\draw [color={rgb, 255:red, 0; green, 100; blue, 0 }  ,draw opacity=1 ][line width=1.5]    (465.45,647.91) -- (469.45,659.82) ;
			\draw [color={rgb, 255:red, 0; green, 100; blue, 0 }  ,draw opacity=1 ][line width=1.5]    (472.45,647.91) -- (469.45,659.82) ;
			\draw [color={rgb, 255:red, 0; green, 100; blue, 0 }  ,draw opacity=1 ][line width=1.5]    (479.45,647.91) -- (469.45,659.82) ;
			\draw [color={rgb, 255:red, 0; green, 100; blue, 0 }  ,draw opacity=1 ][line width=1.5]    (421.45,650.91) -- (429,660.09) ;
			\draw [color={rgb, 255:red, 0; green, 100; blue, 0 }  ,draw opacity=1 ][line width=1.5]    (439.45,650.18) -- (429,660.09) ;
			\draw [color={rgb, 255:red, 0; green, 100; blue, 0 }  ,draw opacity=1 ][line width=1.5]    (428.45,648.91) -- (429,660.09) ;
			\draw  [draw opacity=0][fill={rgb, 255:red, 80; green, 227; blue, 194 }  ,fill opacity=0.28 ] (481.03,569.96) -- (569.49,569.96) -- (517.46,660.09) -- (429,660.09) -- cycle ;
			
			\draw (461,663) node [anchor=north west][inner sep=0.75pt]   [align=left] {$\displaystyle I$};
			\draw (413,661) node [anchor=north west][inner sep=0.75pt]   [align=left] {$\displaystyle I^{c}$};
			\draw (529,650.09) node [anchor=north west][inner sep=0.75pt]   [align=left] {$\displaystyle D_{k}$};
			\draw (436,546.09) node [anchor=north west][inner sep=0.75pt]   [align=left] {Other divisors};

		\end{tikzpicture}

		\caption{A planted forest where points with labels in $I$ lie on an expansion along $D_k$.}
		\label{fig:forest_div_exp}
	\end{minipage}%
	\begin{minipage}{.5\textwidth}
		\centering

	\tikzset{every picture/.style={line width=0.75pt}} 
	
	\begin{tikzpicture}[x=0.75pt,y=0.75pt,yscale=-1,xscale=1]
		
		\draw  (261,819.73) -- (361,819.73)(322.96,729.73) -- (265.23,829.73) (356.89,814.73) -- (361,819.73) -- (351.11,824.73) (313.92,736.73) -- (322.96,729.73) -- (323.92,736.73)  ;
		\draw   (267.89,817.01) -- (273.35,822.61)(273.42,817.08) -- (267.82,822.54) ;
		\draw [color={rgb, 255:red, 0; green, 100; blue, 0 }  ,draw opacity=1 ][line width=1.5]    (271,807.55) -- (271,819.73) ;
		\draw [color={rgb, 255:red, 0; green, 100; blue, 0 }  ,draw opacity=1 ][line width=1.5]    (264.45,798.55) -- (270.45,807.55) ;
		\draw [color={rgb, 255:red, 0; green, 100; blue, 0 }  ,draw opacity=1 ][line width=1.5]    (277,797.73) -- (270.45,807.55) ;
		\draw [color={rgb, 255:red, 0; green, 100; blue, 0 }  ,draw opacity=1 ][line width=1.5]    (255.45,799.55) -- (270.45,807.55) ;
		\draw [color={rgb, 255:red, 0; green, 100; blue, 0 }  ,draw opacity=1 ][line width=1.5]    (270.45,807.55) -- (284.45,799.55) ;
		\draw [color={rgb, 255:red, 0; green, 100; blue, 0 }  ,draw opacity=1 ][line width=1.5]    (264.45,810.55) -- (271,819.73) ;
		\draw [color={rgb, 255:red, 0; green, 100; blue, 0 }  ,draw opacity=1 ][line width=1.5]    (282.45,812.18) -- (271,819.73) ;
		\draw [color={rgb, 255:red, 0; green, 100; blue, 0 }  ,draw opacity=1 ][line width=1.5]    (258.45,812.55) -- (271,819.73) ;
		\draw  [draw opacity=0][fill={rgb, 255:red, 80; green, 227; blue, 194 }  ,fill opacity=0.28 ] (323.03,729.6) -- (411.49,729.6) -- (359.46,819.73) -- (271,819.73) -- cycle ;
		
		\draw (258,776.73) node [anchor=north west][inner sep=0.75pt]    {$J$};
		\draw (242,805.73) node [anchor=north west][inner sep=0.75pt]    {$J^{c}$};

	\end{tikzpicture}
		\caption{A planted forest where all points are on $X \backslash D$, and points with labels in $J$ are on an FM degeneration.}
		\label{fig:forest_div_fm}
	\end{minipage}
	
\end{figure}

Below are some examples of stable $n$-pointed Fulton--MacPherson grid expansions with their combinatorial types:
\begin{figure}[H]
	\centering

	\tikzset{every picture/.style={line width=0.75pt}} 
	
	\begin{tikzpicture}[x=0.75pt,y=0.75pt,yscale=-1,xscale=1]
		
		\draw  (41.89,116.6) -- (141.89,116.6)(103.85,26.6) -- (46.12,126.6) (137.78,111.6) -- (141.89,116.6) -- (132,121.6) (94.81,33.6) -- (103.85,26.6) -- (104.81,33.6)  ;
		\draw  [draw opacity=0][fill={rgb, 255:red, 80; green, 227; blue, 194 }  ,fill opacity=0.31 ] (104.28,25.85) -- (193.86,25.85) -- (141.47,116.6) -- (51.89,116.6) -- cycle ;
		
		\draw    (140.83,29.42) -- (90.91,115.88) ;
		
		\draw [color={rgb, 255:red, 0; green, 100; blue, 0 }  ,draw opacity=1 ][line width=1.5]    (91,98.28) -- (91,116.28) ;
		\draw [color={rgb, 255:red, 0; green, 100; blue, 0 }  ,draw opacity=1 ][line width=1.5]    (80.91,88.19) -- (91,98.28) ;
		\draw [color={rgb, 255:red, 0; green, 100; blue, 0 }  ,draw opacity=1 ][line width=1.5]    (91,98.28) -- (101.4,87.88) ;
		\draw   (250,91.09) .. controls (250,70.06) and (281.21,53) .. (319.7,53) .. controls (358.19,53) and (389.4,70.06) .. (389.4,91.09) .. controls (389.4,112.13) and (358.19,129.19) .. (319.7,129.19) .. controls (281.21,129.19) and (250,112.13) .. (250,91.09) -- cycle ;
		\draw [color={rgb, 255:red, 0; green, 0; blue, 0 }  ,draw opacity=1 ]   (277.4,75.19) .. controls (325.4,66.19) and (320.4,89.19) .. (363.4,74.19) ;
		\draw  [color={rgb, 255:red, 0; green, 0; blue, 0 }  ,draw opacity=1 ][fill={rgb, 255:red, 126; green, 211; blue, 33 }  ,fill opacity=0.3 ] (380.93,64) .. controls (379.09,64) and (377.6,70.51) .. (377.6,78.55) .. controls (377.6,86.58) and (379.09,93.09) .. (380.93,93.09) .. controls (382.76,93.09) and (384.25,99.61) .. (384.25,107.64) .. controls (384.25,115.67) and (382.76,122.19) .. (380.93,122.19) -- (354.33,122.19) .. controls (356.16,122.19) and (357.65,115.67) .. (357.65,107.64) .. controls (357.65,99.61) and (356.16,93.09) .. (354.33,93.09) .. controls (352.49,93.09) and (351,86.58) .. (351,78.55) .. controls (351,70.51) and (352.49,64) .. (354.33,64) -- cycle ;
		\draw   (310.89,93.28) -- (316.35,98.88)(316.42,93.35) -- (310.82,98.81) ;
		\draw  [fill={rgb, 255:red, 0; green, 0; blue, 0 }  ,fill opacity=1 ] (365.81,91.09) .. controls (365.81,89.94) and (366.75,89) .. (367.91,89) .. controls (369.06,89) and (370,89.94) .. (370,91.09) .. controls (370,92.25) and (369.06,93.19) .. (367.91,93.19) .. controls (366.75,93.19) and (365.81,92.25) .. (365.81,91.09) -- cycle ;
		\draw   (88.49,114.28) -- (93.95,119.88)(94.02,114.35) -- (88.42,119.81) ;
		\draw   (48.49,114.28) -- (53.95,119.88)(54.02,114.35) -- (48.42,119.81) ;
		\draw  [fill={rgb, 255:red, 74; green, 144; blue, 226 }  ,fill opacity=0.7 ] (385.52,43.79) .. controls (387.99,40.78) and (392.42,40.36) .. (395.43,42.83) -- (411.73,56.25) .. controls (414.73,58.72) and (415.16,63.16) .. (412.69,66.16) -- (382.17,103.22) .. controls (382.17,103.22) and (382.17,103.22) .. (382.17,103.22) -- (355,80.85) .. controls (355,80.85) and (355,80.85) .. (355,80.85) -- cycle ;
		\draw   (384.96,59.21) -- (390.56,64.95)(390.63,59.28) -- (384.89,64.88) ;
		\draw   (375.96,76.21) -- (381.56,81.95)(381.63,76.28) -- (375.89,81.88) ;
		
		\draw (40,119) node [anchor=north west][inner sep=0.75pt]  [font=\small]  {$1$};
		\draw (73.2,69.2) node [anchor=north west][inner sep=0.75pt]  [font=\small]  {$2$};
		\draw (97.6,69.2) node [anchor=north west][inner sep=0.75pt]  [font=\small]  {$3$};
		\draw (307,99) node [anchor=north west][inner sep=0.75pt]  [font=\small]  {$1$};
		\draw (381,42) node [anchor=north west][inner sep=0.75pt]  [font=\small]  {$2$};
		\draw (382,72) node [anchor=north west][inner sep=0.75pt]  [font=\small]  {$3$};
		\draw (340,126) node [anchor=north west][inner sep=0.75pt]  [font=\small]  {$D_{1}$};
		\draw (256,76) node [anchor=north west][inner sep=0.75pt]  [font=\small]  {$D_{2}$};
		\draw (359,86) node [anchor=north west][inner sep=0.75pt]  [font=\small]  {$p$};
		\draw (407,39) node [anchor=north west][inner sep=0.75pt]  [font=\small]  {$T_{p} N_{D_{1}} X$};
		\draw (380.93,122.19) node [anchor=north west][inner sep=0.75pt]  [font=\small]  {$N_{D_{1}} X$};
		\draw (243,36) node [anchor=north west][inner sep=0.75pt]    {$X$};
		\draw (146,106.5) node [anchor=north west][inner sep=0.75pt]    {$D_{1}$};
		\draw (92.5,1.5) node [anchor=north west][inner sep=0.75pt]    {$D_{2}$};

	\end{tikzpicture}
	\caption{An example of a stable 3-pointed FM grid expansion of $(\PP^2 | D_1 + D_2)$}
	\label{fig:3-FMexp}
\end{figure}

\begin{figure}[H]
	\centering

	\tikzset{every picture/.style={line width=0.75pt}} 
	
	\begin{tikzpicture}[x=0.75pt,y=0.75pt,yscale=-1,xscale=1]
		
		\draw  [draw opacity=0][fill={rgb, 255:red, 80; green, 227; blue, 194 }  ,fill opacity=0.28 ] (115.94,1042.55) -- (204.4,1042.55) -- (152.36,1132.68) -- (63.91,1132.68) -- cycle ;
		\draw  (53.91,1132.68) -- (153.91,1132.68)(115.87,1042.68) -- (58.13,1142.68) (149.79,1127.68) -- (153.91,1132.68) -- (144.02,1137.68) (106.83,1049.68) -- (115.87,1042.68) -- (116.83,1049.68)  ;
		\draw    (153.06,1045.82) -- (103.03,1132.46) ;
		
		\draw   (60.89,1129.47) -- (66.35,1135.06)(66.42,1129.53) -- (60.82,1134.99) ;
		\draw [color={rgb, 255:red, 0; green, 100; blue, 0 }  ,draw opacity=1 ][line width=1.5]    (103,1113.46) -- (103,1131.46) ;
		\draw [color={rgb, 255:red, 0; green, 100; blue, 0 }  ,draw opacity=1 ][line width=1.5]    (92.91,1103.37) -- (103,1113.46) ;
		\draw [color={rgb, 255:red, 0; green, 100; blue, 0 }  ,draw opacity=1 ][line width=1.5]    (103,1113.46) -- (113.4,1103.06) ;
		\draw   (100.89,1129.47) -- (106.35,1135.06)(106.42,1129.53) -- (100.82,1134.99) ;
		\draw [color={rgb, 255:red, 0; green, 100; blue, 0 }  ,draw opacity=1 ][line width=1.5]    (103.31,1092.97) -- (113.4,1103.06) ;
		\draw [color={rgb, 255:red, 0; green, 100; blue, 0 }  ,draw opacity=1 ][line width=1.5]    (113.4,1103.06) -- (123.8,1092.66) ;
		\draw   (215.27,1087.11) .. controls (215.27,1066.07) and (246.47,1049.02) .. (284.97,1049.02) .. controls (323.46,1049.02) and (354.66,1066.07) .. (354.66,1087.11) .. controls (354.66,1108.15) and (323.46,1125.2) .. (284.97,1125.2) .. controls (246.47,1125.2) and (215.27,1108.15) .. (215.27,1087.11) -- cycle ;
		\draw [color={rgb, 255:red, 0; green, 0; blue, 0 }  ,draw opacity=1 ]   (242.66,1071.2) .. controls (290.66,1062.2) and (285.66,1085.2) .. (328.66,1070.2) ;
		\draw  [color={rgb, 255:red, 0; green, 0; blue, 0 }  ,draw opacity=1 ][fill={rgb, 255:red, 126; green, 211; blue, 33 }  ,fill opacity=0.3 ] (346.19,1060.02) .. controls (344.36,1060.02) and (342.87,1066.53) .. (342.87,1074.56) .. controls (342.87,1082.6) and (344.36,1089.11) .. (346.19,1089.11) .. controls (348.03,1089.11) and (349.52,1095.62) .. (349.52,1103.66) .. controls (349.52,1111.69) and (348.03,1118.2) .. (346.19,1118.2) -- (319.59,1118.2) .. controls (321.43,1118.2) and (322.92,1111.69) .. (322.92,1103.66) .. controls (322.92,1095.62) and (321.43,1089.11) .. (319.59,1089.11) .. controls (317.76,1089.11) and (316.27,1082.6) .. (316.27,1074.56) .. controls (316.27,1066.53) and (317.76,1060.02) .. (319.59,1060.02) -- cycle ;
		\draw   (276.16,1089.3) -- (281.62,1094.9)(281.68,1089.37) -- (276.09,1094.83) ;
		\draw  [fill={rgb, 255:red, 0; green, 0; blue, 0 }  ,fill opacity=1 ] (331.08,1087.11) .. controls (331.08,1085.95) and (332.02,1085.02) .. (333.17,1085.02) .. controls (334.33,1085.02) and (335.27,1085.95) .. (335.27,1087.11) .. controls (335.27,1088.27) and (334.33,1089.2) .. (333.17,1089.2) .. controls (332.02,1089.2) and (331.08,1088.27) .. (331.08,1087.11) -- cycle ;
		\draw  [fill={rgb, 255:red, 74; green, 144; blue, 226 }  ,fill opacity=0.61 ] (350.78,1039.8) .. controls (353.25,1036.8) and (357.69,1036.37) .. (360.69,1038.84) -- (377,1052.27) .. controls (380,1054.74) and (380.43,1059.18) .. (377.96,1062.18) -- (347.44,1099.24) .. controls (347.44,1099.24) and (347.44,1099.24) .. (347.44,1099.24) -- (320.27,1076.86) .. controls (320.27,1076.86) and (320.27,1076.86) .. (320.27,1076.86) -- cycle ;
		\draw   (350.23,1055.23) -- (355.83,1060.97)(355.9,1055.3) -- (350.16,1060.9) ;
		\draw  [fill={rgb, 255:red, 74; green, 144; blue, 226 }  ,fill opacity=0.79 ] (399.23,1080.74) .. controls (403.07,1083.59) and (403.86,1089.02) .. (401.01,1092.85) -- (385.51,1113.69) .. controls (382.66,1117.53) and (377.24,1118.32) .. (373.4,1115.47) -- (335.27,1087.11) .. controls (335.27,1087.11) and (335.27,1087.11) .. (335.27,1087.11) -- (361.1,1052.38) .. controls (361.1,1052.38) and (361.1,1052.38) .. (361.1,1052.38) -- cycle ;
		\draw   (356.83,1077.59) -- (362.43,1083.33)(362.5,1077.66) -- (356.76,1083.26) ;
		\draw   (378.43,1078.39) -- (384.03,1084.13)(384.1,1078.46) -- (378.36,1084.06) ;
		\draw  [fill={rgb, 255:red, 0; green, 0; blue, 0 }  ,fill opacity=1 ] (343.81,1072.81) .. controls (343.81,1071.65) and (344.75,1070.71) .. (345.91,1070.71) .. controls (347.06,1070.71) and (348,1071.65) .. (348,1072.81) .. controls (348,1073.96) and (347.06,1074.9) .. (345.91,1074.9) .. controls (344.75,1074.9) and (343.81,1073.96) .. (343.81,1072.81) -- cycle ;
		
		\draw (52,1134.18) node [anchor=north west][inner sep=0.75pt]  [font=\small]  {$1$};
		\draw (79,1091.18) node [anchor=north west][inner sep=0.75pt]  [font=\small]  {$2$};
		\draw (95,1075.91) node [anchor=north west][inner sep=0.75pt]    {$3$};
		\draw (119,1075.91) node [anchor=north west][inner sep=0.75pt]    {$4$};
		\draw (272.27,1095.02) node [anchor=north west][inner sep=0.75pt]  [font=\small]  {$1$};
		\draw (346.27,1038.02) node [anchor=north west][inner sep=0.75pt]  [font=\small]  {$2$};
		\draw (332.13,1059.08) node [anchor=north west][inner sep=0.75pt]  [font=\small]  {$q$};
		\draw (305.27,1122.02) node [anchor=north west][inner sep=0.75pt]  [font=\small]  {$D_{1}$};
		\draw (221.27,1072.02) node [anchor=north west][inner sep=0.75pt]  [font=\small]  {$D_{2}$};
		\draw (324.27,1082.02) node [anchor=north west][inner sep=0.75pt]  [font=\small]  {$p$};
		\draw (372.27,1035.02) node [anchor=north west][inner sep=0.75pt]  [font=\small]  {$T_{p} N_{D_{1}} X$};
		\draw (346.19,1118.2) node [anchor=north west][inner sep=0.75pt]  [font=\small]  {$N_{D_{1}} X$};
		\draw (342.87,1074.56) node [anchor=north west][inner sep=0.75pt]  [font=\small]  {$3$};
		\draw (371.67,1084.18) node [anchor=north west][inner sep=0.75pt]    {$4$};
		\draw (397.33,1096.11) node [anchor=north west][inner sep=0.75pt]    {$T_{q} T_{p} N_{D_{1}} X$};
		\draw (158,1120) node [anchor=north west][inner sep=0.75pt]    {$D_{1}$};
		\draw (105.5,1019.5) node [anchor=north west][inner sep=0.75pt]    {$D_{2}$};

	\end{tikzpicture}
	
	\caption{An example of a stable 4-pointed FM grid expansion of $(\PP^2|D_1+D_2)$}
	\label{fig:4-FMexp}
\end{figure}

\subsection{Construction of $\logFM$ via an iterated blow-up} \label{subsection:iteratedbl}
The key idea is to relativise Fulton--MacPherson's construction of $\FM$ by working over the moduli space of stable $n$-pointed grid expansions $\logProd$. 

Given a subset $I\subset \oneton$ and a subset $J\subset \{1, 2, \ldots, n+1\}$, we consider the diagonals \[\Delta_I = \{(x_1, x_2, \dots, x_n) \in X^n \; | \; x_i = x_j \; \forall \; i, j \in I\} \subset X^n,\] and \[\Delta_J^+ = \{(x_1, x_2, \dots, x_{n+1}) \in X^{n+1} \; | \; x_i = x_j \; \forall \; i, j \in J\} \subset X^{n+1}.\] Let $\D_I$ and $\D^+_J$ be the strict transforms of $\Delta_I$ and $\Delta^+_J$ under the logarithmic modifications $\logProd \to X^n$ and $\fam \to X^{n+1}$ respectively. We define the space $\logFM$ as the iterated blow-up of $\logProd$ in the following sequence of centres, where at each step the dominant transform of the centre under existing blow-ups is taken:
\begin{equation}\label{seq:FM}
    \D({\{1,2\}}); \D({\{1,2,3\}}), \D({\{1,3\}}), \D({\{2,3\}}); \dots; \D({\{1,2,\dots,n\}}), \dots, \D({\{n-1,n\}}). \tag{$*$}
\end{equation}

The boundary divisor of $\logFM$, i.e. the complement of $\logConf$, consists of two types of components: 
\begin{itemize}
	\item strict transforms of boundary divisor components of $\logProd$, and
	\item exceptional divisors $D(J)$ corresponding to the blow-up centre $\D(J)$.
\end{itemize}


Consider the fibre product $\logFM \times_{\logProd} \fam$. For any $j \in J$, the section $\mathbf{s}_j: \logProd \to \fam$ gives rise to a closed immersion of the divisor $D(J) \subseteq \logFM$ into $\logFM \times_{\logProd} \fam$. Call the image $s(J) \subseteq \logFM \times_{\logProd} \fam$.

The scheme $\logFMfam$ is defined as the iterated blow-up of $\logFM \times_{\logProd} \fam$ in the sequence of strict transforms of sections
\begin{equation}\label{seq:FM+}
	s({\{1,2,\ldots, n\}}); s({\{1,2,\ldots, n-1\}}), \ldots; s({\{1,2\}}), \ldots, s({\{n-1,n\}}). \tag{$**$}
\end{equation}
The existence of a morphism $\piFM \colon \logFMfam \to \logFM$ is immediate.

\begin{observation} \label{obs:logFMfibres}
	Denote by $\fam^{\text{FM}}_J$ the iterated blow-up of $\logFM \times_{\logProd} \fam$ along the centres in (\ref{seq:FM+}) until $s(J)$. Denote by $\fam^{\text{FM}}_I$ the iterated blow-up at the preceding step, and denote by $\fam^{\text{FM}}_I |_{D(J)}$ the preimage in $\fam^{\text{FM}}_I$ over the divisor $D(J) \subset \logFM$. Let $s$ denote the image of the induced section $D(J) \to \fam^{\text{FM}}_I |_{D(J)}$; one can verify by local computations that $s$ is contained in the relative smooth locus of $\fam^{\text{FM}}_I \to \logFM$.
	
	Then the preimage of {$\fam^{\text{FM}}_J \to \logFM$} over the divisor $D(J) \subset \logFM$ is 
	\[\bl{s}{\fam^{\text{FM}}_I |_{D(J)}} \coprod_{\PP(T_s \fam^{\text{FM}}_I |_{D(J)})} \PP(T_s \fam^{\text{FM}}_I |_{D(J)} \oplus \mathbbm{1}).\]
	
	The divisors $D(J)$ of $\logFM$, along with the strict transforms of the existing strata on $\logProd$, induce a stratification on $\logFM$. Over each locally closed stratum $Z$, as the divisors $D(J)$ intersect transversely, we can extend the above observation to describe the preimage of $Z$ in $\logFMfam$. In particular, we see that the preimage in $\logFMfam$ over each point in the stratum, together with its $n$ points, is a stable $n$-pointed FM expansion of the expected combinatorial type.
\end{observation}

\begin{lemma} \label{lem:sections}
	The family $\piFM\colon \logFMfam \to \logFM$ comes with $n$ sections $t_i$.
\end{lemma}
\begin{proof}
	This follows from the universal property of blow-ups \cite[II.7.14]{hartshorne}. Using the notation from Observation~\ref{obs:logFMfibres}, at each step of the blow-up $\fam^{\text{FM}}_J$, suppose there are already $n$ sections to the previous blow-up $s_i \colon \FM \to \fam^{\text{FM}}_I$. We observe
	\[ s_i^{-1}(s(J)) = 
	\begin{cases}
		D(J),& \text{if } i \in J\\
		\emptyset,              & \text{otherwise},
	\end{cases}
	\] where in both cases the section $s_i$ lifts to $\fam^{\text{FM}}_J$. 
\end{proof}

\begin{lemma}\label{lem:logFMfibres}
	Given any closed point $f \colon \Spec k \to \logFM$, the pullback of \[\piFM \colon \logFMfam \to \logFM\] along $f$, together with $n$ points $f^*t_1, f^*t_2, \ldots, f^*t_n$, is a stable $n$-pointed FM grid expansion. 
\end{lemma}
\begin{proof}
	Follows from Observation~\ref{obs:logFMfibres}.
\end{proof}

\begin{proposition} \label{prop:log-smooth}
	The morphism $\piFM$ is logarithmically smooth.
\end{proposition}
\begin{proof}
	By construction, $\logFMfam$ is an iterated blow-up of $\logFM \times_{\logProd} \fam$. Since $\fam \to \logProd$ is logarithmically smooth by Proposition~\ref{prop:grid-exp}, so is \[\logFM \times_{\logProd} \fam \to \logFM.\] Thus, it remains to prove that the morphism is logarithmically smooth on the exceptional divisors of $\logFMfam$.
	
	Let $p$ be a point on an exceptional divisor $D(J)^+$ of $\FMdegenfam$. Let $q$ be the image of $p$ in $\FMdegen$; it lies on the exceptional divisor $D(J)$ of $\FMdegen$. We note that $D(J)$ intersects other boundary divisors transversely, and by Observation~\ref{obs:logFMfibres} the preimage of $D(J)$ consists of two boundary divisor components intersecting transversely. So one can choose neighbourhood $V$ of $q$ small enough such that $\pi$, restricted to the preimage of $V$, has local model $\AA^{d+1} \times \AA^{dn-1} \to \AA^1 \times \AA^{dn-1}$ given by $(x_1, \ldots, x_{d+1}) \mapsto x_1 x_2$ on the first factor and identity on the second factor, where $d$ denotes the dimension of $X$. This morphism is toric, hence logarithmically smooth.	
\end{proof}

\subsection{Modular description of $\logFM$} 
We show that the space $\logFM$ parametrises stable $n$-pointed FM grid expansions, and like $\logProd$, admits a stratification by combinatorial types.

\begin{theorem} \label{thm:logFM_fibres}\hfill
    \begin{enumerate}
        \item The scheme $\logFM$ is a simple normal crossings compactification of $\logConf$, i.e. the boundary divisor components intersect transversely.
        \item The morphism $\piFM\colon  \logFMfam \to \logFM$ is a logarithmically smooth, flat family of stable Fulton--MacPherson grid expansions with $n$ sections $t_i$ supported on the smooth locus of the Fulton--MacPherson grid expansions.
        \item The boundary of $\logFM$ is stratified by combinatorial types of planted forests.
    \end{enumerate}
\end{theorem}

\begin{proof}
	
	By Lemma~\ref{lem:logFMfibres}, the morphism $\piFM$ with $n$ sections $t_i$ is a family of stable Fulton--MacPherson grid expansions. Flatness of $\pi_{\mathrm{FM}}$ follows from miracle flatness: the variety $\logFM$ is smooth as it is an iterated blow-up of smooth variety $\logProd$ along smooth centres, the variety $\logFMfam$ is Cohen--Macaulay as it is an iterated blow-up of a Cohen--Macaulay variety $\fam \times_{\logProd} \logFM$ over smooth centres, and the fibres of $\pi_{\mathrm{FM}}$ are FM grid expansions which are equidimensional. 
	
	The stratification follows from the description of fibres of $\piFM$ in Statement (2), and also from the work of \cite[Section 2]{FM}. 
	
	It remains to prove Statement (1). We recall that the boundary divisor components of $\logProd$ intersect transversely, so do their strict transforms in $\logFM$.
	
	Proposition~\ref{prop:diagonal} gives a concrete description of the blow-up centres $\delta_I$ as diagonals in $\logProd$, and it follows that any intersection of a stratum $V_{\tau} \subset \logProd$ with a diagonal $\delta_I$ is transverse. By \cite[Lemma~2.9(v)]{wonderful}, after replacing $V_{\tau}$ and $\delta_I$ by their dominant transforms, $\widetilde{V_{\tau}}$ and $D(I)$ intersect transversely.
	
	Since the blow-up centres $\delta_I$ are diagonals in $\logProd$, they have the same building set structure as an arrangement of subvarieties of $\logProd$ as the arrangement of diagonals $\Delta_I \subset X^n$. See \cite{wonderful} for the theory of arrangements of subvarieties. Therefore, as we take the same blow-up sequence (\ref{seq:FM}) as in the construction of $\FM$, we conclude that the exceptional divisors $D(I)$ intersect transversely.
\end{proof}

We give a formula for the codimensions of boundary strata.
\begin{proposition}\label{prop:codimFM}
    Given a combinatorial type $\nu=(\F \to \F_{\S}, m, \T_{v_1}, \ldots, \T_{v_k})$, consider any $n$-pointed grid expansion $(\cal{P},m)$ with type $\tau = (\F \to \F_{\S}, m)$. Writing \[\cal{P} = \prod_{1\leq i \leq n} \cal{P}_i,\] we let $w_i$ be the number of vertices on $\cal{P}_i$ apart from the origin. Then the codimension of the strata $W_{\nu}$ in $\logFM$ is \[\sum_{1\leq i\leq n}w_i + \sum_{1\leq j\leq k} (|V(\T_{v_j})| - 1),\] where $|V(\T_{v_j})|$ denotes the number of vertices in $\T_{v_j}$.
\end{proposition}
\begin{proof}
	We recall that a locally closed stratum $W_{\nu}$ is an intersection of the strict transform of the locally closed stratum $V_{\tau} \subseteq \logProd$ and boundary divisors of the form $D(J)$. 
	
    Let $V_{\tau}$ be the locally closed stratum in $\logProd$ of combinatorial type $\tau$. Consider the subscheme $W_{\tau}$ which is the strict transform of $V_{\tau}$ under the map $\logFM \to \logProd$. Then $W_{\tau}$ has codimension equal to the codimension of $V_{\tau}$ in $\logProd$, which gives the first term in the sum. 
    By Observation~\ref{obs:logFMfibres}, the total number of non-root vertices of trees $\T_{v_i}$ is equal to the number of divisors $D(J)$ which contain $W_{\nu}$. Alternatively, we see that each vertex in the tree $\T_{v_i}$, apart from the root vertex, corresponds to an added component of the form $\PP(T_x Y_{v_i} \oplus \mathbbm{1})$ (or a blow-up), and the point configurations on this component are considered modulo translation and homothety, hence contributing codimension 1 to the boundary stratum. This gives the second term in the sum.
\end{proof}

\subsection{Remarks about functoriality and birational models}
Functoriality does not hold in general for Fulton--MacPherson spaces; this is because functoriality fails in general for blow-ups. A specific case where functoriality holds is where $f \colon X \hookrightarrow Y$ is an embedding; since the preimage of a diagonal $\Delta^{Y}_I$ in $Y^n$ is the diagonal $\Delta^X_I$ in $X^n$, by standard results e.g. \cite[Corollary~II.7.15]{hartshorne} this gives an embedding $\FMexample{n}{X} \hookrightarrow \FMexample{n}{Y}$.

A weaker version of functoriality holds for a log blow-up $f \colon (Y,E) \to (X,D)$: there is a map between the logarithmic FM spaces after passing to a further modification, i.e.~$\widetilde{\logFMexample{n}{Y}{E}} \to \logFM$. Note that there is no map $\logProdexample{Y}{E}{n} \to \logProdexample{X}{D}{n}$ due to incompatibility of grid subdivisions of $\Sigma_Y$ and $\Sigma_X$; the map of sets $\Pi_n(\Sigma_Y) \to \Pi_n(\Sigma_X)$ is not a map of cone complexes. However, we can pass to a common refinement of $n$-marked grid subdivisions of $\Sigma_Y$ and $\Sigma_X$, thus obtaining a new tropical moduli space $\widetilde{\Pi}_n$ which is a subdivision of both $\Pi_n(\Sigma_Y)$ and $\Pi_n(\Sigma_X)$. This gives logarithmic modifications $\widetilde{\logProdexample{Y}{E}{n}} \to \logProdexample{Y}{E}{n}$, $\widetilde{\logProdexample{Y}{E}{n}} \to \logProd$.
 
Since the $n$ marked points lie in the interior of expansion components, the preimage of the diagonal $\delta_I$ in $\logProdexample{Y}{E}{n}$ is the diagonal $\delta_I$ in $\widetilde{\logProdexample{Y}{E}{n}}$, and an analogous statement holds for $\logProd$. Therefore, there are maps $\widetilde{\logFMexample{n}{Y}{E}} \to \logFMexample{n}{Y}{E}$ and $\widetilde{\logFMexample{n}{Y}{E}} \to \logFM$; in fact one can check both maps are logarithmic modifications.

\begin{equation*}
	\begin{tikzcd}
		& \widetilde{\logFMexample{n}{Y}{E}} \arrow[ld, "\text{modification}"'] \arrow[rd, "\text{modification}"] &\\
		\logFMexample{n}{Y}{E} & & \logFM.
	\end{tikzcd}
\end{equation*}

A related question is, given two toric pairs $(X,D)$, $(Y,E)$ of the same dimension, how the logarithmic FM spaces $\logFM$ and $\logFMexample{n}{Y}{E}$ are related. A similar argument holds: we may pass to a common refinement of an $n$-marked grid subdivision of $\Sigma_X$ and of $\Sigma_Y$, and the result is some common logarithmic modification $\widetilde{\mathrm{FM}_n} \to \logFMexample{n}{Y}{E}$ and $\widetilde{\mathrm{FM}_n} \to \logFM$.

\subsection{Examples}
\subsubsection {Moduli space of $n$-pointed stable rational curves}\cite[Section 3.4.1]{KS}

For $n\geq 3$, the space $\logFMexample{{n-3}}{\PP^1}{0+1+\infty}$ coincides with the moduli space $\overline{M}_{0,n}$ of $n$-pointed stable rational curves.	

\section{Points on a degeneration of a variety} \label{section:points_degen}
Let $\omega\colon W \to B$ be a proper simple normal crossings degeneration of a smooth variety $X$ over a curve $B$, where the special fibre $W_0$ has at least two components $Y_1, Y_2, \ldots, Y_r$.

Generalising the construction in \cite{abram_fantechi}, we construct a logarithmically smooth degeneration of $X^n$ which is flat and have reduced fibres, where the special fibre parametrises $n$ points on expansions of $W_0$.

For simplicity, we endow $X$ with a trivial logarithmic structure, but the constructions in this paper can be adapted to the case where $X$ is given the divisorial logarithmic structure with respect to a simple normal crossings pair $(X,D)$.

\subsection{Simplex-lattice subdivisions}
The pair $(W, W_0)$ is a simple normal crossings pair, and the pair $(B, b_0)$ is a smooth pair. With respect to the induced divisorial logarithmic structures, the logarithmically smooth morphism $\omega\colon  W \to B$ induces the morphisms $\A_W \to \A_B$, and correspondingly, the morphism $w\colon \Sigma_W \subset \R_{\geq 0}^r \to \Sigma_B = \R_{\geq 0}$ given by \[(x_1, x_2, \ldots, x_r) \mapsto x_1 + x_2 + \cdots + x_r.\] 

We consider the \emph{height 1 slice} $\Delta\colonequals w^{-1}(1)$ of $\Sigma_W$. A section of the fibered product $W^{(n)}_{\omega} \to B$ gives $n$ points on $\Sigma_W$, and as the degeneration is described by the local model $t = y_1 y_2 \cdots y_r$, the $n$ points all lie on $\Delta$. 

Let $(u_1, u_2, \ldots, u_n)$ be an $n$-tuple of points on $\Delta$, and let $u_i^{(j)}$ denote the $j$th coordinate of $u_i$. The hyperplanes in $\R^r$ \[H_{ij} \colonequals \{(x_1, x_2, \ldots, x_r) \; |\; x_j = u_i^{(j)}\}\] induce a polyhedral subdivision $\cal{S}$ of $\Delta$. The points $u_1, u_2, \ldots, u_n$ are vertices of $\cal{S}$, and we consider the induced marking function \[m \colon \oneton \to V(\cal{S}).\]

The cone $\Upsilon = \cone(\cal{S})$ is a conical subdivision of $\Sigma_W \subset \R_{\geq 0}^r$. It inherits a map \[(\Upsilon, \Z^r) \to (\R_{\geq 0},\Z).\] We pass to a sublattice $h\Z$ at the base $(\R_{\geq 0},\Z)$ in order to make this map \emph{combinatorially reduced}; we take $h\Z$ to be the intersection of the images of lattices of cones of the source, and pull back the lattice structure $h\Z \subset \Z$ at the base to obtain a sublattice $N' \subset \Z^r$ at the source, so that the image of the lattice of every cone in the source is equal to the lattice $h\Z$ of the base. We endow $\cal{S}$ with the new integral structure.

\begin{definition}[Simplex-lattice subdivision]
	An \emph{$n$-marked simplex-lattice subdivision} is a tuple $(\cal{S}, m)$, where $\cal{S}$ is a simplex-lattice subdivision induced by a tuple $(u_1, u_2, \ldots, u_n)$, and $m$ is the associated marking function.
\end{definition}

By construction, the set of $n$-marked simplex-lattice subdivisions biject with $\Delta^n$.
\begin{figure}[!htb]
	\centering
	\begin{minipage}{.35\textwidth}
		\centering
		\def\svgwidth{.9\textwidth}
\begingroup%
  \makeatletter%
  \providecommand\color[2][]{%
    \errmessage{(Inkscape) Color is used for the text in Inkscape, but the package 'color.sty' is not loaded}%
    \renewcommand\color[2][]{}%
  }%
  \providecommand\transparent[1]{%
    \errmessage{(Inkscape) Transparency is used (non-zero) for the text in Inkscape, but the package 'transparent.sty' is not loaded}%
    \renewcommand\transparent[1]{}%
  }%
  \providecommand\rotatebox[2]{#2}%
  \newcommand*\fsize{\dimexpr\f@size pt\relax}%
  \newcommand*\lineheight[1]{\fontsize{\fsize}{#1\fsize}\selectfont}%
  \ifx\svgwidth\undefined%
    \setlength{\unitlength}{267.71413007bp}%
    \ifx\svgscale\undefined%
      \relax%
    \else%
      \setlength{\unitlength}{\unitlength * \real{\svgscale}}%
    \fi%
  \else%
    \setlength{\unitlength}{\svgwidth}%
  \fi%
  \global\let\svgwidth\undefined%
  \global\let\svgscale\undefined%
  \makeatother%
  \begin{picture}(1,0.96455504)%
    \lineheight{1}%
    \setlength\tabcolsep{0pt}%
    \put(0,0){\includegraphics[width=\unitlength,page=1]{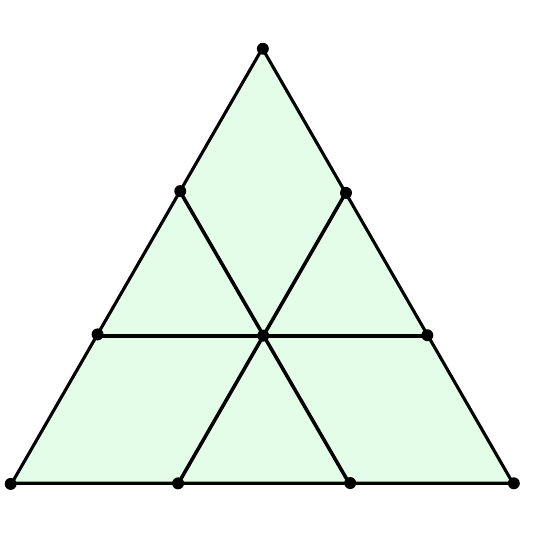}}%
    \put(0.45281381,0.93243114){\color[rgb]{0,0,0}\makebox(0,0)[lt]{\lineheight{0}\smash{\begin{tabular}[t]{l}$Y_1$\end{tabular}}}}%
    \put(-0.00049828,0.00891113){\color[rgb]{0,0,0}\makebox(0,0)[lt]{\lineheight{0}\smash{\begin{tabular}[t]{l}$Y_2$\end{tabular}}}}%
    \put(0.89853913,0.01641066){\color[rgb]{0,0,0}\makebox(0,0)[lt]{\lineheight{0}\smash{\begin{tabular}[t]{l}$Y_3$\end{tabular}}}}%
    \put(0.38215089,0.39005742){\color[rgb]{0,0,0}\makebox(0,0)[lt]{\lineheight{0}\smash{\begin{tabular}[t]{l}$u_1$\end{tabular}}}}%
    \put(0.50419749,0.38737512){\color[rgb]{0,0,0}\makebox(0,0)[lt]{\lineheight{0}\smash{\begin{tabular}[t]{l}$u_2$\end{tabular}}}}%
  \end{picture}%
\endgroup%

	\end{minipage}%
	\begin{minipage}{0.65\textwidth}
		\caption{A $2$-marked simplex-lattice subdivision.}
		\label{fig:degen_sub}
	\end{minipage}
\end{figure}

\subsubsection{Combinatorial types}
Given a polyhedral subdivision $\cal{S} \to \Delta$ with a marking function $m \colon \oneton \to V(\cal{S})$, there is an induced morphism of \emph{face posets} $\F_{\cal{S}} \to \F_{\Delta}$. The marking function $m$ on $V(\cal{S})$ induces a marking function \[m \colon \oneton \to \F_{\cal{S}}.\]
\begin{definition}
 The \emph{combinatorial type} of an $n$-marked simplex-lattice subdivision $(\cal{S},m)$ is the pair $\tau = (\F_{\cal{S}} \to \F_{\Delta}, m)$. 
\end{definition}

Given a tuple $(u_1, u_2, \ldots, u_n)$ of points in $\Delta$, let $H_{ij}^+$ and $H_{ij}^-$ denote the half spaces
\begin{align*}
	H_{ij}^+ &\colonequals \{(x_1, x_2, \ldots, x_r) \; |\; x_j \geq u_i^{(j)}\}, \\
	H_{ij}^- &\colonequals \{(x_1, x_2, \ldots, x_r) \; |\; x_j \leq u_i^{(j)}\}.
\end{align*}

\begin{fact}\label{fact:halfspaces}
	Consider a collection $\cal{C}$ of half spaces of the form $H_{ij}^+$ or $H_{ij}^-$. There exists a collection $\mathfrak{P}$ of linear inequalities of the coordinates $\{u_i^{(j)}\}_{i,j}$, such that the intersection of the half spaces in $\cal{C}$ with $\Delta$ is non-empty if and only if $\mathfrak{P}$ is satisfied.
\end{fact}

\begin{remark}\label{rem:combtype_degen}
	Consider an $n$-marked simplex-lattice subdivision $(\cal{S}, m)$ of a fixed combinatorial type $\tau$, induced by a tuple $(u_1, u_2, \ldots, u_n)$. Each polyhedron in the polyhedral complex is an intersection of a collection of half spaces of the form $H_{ij}^+$ or $H_{ij}^-$. It follows from Fact~\ref{fact:halfspaces} that the combinatorial type is completely characterised by the collections $\mathfrak{P}$ of linear inequalities, one for each polyhedron. 
\end{remark}


\subsection{Simplex-lattice expanded degenerations}

Just like the grid expansions in Section \ref{section:gridsub}, this induces an analogous geometric construction, which we will call a \emph{simplex-lattice expanded degeneration}.

The conical subdivision $\Upsilon \to \Sigma_W$ and the change of lattice $N' \subset \Z^r$ induces a modification of $W$ followed by a base change along the alteration $B' \to B$ induced by taking the sublattice $h\Z$ of the cone $\Sigma_B = (\Rgeq, \Z)$. Call the resulting scheme $W_\Upsilon$.
\begin{definition}[Simplex-lattice expanded degeneration] \label{def:expdegen}
	An $n$-marked simplex-lattice expanded degeneration is a weakly semistable degeneration $W_\Upsilon \to B'$ of $X$, induced by an $n$-marked simplex lattice subdivision $(\cal{S},m)$. 
\end{definition}
\begin{remark} \label{rem:special_fibre_expansion}
	The special fibre $Y_\Upsilon$ of $W_\Upsilon$ is an \emph{expansion} over $b_0$ induced by the polyhedral subdivision $\cal{S}$.
\end{remark}

See Figures~\ref{fig:degensub-fam} and \ref{fig:exp_degen} for an example.
                                                                                                       
\begin{figure}[!htb]
	\centering
	\begin{minipage}{.4\textwidth}
		\centering
		\def\svgwidth{.8\textwidth}
		
		\caption{A $2$-marked simplex-lattice subdivision.}
		\label{fig:degensub-fam}
	\end{minipage}\hfill
	\begin{minipage}{0.6\textwidth}
		\centering
		\def\svgwidth{\textwidth}
		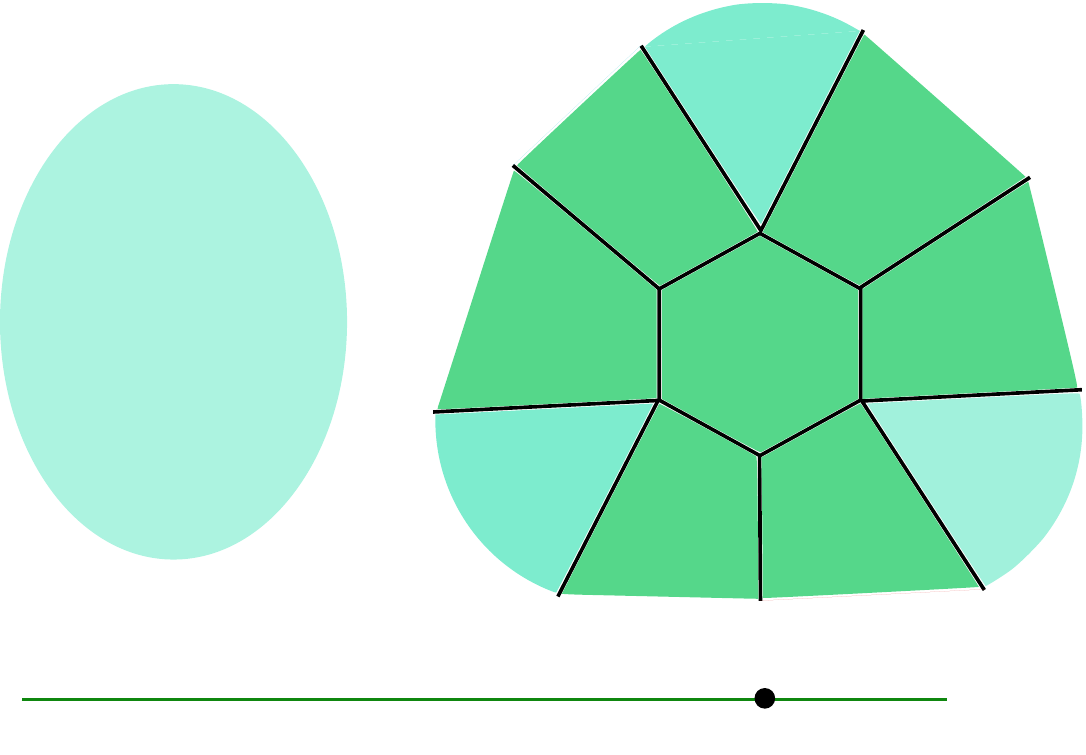
		\caption{A $2$-marked simplex-lattice expanded degeneration.}
		\label{fig:exp_degen}
	\end{minipage}
\end{figure}

\subsection{Tropical moduli space of simplex-lattice subdivisions}
We observe that the tropicalisation of $W^{(n)}_{\omega}$, is ${(\Sigma_W)}_{\Rgeq}^{(n)} = \cone(\Delta^n)$.

We have the following diagram, where $s_1, \ldots, s_n$ are the sections:
\begin{equation}\label{cd:naive}
	\begin{tikzcd}
		\cone(\Delta \times \Delta^n) \arrow[d, "p"']                       \\
		\cone(\Delta^n) \arrow[u, "{s_1,\ldots,s_n}"', bend right] \arrow[d, "h"']\\
		\R_{\geq 0}
	\end{tikzcd}
\end{equation}
The map $p$ is induced by the projection onto the second factor, and $s_i$ is the $i$th section induced by the map $\mathrm{pr}_i \times \mathrm{id}$, where $\mathrm{pr}_i\colon  \Delta^n \to \Delta$ is the projection onto the $i$th factor.

\noindent \textit{Lattices.} 
We take the lattices of $\cone(\Delta^n)$ and $\cone(\Delta \times \Delta^n)$ to be $N \colonequals (\Z^r)^{(n)}_{\Z}$ and $N^+ \colonequals (\Z^r)^{(n+1)}_{\Z}$ respectively. 
%
%
%

We construct the tropical moduli space $\Pi_n(\Delta)$ by a process similar to Section~\ref{section:tropmoduli_grid}. 
\begin{construction} \label{prop: semistable_degen}
	We construct a diagram of cone complexes
	\begin{equation} \label{diag:tropmoddegen}
		\begin{tikzcd} 
			(\tropfamdegen, (N')^+) \arrow[r]\arrow[d, "p"'] & (\cone(\Delta \times 	\Delta^n),(\Z^r)^{(n+1)}_{\Z}) \arrow[d, "p"']                       \\
			(\tropmoddegen, N') \arrow[r] \arrow[d, "h"'] \arrow[u, 	"{s_1,\ldots,s_n}"', bend right] & (\cone(\Delta^n),(\Z^r)^{(n)}_{\Z}) \arrow[u, "{s_1,\ldots,s_n}"', bend right] \arrow[d, "h"']\\
			(\R_{\geq 0},\Z) \arrow[r, "1 \mapsto h_{\mathrm{tot}}"'] & (\R_{\geq 0},\Z).
		\end{tikzcd}
	\end{equation}

	where each horizontal arrow is a subdivision followed by taking a sublattice.
	
	This diagram satisfies:
	\begin{enumerate}
		\item \textbf{\emph{Transversality.}} The image of each $s_i$ is a union of cones in $\tropfamdegen$. 
		\item \textbf{\emph{Combinatorial flatness.}} Every cone of $\tropfamdegen$ surjects under $p$ onto a cone of $\tropmoddegen$, and analogously for $h$.
		\item \textbf{\emph{Combinatorial reducedness.}} The image under $p$ of the lattice of every cone in $\tropfamdegen$ is equal to the lattice of the image cone in $\tropmoddegen$, and analogously for $h$.
	\end{enumerate}
	
\end{construction}
\begin{proof}
	\begin{enumerate}[wide, labelindent=0pt]
		\item \textbf{Subdivision of $\cone(\Delta^n)$ by combinatorial types.}   
		Consider coordinates $u_i^{(j)}$ of $\Sigma_W^n$. By Remark~\ref{rem:combtype_degen}, any combinatorial type $\tau$ is characterised by a collection of linear inequalities in terms of $u_i^{(j)}$. The resulting intersection of half spaces defines a polyhedron $\mu_{\tau}$ in $\Delta^n$. The collection of all such polyhedra forms a polyhedral subdivision $\tropmodpoly$ of $\Delta^n$, and taking the cone over induces a conical subdivision $\tropmoddegen$ of $\cone(\Delta^n)$.

		\smallskip
		\item \textbf{Subdivision of $\cone(\Delta \times \Delta^n)$ for transversality.} 
		The polyhedral complex $\Delta \times \tropmodpoly$ is a polyhedral subdivision of $\Delta \times \Delta^n$. Consider coordinates $x_1, x_2, \ldots, x_r$ of $\Sigma_W$, where $x_1 + x_2 + \cdots + x_r = 1$ on the first factor $\Delta$. The hyperplane arrangement $\{H_{ij}\}_{ij}$ defines a further polyhedral subdivision of $\Delta \times \tropmodpoly$, denoted $\tropfampoly$, and taking cone induces a conical subdivision $\tropfamdegen$ of $\cone(\Delta \times \Delta^n)$.
		
		For each $i$ and any polyhedron $\mu_{\tau}$ in $\tropmodpoly$, the image $s_i(\mu_{\tau})$ is \[
		\bigcap_{j=1}^r H_{ij} \cap \Delta \times \mu_{\tau},\] which is a polyhedron $\sigma$ in $\tropfampoly$. Therefore, $s_i(\cone(\mu_{\tau})) = \cone(\sigma)$, proving that $s_i$ is a morphism of cone complexes. For transversality, note that the image $s_i(\Delta^n)$ is precisely \[
		\bigcap_{j=1}^r H_{ij} \cap \Delta \times \Delta^n,\] which is a union of polyhedra, and we conclude by taking cones.

		\item \textbf{Combinatorial flatness.} 
		Given any non-empty polyhedron $\sigma$ in $\tropfampoly$, since $\tropfampoly$ is a subdivision of $\Delta \times \tropmodpoly$, the image $p(\sigma)$ is contained in a polyhedron in $\tropmodpoly$. Let $\mu_{\tau}$ be the minimal polyhedron containing $p(\sigma)$. It suffices to show that for any tuple $(u_1, u_2, \ldots, u_n) \in \Delta^n$, the polyhedron \[p^{-1}(u_1, u_2, \ldots, u_n) \cap \sigma\] is non-empty, but this follows immediately from Remark~\ref{rem:combtype_degen}. Combinatorial flatness of $h$ is immediate. 

		\smallskip
		\item \textbf{Combinatorial reducedness.} 
		We first show that $p$ is combinatorially reduced. For an arbitrary cone $\sigma$ in $\tropfamdegen$ and its image cone $\mu = p(\sigma)$ in $\tropmoddegen$, above each integral point $(u_1, u_2, \ldots, u_n)$ in $\mu$, the preimage \[p^{-1}(u_1, u_2, \ldots, u_n) \cap \sigma\] is a polyhedron $P$ in the polyhedral complex $p^{-1}(u_1, u_2, \ldots, u_n)$. Each vertex $v$ of $P$ is the intersection of $(r-1)$ coordinate hyperplanes, where each coordinate hyperplane is of the form $x_j=0$ or of the form $x_j = u_i^{(j)}$, so these $(r-1)$ coordinates out of $r$ coordinates of $v$ are integers. But the sum of all $r$ coordinates of $v$ is equal to the \emph{height} $h(u_1, u_2, \ldots, u_n)$, which is an integer. Therefore, the vertex $v$ is integral, and this proves that the image of the lattice of $\sigma$ is the lattice of $\mu$.
		
		However, the map $h\colon  \tropmoddegen \to \Rgeq$ is not combinatorially reduced: there exist cones $\mu$ in $\tropmoddegen$ such that the lattice point $1\in \Z$ is not in the image of the lattice of $\mu$. An example is the combinatorial type of the middle subdivision in Figure~\ref{fig:combtypesdegen}.

		To make $h$ combinatorially reduced, we replace the lattice $\Z$ of $\Rgeq$ by a sublattice. 
		For each cone $\tau$, there is a minimum height $h_\tau \in \N$ such that there exist integral points $(u_1, u_2, \ldots, u_n) \in \tau$ with height $h_\tau$. Since there are finitely many cones in $\tropmoddegen$, we can consider the lowest common multiple of $h_\tau$ over all cones $\tau$, denoted $h_{\mathrm{tot}}$. We replace the lattice $\Z$ of $\Rgeq$ by $h_{\mathrm{tot}}\Z$.
		
		We replace the lattices $N$ and $N^+$ with the sublattices $N' \colonequals N \times_{\Z} h_{\mathrm{tot}}\Z$ and $(N')^+ \colonequals N^+ \times_{\Z} h_{\mathrm{tot}}\Z$.
		
		
		We check that $p$ is still combinatorially reduced. For a point $(u_1, u_2, \ldots, u_n) \in \tropmoddegen$ with integral coordinates and height divisible by $h_{\mathrm{tot}}$, the points of $p^{-1}(u_1, u_2, \ldots, u_n)$ has the same height which is also divisible by $h_{\mathrm{tot}}$. Therefore, the previous argument proving $p$ is combinatorially reduced still holds.

	\end{enumerate}	
\end{proof}

We therefore have a diagram of cone complexes:

\begin{equation} \label{diag:tropmoddegen_sigma}
		\begin{tikzcd}
				\tropfamdegen \arrow[d, "p"'] \arrow[r] & \S_W                       \\
				\tropmoddegen \arrow[u, "{s_1,\ldots,s_n}"', bend right]\arrow[d, "h"'] & \\
				\R_{\geq 0}.
		\end{tikzcd} 
\end{equation} 

For simplicity of notation, we will write $\rho$ to denote a combinatorial type and its polyhedron.
\subsection{Stable $n$-pointed expanded degenerations}
From now on, expanded degenerations denote simplex-lattice expanded degenerations.

By the equivalence between the 2-category of cone stacks and the 2-category of Artin fans \cite[Theorem 6.11]{CCUW}, Diagram~(\ref{diag:tropmoddegen}) corresponds to an analogous diagram of cone complexes below. Here $\A$ denotes the stack $[\AA^1/\Gm]$ (the Artin fan of $B$), and $\A_W$ denotes the stack $[\AA^r/\Gm^r]$ (the Artin fan of $W$).
\begin{equation} \label{diag:simpexp}
	\begin{tikzcd} 
		\simpexpfam \arrow[r]\arrow[d, "p"'] & (\A_W)_\A^{(n+1)} \arrow[d, "p"']                       \\
		\simpexp \arrow[r] \arrow[d, "h"'] \arrow[u, 	"{s_1,\ldots,s_n}"', bend right] & (\A_W)_\A^{(n)} \arrow[u, "{s_1,\ldots,s_n}"', bend right] \arrow[d, "h"']\\
		\hat{\A} \arrow[r, "t \mapsto t^{h_{\mathrm{tot}}}"'] & \A.
	\end{tikzcd}
\end{equation}

We replace $\A$ with $\hat{\A}$, and take the corresponding base change of $B$ and $W$ throughout. The Artin fan $\simpexp$ is the \emph{stack of $n$-marked expanded degenerations over $\A$} with universal family $p\colon \simpexpfam \to \simpexp$.

We obtain $\simpexp_B$ and $\simpexpfamW$ from the base change of $\simpexp \to \A$ along $B \to \A$, and the base change of $\simpexpfam \to \A_W$ along $W \to \A_W$ respectively.


\begin{proposition} \label{prop:simpexp}
	The stack $\simpexp_B$ is the stack over $B$ of $n$-marked expanded degenerations of $W$ over $B$, together with a morphism $\simpexp_B \to B$ and a universal family $\eta\colon \simpexpfamW \to \simpexp_B$ that are both logarithmically smooth, flat and have reduced fibres, with $n$ sections $e_1, e_2, \ldots, e_n$. The morphism $\eta$ is representable.
\end{proposition}
\begin{proof}
	As $\simpexp$ is a logarithmic modification of $(\A_W)_{\A}^{(n)}$, the morphism $\simpexp_B \to B$ is a base change of the logarithmically smooth morphism $(\A_W)_{\A}^{(n)} \to \A$, which is logarithmically smooth. Similarly, logarithmic smoothness of $\eta$ follows from the logarithmic smoothness of $W \to B$, and the fact that logarithmic modifications are logarithmically smooth. By the result of Molcho \cite[Theorem 2.1.4]{uni_semistable} and Tsuji \cite[Theorem II.4.2]{tsuji}, even for non-simplicial $\tropmoddegen$, the combinatorial flatness of $\tropfamdegen \to \tropmoddegen$ guarantees flatness of $\eta$. Assuming an appropriate base change of $B$ and $W$, reducedness of fibres of $\eta$ follows from combinatorial reducedness. It is representable because the fibres of $\eta$ are families of expanded degenerations of $W$ over $B$. 
\end{proof}

\begin{definition} \label{def:pointedexpdegen}
	An \emph{$n$-pointed expanded degeneration $(W_\Upsilon, q_1, q_2, \ldots, q_n)$ of $W$ over $B$} is an $n$-marked expanded degeneration $(W_\Upsilon, (W_\Upsilon)_0) \to (B,b_0)$ together with $n$ sections $q_1, q_2, \ldots, q_n$, such that for each $i$, $q_i(b_0)$ lies on the the smooth locus of the special fibre. It is \emph{stable} if for each $i$, $q_i(b_0)$ lies on the $m(i)$th irreducible component of the special fibre. 
\end{definition}

\begin{figure}[!htb]
	\centering
	\begin{minipage}{.4\textwidth}
		\centering
		\def\svgwidth{.8\textwidth}
		
		\caption{A $2$-marked simplex-lattice subdivision.}
		\label{fig:degensub}
	\end{minipage}\hfill
	\begin{minipage}{0.6\textwidth}
		\centering
		\def\svgwidth{\textwidth}
		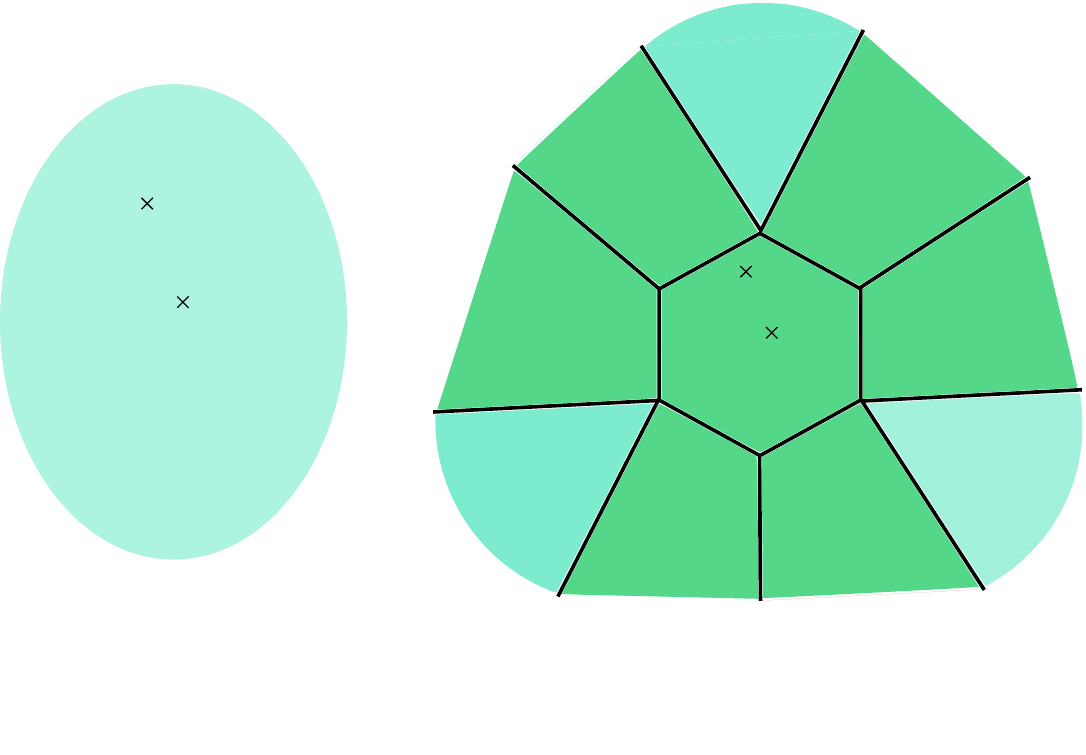
		\caption{A stable $2$-pointed simplex-lattice expanded degeneration.}
		\label{fig:pointed_exp_degen}
	\end{minipage}
\end{figure}

\begin{definition}
	A \emph{family of $n$-pointed expanded degenerations of $W \to B$} over a $B$-scheme $S$ is a family $\cal{W}_S \to S$ of $n$-marked expanded degenerations over $S$ pulled back from the universal expansion along some morphism $f: S\to \simpexp_B$, and $n$ sections \[s_1, s_2, \ldots, s_n\colon  S \to \cal{W}_S.\] It is \emph{stable} if:
	\begin{enumerate}
		\item for every section $\xi\colon B \to S$ of $S \to B$, the fibre $(\cal{W}_\xi, s_{i,\xi})$ is a stable $n$-pointed expanded degeneration of $W$ over $B$; 
		\item the morphisms $e_i \circ f$ and $\hat{f} \circ s_i$ 2-commute -- see diagram below.
	\end{enumerate}
	
	\begin{equation*}
		\begin{tikzcd}
			\cal{W}_S \arrow[r, "\hat{f}"] \arrow[d] \pullback{rd} & \simpexpfamW \arrow[r] \arrow[d] & W \arrow[d]\\
			S \arrow[u, "{s_1, s_2, \ldots, s_n}", bend left] \arrow[r, "f"'] & \simpexp_B \arrow[u, "{e_1, e_2, \ldots, e_n}"', bend right] \arrow[r] & B
		\end{tikzcd}
	\end{equation*}
\end{definition}

We define morphisms between families of stable $n$-pointed expanded degenerations using an analogue of Lemma~\ref{lem:2-iso}.

\begin{lemma}
	A 2-isomorphism $\lambda$ between two morphisms $f_1\colon S\to \simpexp_B$ and $f_2\colon S \to \simpexp_B$ induces a unique isomorphism on the pullbacks $\lambda^{\dagger}\colon  \cal{W}_{f_1} \to \cal{W}_{f_2}$. 
\end{lemma}

\begin{definition}
	A \emph{morphism} between families of stable $n$-pointed expanded degenerations $(f_1\colon  S_1 \to \simpexp_B)$ and $(f_2\colon  S_2 \to \simpexp_B)$ consists of
	\begin{enumerate}
		\item a morphism $\psi\colon S_1 \to S_2$ fitting into the 2-commutative diagrams
		\begin{equation*} 
			\begin{tikzcd}
				& \cal{W}_{f_2} \arrow[rd] \arrow[d]& \\                               
				& S_2 \arrow[rd, "f_2", ""{name=U,inner sep=1pt,below}] \arrow[u, "{s_{2,1}, s_{2,2}, \ldots, s_{2,n}}"', bend right]    & \simpexpfamW \arrow[d] \\
				\cal{W}_{f_1} \arrow[rru] \arrow[d] \arrow[ruu, "\Psi \circ \lambda^{\dagger}"]  &           & \simpexp_B  \\
				S_1 \arrow[rru, "f_1"', ""{name=D,inner sep=1pt}] \arrow[ruu, "\psi"'] \arrow[u, "{s_{1,1}, s_{1,2}, \ldots, s_{1,n}}", bend left] \arrow[Rightarrow, from=D, to=U, "\lambda"]&                                                                          &                     
			\end{tikzcd}
		\end{equation*}
		\item a 2-isomorphism $\lambda\colon  f_1 \Rightarrow f_2 \circ \psi$ and equalities $(\Psi \circ \lambda^\dagger) \circ s_{1,i} = s_{2,i} \circ \psi$, where $\Psi\colon \psi^*\cal{W}_{f_2} \to \cal{W}_{f_2}$ is the canonical morphism.
	\end{enumerate}
\end{definition}
We then define isomorphisms as follows:
\begin{definition}
	An isomorphism between stable $n$-pointed expanded degenerations $(f_1\colon S_1 \to \simpexp_B)$ and $(f_2\colon S_2 \to \simpexp_B)$ is the data of a morphism $(\psi, \lambda)$ such that $\psi\colon S_1 \to S_2$ is an isomorphism. 
\end{definition}
Note that $f_1$ and $f_2$ above necessarily factor through some substack $\B T_{\cone(\tau)}\times_{\A} B$ of $\simpexp_B$, where $\tau$ is some combinatorial type. 
We will see in Section~\ref{subsection:rubbertori_degen} that stable $n$-pointed expanded degenerations have no non-trivial automorphisms.

\subsection{Rubber actions and stability} \label{subsection:rubbertori_degen}
Like in Section~\ref{subsection:rubbertori}, the rubber actions give a geometric characterisation of the isomorphisms of stable $n$-pointed expanded degenerations.

Given a polyhedron $\tau$ corresponding to a combinatorial type, let $N_{\tau}$ be the kernel of the lattice map $N_{\cone(\tau)} \to N_{\Sigma_B} = \Z$, and let $T_{\tau}$ denote the \emph{rubber torus} $T_{\tau} \colonequals  N_{\tau} \otimes \Gm$. We note that \[(\B T_{\cone(\tau)})_B \colonequals  \B T_{\cone(\tau)}\times_{\A} B = \B T_{\cone(\tau)}\times_{\B \Gm} b_0,\] and this stack coincides with $\B T_{\tau}$ over the point $b_0$.  
 
Let $U_{\cone(\tau)}$ be the affine toric variety corresponding to $\cone(\tau)$, and $W_{\tau}$ be the logarithmic modification of $W \times_\A U_{\cone(\tau)}$ induced by the subdivision $(\Sigma_W \times_{\Rgeq} \cone(\tau))^{\dagger} \subset \tropfamdegen$. Let $Y_{\tau}$ be the fibre of $W_{\tau} \to U_{\cone(\tau)}$ over the torus fixed point. The restriction of the family $\simpexpfamW$ to the closed substack $\B T_{\tau}$ is \[\simpexpfamW_{\tau} = [Y_{\tau}/T_{\tau}] \to \B T_{\tau}.\] 

Consider two morphisms $f_1\colon S \to \simpexp_B$ and $f_2\colon  S \to \simpexp_B$ both factoring through the substack $\B T_{\tau}$. An element $\lambda \in \B T_{\tau}$ induces a fibrewise \emph{rubber action} $\lambda^{\dagger}\colon  \cal{W}_{f_1} \to \cal{W}_{f_2}$, as follows. Let $\varphi_v\colon N_{\tau} \to N_{\Delta}$ be the tropical position map, and $N_{\Delta}$ be the kernel of the lattice map $N_{\Sigma_W} \to N_{\Sigma_B} = \Z$, then  
the \emph{rubber torus} $T_{\tau}$ acts on each expanded component $Y_v$ by $\varphi_v \otimes \Gm$.

Similarly to Observation~\ref{obs:stability}, as the rubber action on the points is free, the stability condition in Definition~\ref{def:pointedexpdegen} guarantees that there are no non-trivial automorphisms.

\begin{figure}[!htb]
	\centering
	\def\svgwidth{\textwidth}
	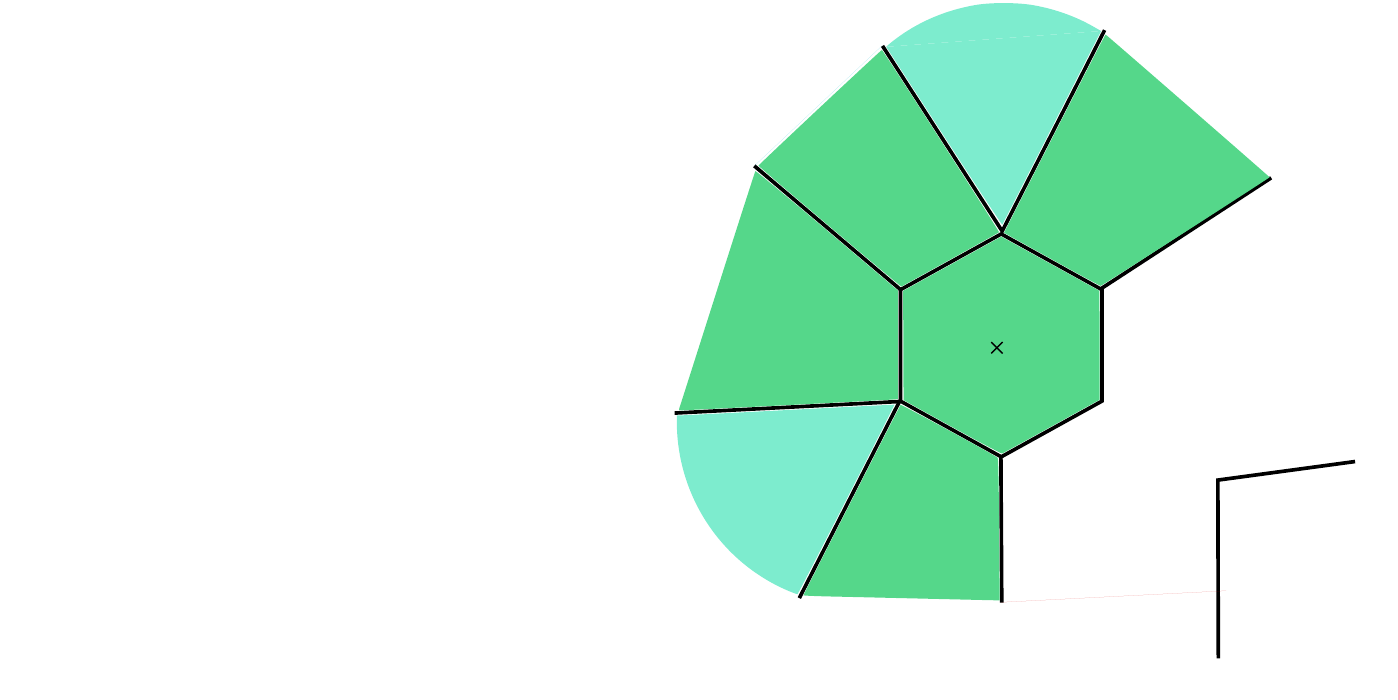
	\caption{The marked points on the subdivision are $u_1 = (t_1, 1-t_2, t_2-t_1)$ and $u_2 = (0, 1-t_2, t_2)$. The position of each vertex $v$ is determined by a number of independent parameters, and after tensoring by $\Gm$, these parameters determine the torus action on the component corresponding to $v$.}
	\label{fig:degen_rubber}
\end{figure}

\subsection{Construction of the moduli space of points $\pointsdegen$}

The goal of this subsection is to construct the space $\pointsdegen$, and to show that $\pointsdegen$ is the moduli space of stable $n$-pointed expanded degenerations equipped with a flat universal family $\pointsdegenfam \to \pointsdegen$.

\begin{definition}
	The scheme $\pointsdegen$ is the fibre product
	\begin{equation*}
		\begin{tikzcd}
			\pointsdegen \arrow[r] \arrow[d] \pullback{rd} & \simpexp_B \arrow[d] \\
			W_{\omega}^{(n)} \arrow[r] & (\A_W)_{\A}^{(n)} \times_{\A} B.
		\end{tikzcd}
	\end{equation*}
\end{definition}

\begin{definition}
	The scheme $\pointsdegenfam$ is constructed as the fibre product
	\begin{equation*}
		\begin{tikzcd}
			\pointsdegenfam \arrow[r] \arrow[d, "\text{flat}"] \pullback{rd} & \simpexpfamW \arrow[d, "\text{flat}"']\\
			\pointsdegen \arrow[r] \arrow[u, "{\mathbf{s}_1, \mathbf{s}_2, \ldots, \mathbf{s}_n}", bend left] & \simpexp_B \arrow[u, "{e_1, e_2, \ldots, e_n}"', bend right].
		\end{tikzcd}
	\end{equation*}
\end{definition}

\begin{remark} \label{rem:log-smooth}
	The morphism $\pointsdegenfam \to \pointsdegen$ is flat and logarithmically smooth because it is a pullback of the flat and logarithmically smooth universal family $\simpexpfamW \to \simpexp_B$. 
\end{remark}

\begin{remark}
	We note that $\pointsdegenfam$ is isomorphic to the fibre product
	\begin{equation*}
		\begin{tikzcd}
			\pointsdegenfam \arrow[r] \arrow[d] \pullback{rd} & \simpexpfam \arrow[d]\\
			W_{\omega}^{(n+1)} \arrow[r, "\text{strict}"] & (\A_W)^{(n+1)}_\A,
		\end{tikzcd}
	\end{equation*}
	hence is a logarithmic modification of $W_{\omega}^{(n+1)}$.
\end{remark}

We now state and prove the main theorem of this section.
\begin{theorem}
	Given a proper, simple normal crossings degeneration $\omega\colon  W \to B$ of $X$, the morphism $\pointsdegen \to B$ is a proper, flat, logarithmically smooth degeneration of $X^n$ with reduced fibres. The scheme $\pointsdegen$ represents the moduli stack over $B$ of stable $n$-pointed expanded degenerations of $W$. It is equipped with a universal family $\pointsdegenfam \to \pointsdegen$ which is flat, has reduced fibres, and has $n$ sections $\mathbf{s}_1, \mathbf{s}_2, \ldots, \mathbf{s}_n$. 
\end{theorem}
\begin{proof}
	
	To prove the second statement, we first observe that a stable $n$-pointed expanded degeneration $(f\colon  S \to \simpexp_B, s_1, s_2, \ldots, s_n\colon  S \to \cal{W}_S)$ induces a morphism $S \to W_{\omega}^{(n)}$, which induces a morphism $F\colon  S \to \pointsdegen$. We then note that two stable $n$-pointed expanded degenerations are isomorphic if and only if they induce the same morphism to $\pointsdegen$, again by the universal property of 2-fibre products of stacks.
	
	As the morphisms $\pointsdegen \to B$ and $\pointsdegenfam \to \pointsdegen$ are pulled back from $\simpexp_B \to B$ and $\simpexpfamW \to \simpexp_B$, the first and the last statements follow from Proposition~\ref{prop:simpexp}.
\end{proof}

\begin{remark}
	As opposed to the space $\logProd$, the space $\pointsdegen$ may be singular, e.g. for $r=3$ and $n=2$. In this example, the cone complex $\tropmoddegen$ is not simplicial.
\end{remark}

\subsubsection{Recursive description of diagonals}
Let $I$ be a subset of $\oneton$. Let $\delta \colon  W \hookrightarrow W^I_{\omega}$ be the diagonal morphism to the $I$-fold fibre product $W^I_{\omega}$ over $B$.

Consider the $I$th-\emph{diagonal over $\omega$}, namely \[\Delta_{\omega}(I) \colonequals  \delta(W)\times_B W^{\oneton \backslash I}_{\omega} \subset W^{(n)}_{\omega}.\] There is a canonical identification $\Delta_{\omega}(I) \cong W_{\omega}^{(n-|I|+1)}$, and we endow $\Delta_{\omega}(I)$ with the logarithmic structure on $W_{\omega}^{(n-|I|+1)}$, so that its tropicalisation is $(\Sigma_W)_{\Rgeq}^{(n-|I|+1)} = \cone(\Delta^{n-|I|+1})$.

Consider its strict transform $\delta_{\omega}(I)$ under the modification $\pointsdegen \to W_{\omega}^{(n)}$. We have the following analogue of Proposition~\ref{prop:diagonal}:
\begin{proposition} \label{prop:diagonal_degen}
	The strict transform $\delta_{\omega}(I)$ is the locus in $\pointsdegen$ where points indexed by $I$ coincide, i.e. the \emph{$I$-th diagonal} in $\pointsdegen$. 
\end{proposition}
\begin{proof}
	The proof is analogous to that of Proposition~\ref{prop:diagonal}.
\end{proof}

\subsection{Description of boundary strata} \label{section:boundary_degen}

The combinatorial types endow $\pointsdegen$ with a stratification, with the codimension of strata matching the dimension of the corresponding cone. The stratum $Z_{\tau}$ is the preimage of $\B T_{\tau} \hookrightarrow \simpexp$. We denote by $Z_{\tau}^+$ the restriction of the family $\pointsdegenfam \to \pointsdegen$ to $Z_{\tau}$.

\begin{definition}
	A \emph{rigid combinatorial type} is a combinatorial type $\rho$ whose associated polyhedron in $\tropmodpoly$ is a vertex.
\end{definition}

\begin{remark}\label{rem:rigid_rubber_trivial}
	For a rigid type $\rho$, since the polyhedron $\rho$ is a vertex, the rubber torus $T_{\rho}$ is trivial. The family over the point $\B T_{\rho} = \pt$, as described in Section~\ref{subsection:rubbertori_degen}, is an expansion
	\[\simpexpfam_{\rho} = Y_{\rho} \to \B T_{\rho} = \pt. \]
\end{remark}

We observe that the boundary divisor of $\pointsdegen$ is the sum of the irreducible components of the special fibre of $\pointsdegen \to B$, and they correspond to the strata with rigid combinatorial types. Figure~\ref{fig:combtypesdegen} are some examples of combinatorial types for $r = 3$, where $\Delta$ is an equilateral triangle.

\begin{figure}[!htb]
	\centering
	\def\svgwidth{\textwidth}
	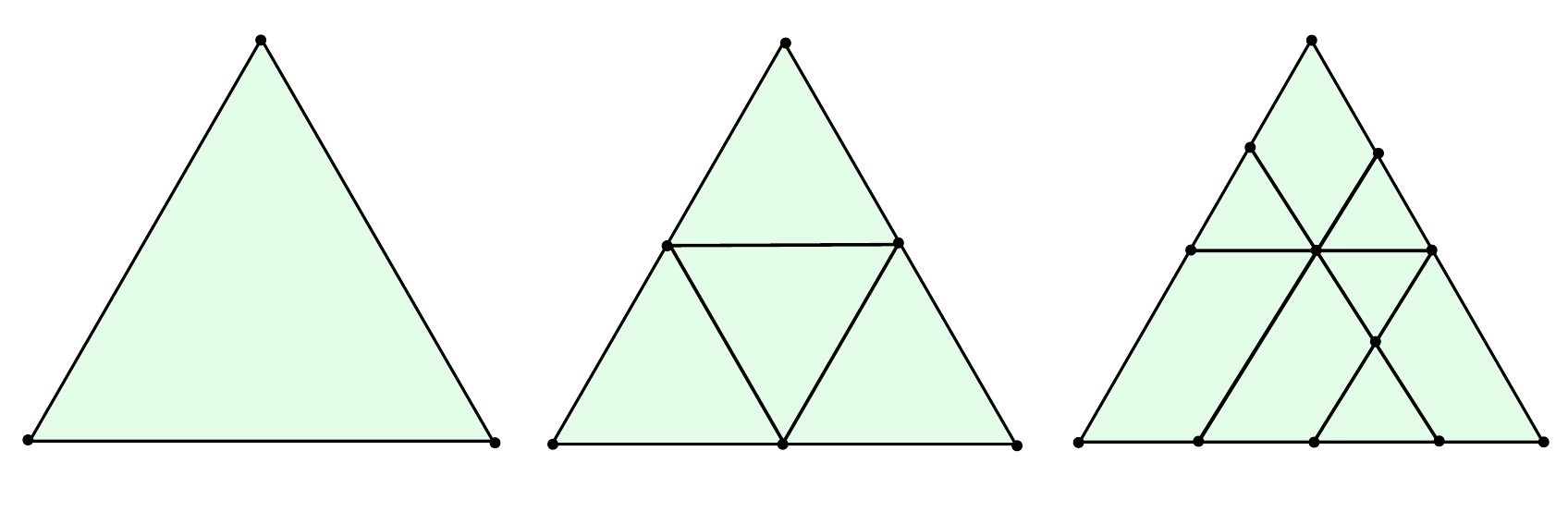
	\caption{The first two combinatorial types are rigid, whereas the third is not rigid.}
	\label{fig:combtypesdegen}
\end{figure}

We will prove in Section~\ref{section:degeneration} a ``gluing formula" for $\pointsdegen(\rho)$, in the style of \cite[Section 11]{log_enum}.

\section{Degenerations of Fulton--MacPherson spaces} \label{section:degeneration}
The goal of this section is to address Question~\ref{question:FMdegen}. Given a simple normal crossings degeneration $\omega\colon  W \to B$  of $X$, we now use constructions in the preceding sections to construct a degeneration of the Fulton--MacPherson configuration spaces $\FM$.
\subsection{Fulton--MacPherson expanded degenerations}
\begin{remark}
	Remark~\ref{rem:special_fibre_expansion} says that the special fibre $Y_{\cal{S}}$ of an expanded degeneration $W_{\Upsilon}$ is an expansion. As noted in Remark~\ref{rem:FMexp}, we extend Definition~\ref{def:FMgridexp} to define stable $n$-pointed FM expansions, where the underlying expansions are of the form $Y_{\cal{S}}$.
\end{remark}

\begin{construction} \label{const:FMexpdegen}
	Let $k$ be a non-negative integer. We consider an iterative construction $W_{\Upsilon}^{(k)}$, where $W_{\Upsilon}^{(0)}$ is the expanded degeneration $W_{\Upsilon}$.
	For $k\geq 1$, we let $W_{\Upsilon}^{(k)}$ be either \[\bl{x}{W_{\Upsilon}^{(k-1)}} \coprod_{\PP(T_x W_{\Upsilon}^{(k-1)})} \PP(T_x W_{\Upsilon}^{(k-1)} \oplus \mathbbm{1}),\] where $x$ is a section of $W_{\Upsilon}^{(k-1)} \to B$, supported on the smooth locus of each fibre of $W_{\Upsilon}^{(k-1)} \to B$; or \[\bl{y}{W_{\Upsilon}^{(k-1)}},\] where $y$ is a point on the special fibre of $W_{\Upsilon}^{(k-1)}$, supported on the smooth locus.
	
	The special fibre of $W_{\Upsilon}^{(k)} \to B$ is an FM expansion with underlying expansion the special fibre of $W_{\Upsilon}$.
\end{construction}
\begin{definition} \label{def:FMdegen}
	An \emph{$n$-pointed FM expanded degeneration} over $B$ is the data \[(W_{\Upsilon}^{\text{FM}}, s_1, s_2, \ldots, s_n),\] where $W_{\Upsilon}^{\text{FM}}$ is of the form $W_{\Upsilon}^{(N)}$ for some initial expanded degeneration $W_{\Upsilon}$ and non-negative integers $N$, and $s_1, s_2, \dots, s_n$ are distinct sections, where on each fibre of $W_{\Upsilon}^{\text{FM}} \to B$, the sections are supported on the smooth locus. 
	
	It is said to be \emph{stable} if
	\begin{itemize}
		\item $(W_{\Upsilon}^{\text{FM}}, s_1, s_2, \ldots, s_n)$ maps down to a stable $n$-pointed expanded degeneration \linebreak $(W_{\Upsilon}^{(0,0)},s_1', s_2', \ldots, s_n')$;
		\item over every point $b\in B \backslash \{b_0\}$, the fibre $(W_{\Upsilon}^{\text{FM}})_b$ together with the $n$ points $s_i(b)$ forms a stable $n$-pointed FM degeneration of $X$, in the sense of \cite[p. 194]{FM};
		\item over $b_0$, the fibre $(W_{\Upsilon}^{\text{FM}})_0$ together with the $n$ points $s_i(b_0)$ form a stable $n$-pointed FM expansion.

	\end{itemize}
\end{definition}

\begin{definition}
	Given a stable $n$-pointed FM expansion $(Y_{\mathcal{S}}^{\text{FM}}, p_1, p_2, \ldots, p_n)$, let the vertex set of $\mathcal{S}$ be $V(\mathcal{S}) = \{v_1, v_2, \ldots, v_k\}$.

	The \emph{combinatorial type} of $(Y_{\mathcal{S}}^{\text{FM}}, p_1, p_2, \ldots, p_n)$ is a tuple \[(\F_{\cal{S}} \to \F_{\Delta}, m, \T_{v_1}, \ldots, \T_{v_k}),\] where $(\F_{\cal{S}} \to \F_{\Delta}, m)$ is the combinatorial type of $Y_{\cal{S}}$, and each $\T_{v_i}$ is a rooted tree defined as follows:
	\begin{itemize}
		\item \textit{Vertices:} The tree $\T_{v_i}$ has a vertex for each component of $Y_{\cal{S}}^{\text{FM}}$ which contracts to the component of $Y_{\cal{S}}$ corresponding to $v_i$. The \emph{root vertex} of $\T_{v_i}$ is $v_i$.
		\item \textit{Edges:} Two vertices share an edge if the corresponding components intersect.
		\item \textit{Legs:} For each vertex $v$ of $\T_{v_i}$, attach a leg to $v$ for each marked point $p_j$ contained in the component of $Y_{\cal{S}}^{\text{FM}}$ corresponding to $v$.
	\end{itemize}
\end{definition}
\begin{remark}
	Given a stable $n$-pointed FM degeneration of $X$, its combinatorial type, as defined in \cite[p. 197]{FM}, is simply a collection of rooted trees with $n$ legs, where each non-root vertex is at least 3-valent. A tropical analogue of a stable $n$-pointed FM degeneration is a collection of rooted metric trees with $n$ legs satisfying the same valency condition. We call such a collection a \emph{forest}.
\end{remark}
\begin{figure}[!htb]
	\centering
	\begin{minipage}{.4\textwidth}
		\centering
		\def\svgwidth{.8\textwidth}
\begingroup%
  \makeatletter%
  \providecommand\color[2][]{%
    \errmessage{(Inkscape) Color is used for the text in Inkscape, but the package 'color.sty' is not loaded}%
    \renewcommand\color[2][]{}%
  }%
  \providecommand\transparent[1]{%
    \errmessage{(Inkscape) Transparency is used (non-zero) for the text in Inkscape, but the package 'transparent.sty' is not loaded}%
    \renewcommand\transparent[1]{}%
  }%
  \providecommand\rotatebox[2]{#2}%
  \newcommand*\fsize{\dimexpr\f@size pt\relax}%
  \newcommand*\lineheight[1]{\fontsize{\fsize}{#1\fsize}\selectfont}%
  \ifx\svgwidth\undefined%
    \setlength{\unitlength}{267.71413007bp}%
    \ifx\svgscale\undefined%
      \relax%
    \else%
      \setlength{\unitlength}{\unitlength * \real{\svgscale}}%
    \fi%
  \else%
    \setlength{\unitlength}{\svgwidth}%
  \fi%
  \global\let\svgwidth\undefined%
  \global\let\svgscale\undefined%
  \makeatother%
  \begin{picture}(1,0.96455504)%
    \lineheight{1}%
    \setlength\tabcolsep{0pt}%
    \put(0,0){\includegraphics[width=\unitlength,page=1]{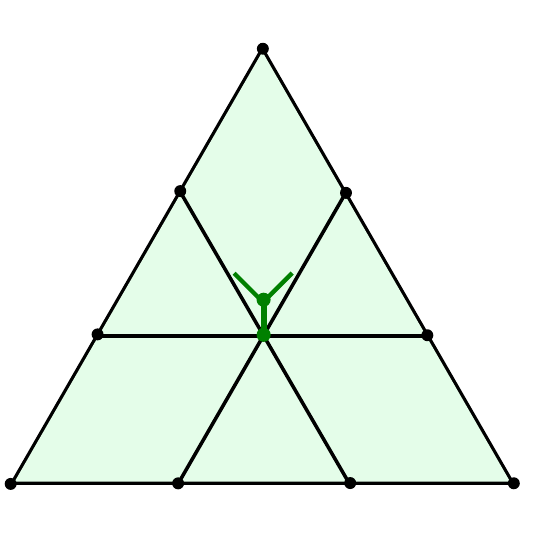}}%
    \put(0.45281381,0.93243114){\color[rgb]{0,0,0}\makebox(0,0)[lt]{\lineheight{0}\smash{\begin{tabular}[t]{l}$Y_1$\end{tabular}}}}%
    \put(-0.00049828,0.00891113){\color[rgb]{0,0,0}\makebox(0,0)[lt]{\lineheight{0}\smash{\begin{tabular}[t]{l}$Y_2$\end{tabular}}}}%
    \put(0.89853913,0.01641066){\color[rgb]{0,0,0}\makebox(0,0)[lt]{\lineheight{0}\smash{\begin{tabular}[t]{l}$Y_3$\end{tabular}}}}%
    \put(0.32984516,0.4611396){\color[rgb]{0,0,0}\makebox(0,0)[lt]{\lineheight{0}\smash{\begin{tabular}[t]{l}$u_1$\end{tabular}}}}%
    \put(0.54845615,0.45711615){\color[rgb]{0,0,0}\makebox(0,0)[lt]{\lineheight{0}\smash{\begin{tabular}[t]{l}$u_2$\end{tabular}}}}%
  \end{picture}%
\endgroup%

		\caption{A degenerate planted forest.}
		\label{fig:degen_planted}
	\end{minipage}\hfill
	\begin{minipage}{0.6\textwidth}
		\centering
		\def\svgwidth{\textwidth}
		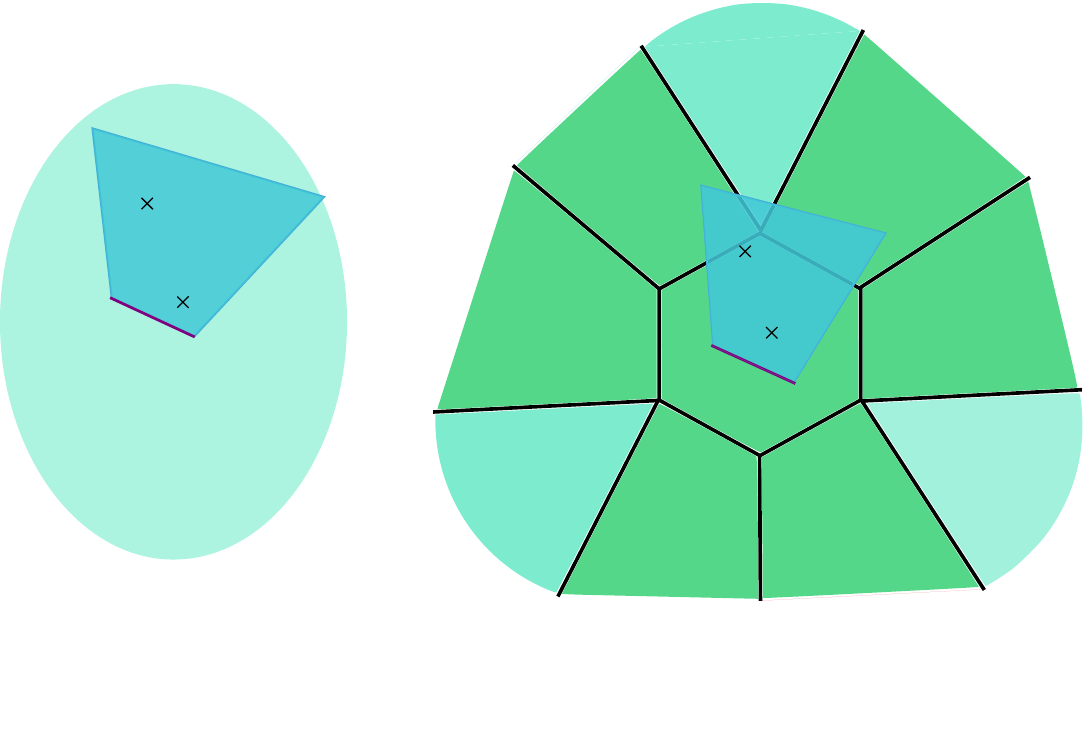
		\caption{A stable $2$-pointed FM expanded degeneration.}
		\label{fig:FM_exp_degen}
	\end{minipage}
\end{figure}

We define a tropical analogue of a stable $n$-pointed FM expansion of the form $Y_{\cal{S}}$, and its combinatorial type.
\begin{definition}
	A \emph{degenerate planted forest} is a tuple $(\cal{S}, m, \T^{\mathrm{m}}_{v_1}, \ldots, \T^{\mathrm{m}}_{v_k})$. The pair $(\cal{S}, m)$ is an $n$-marked simplex-lattice subdivision, and $v_1, v_2, \ldots, v_k$ are not necessarily distinct vertices of $\cal{S}$. Each $\T^{\mathrm{m}}_{v_i}$ is a rooted metric tree with legs and root $v_i$, such that the legs on $\T^{\mathrm{m}}_{v_i}$ are in bijection with markings in $m(v_i)$, and every vertex apart from the root is at least 3-valent.
\end{definition}
\begin{definition}
	The \emph{combinatorial type} of a degenerate planted forest $(\cal{S}, m, \T^{\mathrm{m}}_{v_1}, \ldots, \T^{\mathrm{m}}_{v_k})$ is $(\F \to \F_{\Delta}, m, \T_{v_1}, \ldots, \T_{v_k})$, where each $\T_{v_i}$ is obtained by forgetting the metric on $\T^{\mathrm{m}}_{v_i}$.
\end{definition}

\subsection{Construction of the degeneration}

Given a subset $I\subset \oneton$, we consider the $I$th-\emph{diagonal over $\omega$}, namely $\Delta_{\omega}(I)\colonequals  \delta(W)\times_B W^{\oneton \backslash I}_{\omega} \subset W^{(n)}_{\omega}$, where $\delta\colon  W \to W^I_{\omega}$ is the diagonal morphism to the fibre product $W^I_{\omega}$ over $\omega\colon W \to B$ indexed by $I$. 

We consider the strict transforms $\delta_{\omega}(I)$ 
under the logarithmic modification $\pointsdegen \to W^{(n)}_{\omega}$. We define the scheme $\FMdegen$ as the iterated blow-up of $\pointsdegen$ in the sequence of (dominant transforms) of centres (cf. Section~\ref{subsection:iteratedbl}):
\begin{equation} \label{seq:FMdegen}
	\Domega({\{1,2\}}); \Domega({\{1,2,3\}}), \Domega({\{1,3\}}), \Domega({\{2,3\}}); \ldots; \Domega({\{1,2,\ldots,n\}}), \ldots, \Domega({\{n-1,n\}}). \tag{$\dagger$}
\end{equation}


\begin{observation} \label{obs:special-fibre}
	The special fibre $(\FMdegen)_0$ is the preimage of $(\pointsdegen)_0$ under the iterated blow-up $\FMdegen \to \pointsdegen$. Since $(\pointsdegen)_0$ is transverse to blow-up centres $\delta_{\omega}(J)$, by Lemma~\ref{lem:strict_transverse}, it follows that $(\FMdegen)_0$ is also the strict transform. So, $(\FMdegen)_0$ is an iterated blow-up of $(\pointsdegen)_0$ along the special fibres of diagonals $(\delta_{\omega}(J))_0$.
\end{observation}

Consider the fibre product $\FMdegen \times_{\pointsdegen} \pointsdegenfam$. For any $j \in J$, the section $\mathbf{s}_j: \pointsdegen \to \pointsdegenfam$ gives rise to a closed immersion of the section $D_{\omega}(J)$ into $\FMdegen \times_{\pointsdegen} \pointsdegenfam$.

The scheme $\FMdegenfam$ is defined as the iterated blow-up of $\FMdegen \times_{\pointsdegen} \pointsdegenfam$ in the sequence of strict transforms of sections
\begin{equation}\label{seq:FMdegen+}
	D_{\omega}({\{1,2,\ldots, n\}}); D_{\omega}({\{1,2,\ldots, n-1\}}), \ldots; D_{\omega}({\{1,2\}}), \ldots, D_{\omega}({\{n-1,n\}}). \tag{$\dagger \dagger$}
\end{equation}

There is a morphism $\piFM^W\colon \FMdegenfam \to \FMdegen$ with $n$ sections, induced by the family $\pointsdegenfam \to \pointsdegen$ with its $n$ sections.

\begin{observation} \label{obs:FMdegenfibres}
	Denote by $\pointsdegenfam^{\text{FM}}_J$ the iterated blow-up of $\FMdegen \times_{\pointsdegen} \pointsdegenfam$ along the centres in (\ref{seq:FMdegen+}) until $D_{\omega}(J)$. Denote by $\pointsdegenfam^{\text{FM}}_I$ the iterated blow-up at the preceding step, and denote by $\pointsdegenfam^{\text{FM}}_I |_{D_{\omega}(J)}$ the preimage in $\pointsdegenfam^{\text{FM}}_I$ over the divisor $D_{\omega}(J) \subset \FMdegen$. Let $s$ denote the image of the induced section $D_{\omega}(J) \to \pointsdegenfam^{\text{FM}}_I |_{D_{\omega}(J)}$. One can check by local computations that $s$ lands in the relative smooth locus of $\pointsdegenfam^{\text{FM}}_I \to \FMdegen$.
	
	Then the preimage of {$\pointsdegenfam^{\text{FM}}_J \to \FMdegen$} over the divisor $D_{\omega}(J) \subset \FMdegen$ is 
	\[\bl{s}{\pointsdegenfam^{\text{FM}}_I |_{D_{\omega}(J)}} \coprod_{\PP(T_s \pointsdegenfam^{\text{FM}}_I |_{D_{\omega}(J)})} \PP(T_s \pointsdegenfam^{\text{FM}}_I |_{D_{\omega}(J)} \oplus \mathbbm{1}).\]
	
	The divisors $D_{\omega}(J)$ of $\FMdegen$, along with the strict transforms of the existing strata on $\pointsdegen$, induce a stratification on $\FMdegen$, and we can extend the above observation to describe the preimage of $Z$ in $\FMdegenfam$. In particular, we see that the preimage in $\FMdegenfam$ over each point in the stratum, together with its $n$ points, is a stable $n$-pointed FM expansion of the expected combinatorial type.
\end{observation}


\begin{lemma}\label{lem:FMdegenfibres}
	Given any morphism of $B$-schemes $f \colon B \to \FMdegen$, the pullback of \[\piFM^W \colon \FMdegenfam \to \FMdegen\] along $f$, together with $n$ sections $f^*\mathbf{s}_1, f^*\mathbf{s}_2, \ldots, f^*\mathbf{s}_n$, is a stable $n$-pointed FM expanded degeneration. 
\end{lemma}
\begin{proof}
	Let $Z$ be the stratum in $\FMdegen$ containing the image of the generic point of $B$. The image of $f$ either lies entirely in $Z$, or the image of $b_0$ lies in a deeper stratum, whose closure is the intersection of $Z$ with other divisors $\{D_{\omega}(J)\}$. We apply Observation~\ref{obs:FMdegenfibres} to each of these cases. In the latter case, we also note that the image of $f$ is transverse to each of the other divisors $D_{\omega}(J)$, therefore at the blow-up along each of these divisors $D_{\omega}(J)$, the pullback of the family $\pointsdegenfam^{\text{FM}}_J$ along $f:B\to \FMdegen$ coincides with the strict transform. We thus conclude that the pullback of $\piFM^W$ along $f$, with its $n$ sections, is a stable $n$-pointed FM expanded degeneration.
\end{proof}

We now state the main theorem, which will be proved at the end of this subsection.
\begin{theorem}\label{thm:FMdegen}
	Given a simple normal crossings degeneration $W\to B$ of $X$, the morphism $\FMdegen \to B$ is a proper, flat, logarithmically smooth degeneration of $\FM$ with reduced fibres. The morphism $\piFM^W\colon \FMdegenfam \to \FMdegen$ is a logarithmically smooth and flat family of stable Fulton--MacPherson expanded degenerations with $n$ sections $\mathbf{s}_1, \mathbf{s}_2, \ldots, \mathbf{s}_n$, where on each fibre, the sections are supported on the smooth locus.
\end{theorem}

\begin{observation}
	The space $\FMdegen$ admits a stratification by combinatorial types. 
	We consider a degenerate planted forest with the choice of a distinguished point $p_0$ anywhere on $\cal{S} \cup \T^{\mathrm{m}}_{v_1} \cup \dots \cup \T^{\mathrm{m}}_{v_k}$. The combinatorial types of such an object stratify $\FMdegenfam$. See Figure~\ref{fig:comb_types_FMdegenfam} for examples of such combinatorial types.
\end{observation}

\begin{figure}[h]
	\centering
	\includegraphics[width=.75\textwidth]{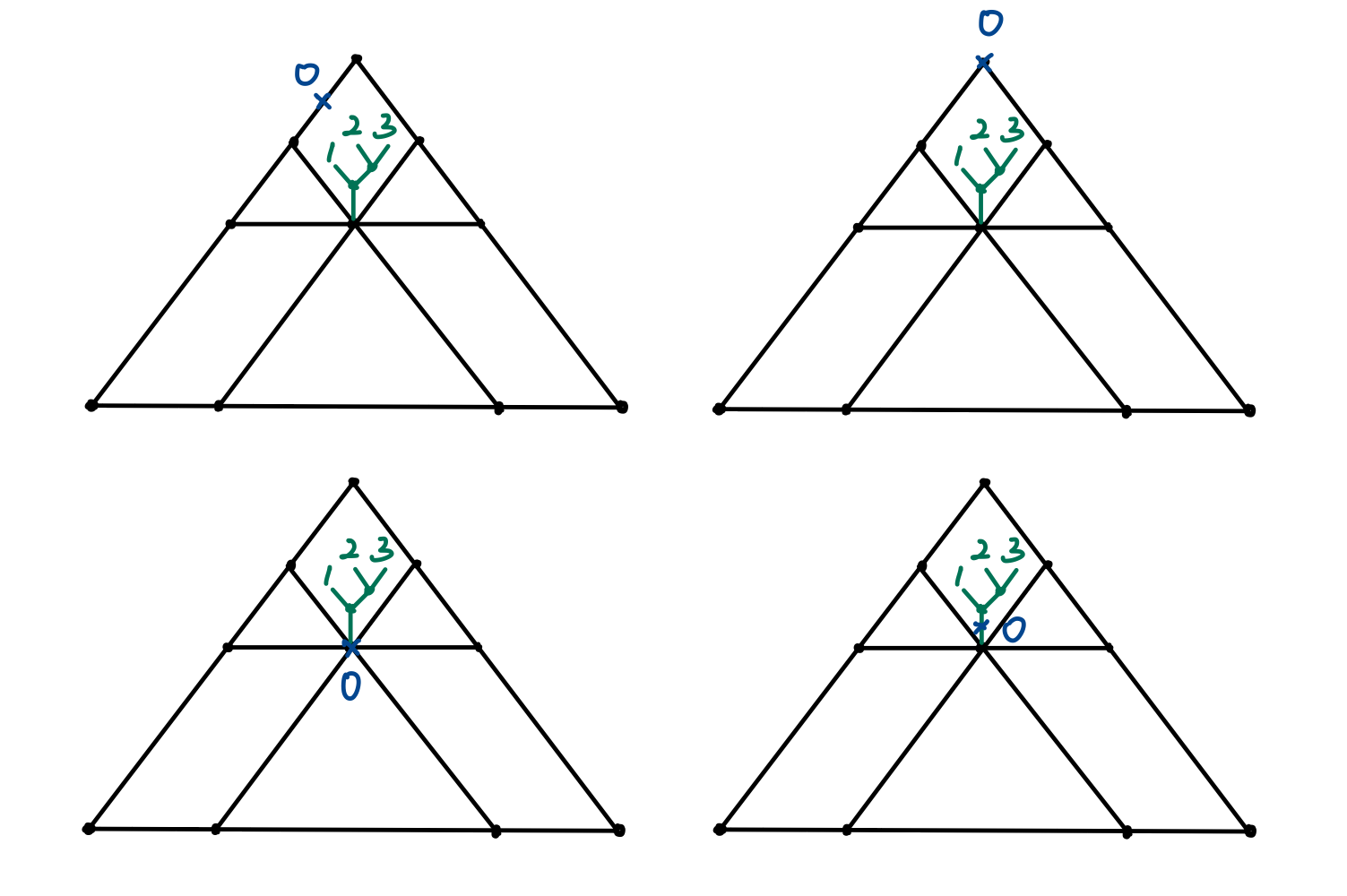}
	\caption{Combinatorial types in $\FMdegenfam$}
	\label{fig:comb_types_FMdegenfam}
\end{figure}

\subsubsection{A refined logarithmic structure}
Rather than pulling back the logarithmic structure from $\pointsdegen$ to $\FMdegen$, we define a refined logarithmic structure on $\FMdegen$, where the induced stratification coincides with the stratification by combinatorial types. In particular, the logarithmic interior is precisely $\mathrm{Conf}_n(X\backslash D) \times B\backslash \{b_0\}$.

There are two types of boundary divisor components of $\FMdegen$:
\begin{itemize}
	\item \textit{Horizontal divisors}: Divisors $D_{\omega}(J)$ of $\FMdegen$. They are the closures of strata with combinatorial type a single tree of exactly one (bounded) edge, on a trivial subdivision. They biject with the divisors of $\FM$.
	\item \textit{Irreducible components of the special fibre} $(\FMdegen)_0$: These are the strict transforms of the irreducible components of $(\pointsdegen)_0$, i.e. the divisors $\FMdegen(\rho)$ are the strict transforms of $\pointsdegen(\rho)$ under the iterated blow up $\FMdegen \to \pointsdegen$. They are the closures of strata with rigid combinatorial type with each attached tree at vertex $v_i$ having only distinct legs and no edges.
\end{itemize}

We observe that the divisors $\FMdegen(\rho)$ are all the irreducible components of the special fibre. This is because the total transform of $\pointsdegen(\rho)$ coincides with its strict transform, as $\pointsdegen(\rho)$ is transverse to the centres of the iterated blow-up $\FMdegen \to \pointsdegen$.


We similarly consider the two types of boundary divisor components of $\FMdegenfam$:
\begin{itemize}
	\item \textit{Preimage of horizontal divisors $D_{\omega}(J)$ of $\FMdegen$}: They biject with the components of the boundary divisor $\FM^+ \backslash (X \times \Conf)$.
	\item \textit{Irreducible components of the special fibre} $(\FMdegenfam)_0$: These are the strict transforms of the irreducible components of $\pointsdegenfam_0$.
\end{itemize}



These boundary divisors of $\FMdegen$ and $\FMdegenfam$ induce a \emph{divisorial logarithmic structure} on $\FMdegen$ and $\FMdegenfam$ respectively, and the logarithmic stratifications coincides with the existing stratifications by combinatorial types. The codimensions of the boundary strata are equal to the dimensions of their corresponding cones, and satisfy a formula similar to Proposition~\ref{prop:codimFM}.

In fact $\pi_{\mathrm{FM}}^W$ is a morphism of logarithmic schemes. We denote the tropicalisations of $\FMdegen$ and $\FMdegenfam$ by $\tropmoddegenFM$ and $\tropfamdegenFM$ respectively.

\begin{observation}
	The tropicalisation $\tropmoddegenFM$ admits a modular description. The morphism $\FMdegen \to B$ induces a morphism $\tropmoddegenFM \to \Rgeq$. The preimage of $0\in \Rgeq$ is the moduli space of forests, and the preimage of $1 \in \Rgeq$ is the moduli space of degenerate planted forests. Each combinatorial type $\tau_{\mathrm{FM}} = (\F \to \F_{\Delta}, m, \T_{v_1}, \ldots, \T_{v_k})$ where $\tau = (\F \to \F_{\Delta}, m)$ is the combinatorial type of the underlying subdivision, corresponds to a cone in $\tropmoddegenFM$ of the form $\cone(\tau) \times \prod_{i = 1}^{k} \Rgeq^{|E(\T_{v_j})|}$ (cf. Proposition~\ref{prop:codimFM}), where $|E(\T_{v_j})|$ denotes the number of edges in $\T_{v_j}$. 
\end{observation}
 
\begin{proposition} \label{prop:log-smooth-degen}
	The morphism $\pi_{\mathrm{FM}}^W$ is logarithmically smooth.
\end{proposition}
\begin{proof}
	Recall that $\FMdegenfam$ is an iterated blow-up of $\FMdegen \times_{\pointsdegen} \pointsdegenfam$. Since $\pointsdegenfam \to \pointsdegen$ is logarithmically smooth by Remark~\ref{rem:log-smooth}, so is \[\FMdegen \times_{\pointsdegen} \pointsdegenfam \to \FMdegen.\] Thus, it remains to prove that the morphism is logarithmically smooth on the exceptional divisors of $\FMdegenfam$.
	
	Let $p$ be a point on an exceptional divisor $D_{\omega}(J)^+$ of $\FMdegenfam$. Let $q$ be the image of $p$ in $\FMdegen$; it lies on the exceptional divisor $D_{\omega}(J)$ of $\FMdegen$. We noted in Observation~\ref{obs:FMdegenfibres} that the strict transform of the section \[D_{\omega}(J) \hookrightarrow \FMdegen \times_{\pointsdegen} \pointsdegenfam\] lands in the relative smooth locus, hence one can choose an open neighbourhood $V$ of $q$ small enough that, after passing to an open neighbourhood of $D_{\omega}(J)^+$, the preimage of $V$ in $\FMdegenfam$ is locally isomorphic to $\bl{D_{\omega}(J)}{\AA^d \times V}$, where $d$ is the dimension of $X$. 
	
	Let $\A_V$ denote the Artin fan of $\FMdegen$ restricted to the open neighbourhood $V$, and let $\A$ denote the Artin stack $[\AA^1/\Gm]$. The local model of $\pi_{\mathrm{FM}}^W$ over $V$ satisfies the following Cartesian diagrams of logarithmic stacks
	
	\begin{equation*}
		\begin{tikzcd}
			\bl{D_{\omega}(J)}{\AA^d \times V} \arrow[r] \arrow[d, "\pi_{\mathrm{FM}}^W"] \fspullback{rd} & \bl{s' }{\AA^d \times \A_{V}} \arrow[r] \arrow[d] \fspullback{rd} & \bl{s}{\AA^d \times \A} \arrow[d, "\text{log smooth}"]\\
			V \arrow[r, "\text{strict}"] & \A_{V} \arrow[r] & \A, 
		\end{tikzcd}
	\end{equation*}
	where $s$ denotes $\{0\} \times \B\Gm$, and $s'$ denotes the preimage of $\{0\} \times \B\Gm$ in $\AA^d \times \A_{V}$.
	
	The logarithmic structure on $\bl{s}{\AA^d \times \A}$ is given by the preimage of $\B\Gm$, consisting of two divisor components intersecting transversely. The morphism $\bl{s}{\AA^d \times \A} \to \A$ is logarithmically smooth because $\bl{(0,0)}{\AA^d \times \AA^1} \to \AA^1$ has local model $\AA^{d+1} \to \AA^1$ given by $(x_1, x_2, \dots, x_{d+1}) \mapsto x_1 x_2$, hence logarithmically smooth. As logarithmic smoothness is stable under base change in the category of logarithmic stacks, the morphism $\pi_{\mathrm{FM}}^W$ is logarithmically smooth at $p$.

\end{proof}

We are now ready to prove Theorem~\ref{thm:FMdegen}.

\begin{proof}[Proof of Theorem~\ref{thm:FMdegen}]
	The scheme $\FMdegen$ is proper because it is constructed as an iterated blow-up of a proper scheme $\pointsdegen$. By \cite[Corollary II.4.8(e)]{hartshorne}, the morphism $\FMdegen \to B$ is proper because $\FMdegen$ is proper and $B$ is separated. 
	
	We check that the morphism $\FMdegen \to B$ is flat. We observe that $B$ is smooth, and the fibres are equidimensional by Observation~\ref{obs:special-fibre}. Since $\pointsdegen$ has at worst toric singularities, hence it is Cohen--Macaulay. It follows that $\FMdegen$, which is an iterated blow-up of $\pointsdegen$ along smooth centres, is Cohen--Macaulay. By miracle flatness, the morphism $\FMdegen \to B$ is flat.
	
	To prove that $\FMdegen \to B$ is logarithmically smooth, it is equivalent to prove that $\FMdegen \to B \times_{\A} \cal{E}_n^{\mathrm{FM}}(\Delta)$ is smooth, where $\cal{E}_n^{\mathrm{FM}}(\Delta)$ is the Artin fan of $\tropmoddegenFM$. Since each stratum of $\FMdegen$ is smooth and of expected dimension, $B \times_{\A} \cal{E}_n^{\mathrm{FM}}(\Delta)$ is smooth and $\FMdegen$ is Cohen--Macaulay, the morphism is flat and has smooth fibres, hence smooth.
	 
	The morphism $\pi_{\mathrm{FM}}^W$ is a family of stable $n$-pointed Fulton--MacPherson expanded degenerations by Lemma~\ref{lem:FMdegenfibres}. 
	 
	As $\pi_{\mathrm{FM}}^W$ is logarithmically smooth by Proposition~\ref{prop:log-smooth-degen}, by Tsuji \cite[Theorem II.4.2]{tsuji} to prove flatness it suffices to prove that the corresponding map of tropicalisations is combinatorially flat. Here it is crucial to use the refined logarithmic structure.

%
	 
	 Over a cone in $\FMdegen$ corresponding to $\tau_{\mathrm{FM}}$, there can be three possible cones in $\FMdegenfam$ mapping to the cone in $\FMdegen$:
	 \begin{itemize}
	 	\item The point $p_0$ lies on $\mathsf{S}$. The combinatorial type is some $(\mathsf{S}^+, m^+, \T_{v_1}, \dots, \T_{v_k})$, where $\tau^+ = (\mathsf{S}^+, m^+)$ is a combinatorial type in the family $\tropfamdegen$ lying over the combinatorial type $\tau$ in $\tropmoddegen$. The corresponding cone is then $\cone(\tau^+) \times \prod_{i = 1}^{k} \Rgeq^{|E(\T_{v_i})|}$. 
	 	\item The point $p_0$ lies on a vertex of a tree $\T_{v_j}$. In this case, the cone corresponding to this combinatorial type is the same as the cone $\cone(\tau) \times \prod_{i = 1}^{k} \Rgeq^{|E(\T_{v_i})|}$, as there is no additional length parameter.
	 	\item The point $p_0$ lies on a bounded edge $e$ of a tree $\T_{v_j}$. In this case, the position of $p_0$ on the edge introduces an additional length parameter $e_1$ on the combinatorial type, therefore the cone is $\cone(\tau) \times \prod_{i \neq j} \Rgeq^{|E(\T_{v_i})|} \times \Rgeq^{|E(\T_{v_j})|+1}$.
	 \end{itemize}
	 We have listed all the cones in $\tropfamdegenFM$. For combinatorial flatness, we need to show that each cone in $\tropfamdegenFM$ surjects onto a cone in $\tropmoddegenFM$. In the first scenario, combinatorial flatness follows from the combinatorial flatness of $\tropfamdegen \to \tropmoddegen$, whereas the second scenario is trivial, as the cone map is the identity. As for the third scenario, we note that the map $\Rgeq^{|E(\T_{v_j})|+1} \to \Rgeq^{|E(\T_{v_j})|}$ is surjective, as the map $({\Rgeq^2})_{e_1,e_2}\to ({\Rgeq})_e$ given by $(e_1,e_2) \mapsto e_1 + e_2$ is surjective. The map $\Rgeq^{|E(\T_{v_j})|+1} \to \Rgeq^{|E(\T_{v_j})|}$ arises from the splitting of the edge $e$ into edges $e_1, e_2$ by $p_0$. 
	 
	 Therefore, the map $\tropfamdegenFM \to \tropmoddegenFM$ is combinatorially flat, and this concludes the proof.
\end{proof}

\subsection{Decomposing the special fibre} \label{subsection:degen_formula}
This subsection relies heavily on constructions in \cite[Section 10]{log_enum}.

Given a degeneration $W \to B$ of $X$, with $Y_1, Y_2, \dots, Y_r$ as the irreducible components of the special fibre $W_0$, we now relate each irreducible component $\FMdegen(\rho)$ of the special fibre of the degeneration $\FMdegen \to B$ to some logarithmic Fulton--MacPherson spaces of pairs $(Y_v, Y_v \cap D_{\rho})$, where $D_{\rho}$ denotes the singular locus of the associated expansion $Y_{\rho}$. 

\subsubsection{Setup and outline} \label{subsubsection:setup}
Given an irreducible component $\pointsdegen(\rho)$ of $(\pointsdegen)_0$ corresponding to a rigid type $\rho$ (see Section~\ref{section:boundary_degen}), we consider the divisor on $\pointsdegen(\rho)$ given by the intersection of $\pointsdegen(\rho)$ with the boundary divisor on $\pointsdegen$. This equips $\pointsdegen(\rho)$ with a divisorial logarithmic structure which is generically trivial, and its tropicalisation is then the star cone complex $\tropmoddegen(\cone(\rho))$. 

\noindent \textbf{Star of a cone.} The reference here is \cite[Section 1.6]{log_enum}. Let $\sigma$ be a cone in a cone complex $\Sigma$, with lattice $N_{\sigma}$. We define the \emph{overstar} of $\sigma$ to be the complement of cones in $\Sigma$ that do not contain $\sigma$ as a face; this is a topological space.
For every cone $\tau$ with lattice $N_{\tau}$ containing $\sigma$ as a face, we have an exact sequence
\[ 0 \to N_{\sigma} \to N_{\tau} \to N_{\tau}(\sigma) \to 0,
\] and after tensoring by $\R$, we have a quotient map $N_{\tau} \otimes {\R} \to N_{\tau}(\sigma) \otimes \R$. Let $\tau_{\circ}^{\sigma}$ be the intersection of $\tau$ with the overstar of $\sigma$. Let $\overline{\tau}$ be the image of $\tau_{\circ}^{\sigma}$ in $N_{\tau}(\sigma) \otimes \R$ under the quotient map; it is a cone in $N_{\tau}(\sigma) \otimes \R$ with lattice $N_{\tau}(\sigma)$, and one can verify that such cones, taken for all cones $\tau$ containing $\sigma$, form a cone complex.

\begin{definition}
	The \emph{star} of $\sigma$ is the cone complex $\Sigma(\sigma)$ consisting of all image cones $\overline{\tau}$ ranging over all cones $\tau$ containing $\sigma$ as a face. 
\end{definition}
This definition, which follows \cite[Definition~1.6.1]{log_enum}, outputs the same construction as the definition in \cite[(3.2.8)]{cls}, but is more suited for the arguments in Section~\ref{subsubsection:metric}.

As explained in \cite[Section~10.2]{log_enum}, if $(X,D)$ is a simple normal crossings pair with tropicalisation $\Sigma_X$, and $W \subset X$ is the closure of the stratum $W_{\sigma}$ associated to a cone $\sigma$ in $\Sigma$, then we can endow $W$ the logarithmic structure associated to the simple normal crossings pair $(W, W \backslash W_{\sigma})$, and the tropicalisation is $\Sigma(\sigma)$.\\

For a rigid combinatorial type $\rho$, let $(\cal{S}_{\rho},m)$ denote the unique $n$-marked simplex lattice subdivision with type $\rho$, which induces an expanded degeneration $W_{\rho} \to B$. Denote by $Y_{\rho}$ its special fibre, with singular locus $D_{\rho}$. For a vertex $v \in \cal{S}_{\rho}$, we denote by $Y_v$ the irreducible component of $Y_{\rho}$ associated to $v$, and $I_v = m^{-1}(v)$ be the subset of markings on $v$. We denote by $\Sigma_v$ the tropicalisation of $Y_v$, where $Y_v$ is equipped with the divisorial logarithmic structure with respect to $D_v \colonequals Y_v \cap D_{\rho}$. As before, the tropicalisation $\Sigma_v$ coincides with the star fan $\Sigma_{W_{\rho}}(\cone(v))$.

We construct a combinatorial cutting map
\[ \tropmoddegen(\cone(\rho)) \xrightarrow{\kappa} \prod_{v \in V(\cal{S}_{\rho})} \Pi_{I_v}(\Sigma_v),\] and consequently a logarithmic modification \[\pointsdegen(\rho) \to \prod_{v \in V(\cal{S}_{\rho})} \logProdexample{Y_v}{D_v}{I_v},\] followed by a lifting to a modification (the \emph{degeneration formula}) \[\FMdegen(\rho) \to \prod_{v \in V(\cal{S}_{\rho})} \logFMexample{{I_v}}{Y_v}{D_v}.\]

\subsubsection{The combinatorial cutting map}
For a combinatorial type $\tau$ with corresponding $\cone(\tau)$ in $\tropmoddegen$, let $\mathcal{G}_{\cone(\tau)}$ denote the preimage of $\cone(\tau)$ in $\tropfamdegen$. Given $\cone(\tau)$ of $\tropmoddegen$ with the ray $\cone(\rho)$ as one of the faces, we may restrict Diagram~(\ref{diag:tropmoddegen_sigma}) to a family of polyhedral complexes
\begin{equation*}
	\begin{tikzcd}
		\mathcal{G}_{\cone(\tau)}  \arrow[d, "p"'] \arrow[r] & \S_{W_{\rho}} \arrow[ldd]                      \\
		\cone(\tau) \arrow[u, "{s_1,\ldots,s_n}"', bend right]\arrow[d, "h"'] & \\
		\R_{\geq 0}.
	\end{tikzcd} 
\end{equation*}

Fix a vertex $v$ of $\cal{S_{\rho}}$; this gives a ray $\cone(v)$ in the polyhedral complex $\mathcal{G}_{\cone(\rho)}$, which is contained in the polyhedral complex $\mathcal{G}_{\cone(\tau)}$.
This gives a diagram of pairs of cone complexes and rays:
\begin{equation*}
	\begin{tikzcd}
		(\mathcal{G}_{\cone(\tau)}, \cone(v))  \arrow[d, "p"'] \arrow[r] & (\S_{W_{\rho}}, \cone(v))                      \\
		(\cone(\tau), \cone(\rho)). \arrow[u, "{s_j, \; j\in I_v}"', bend right] &
	\end{tikzcd}
\end{equation*}

Taking the star of each pair gives a diagram of cone complexes:
\begin{equation*}
	\begin{tikzcd}
		\mathcal{G}_{\cone(\tau)}(\cone(v))  \arrow[d, "p"'] \arrow[r] & \S_v                      \\
		\cone(\tau)(\cone(\rho)). \arrow[u, "{\overline{s}_j, \; j\in I_v}"', bend right] &
	\end{tikzcd}
\end{equation*}

The above diagram can be interpreted as a family over $\cone(\tau)(\cone(\rho))$ of points in $\Sigma_v$. Therefore, we have a map
\[\cone(\tau)(\cone(\rho)) \xrightarrow{\kappa_v} \Pi_{I_v}(\Sigma_v),
\] where $\Pi_{I_v}(\Sigma_v)$ is the tropical moduli space of $I_v$-labelled points. We observe that $\kappa_v$ is a map of cone complexes, as two point configurations on $\Delta_{t_1}, \Delta_{t_2}$ respectively with the same combinatorial type $\tau$ will induce point configurations on $\Sigma_v$ with the same combinatorial type. This is illustrated in Figure~\ref{fig:cutting_map}.

\begin{figure}[h]
	\centering
	\includegraphics[width=.7\textwidth]{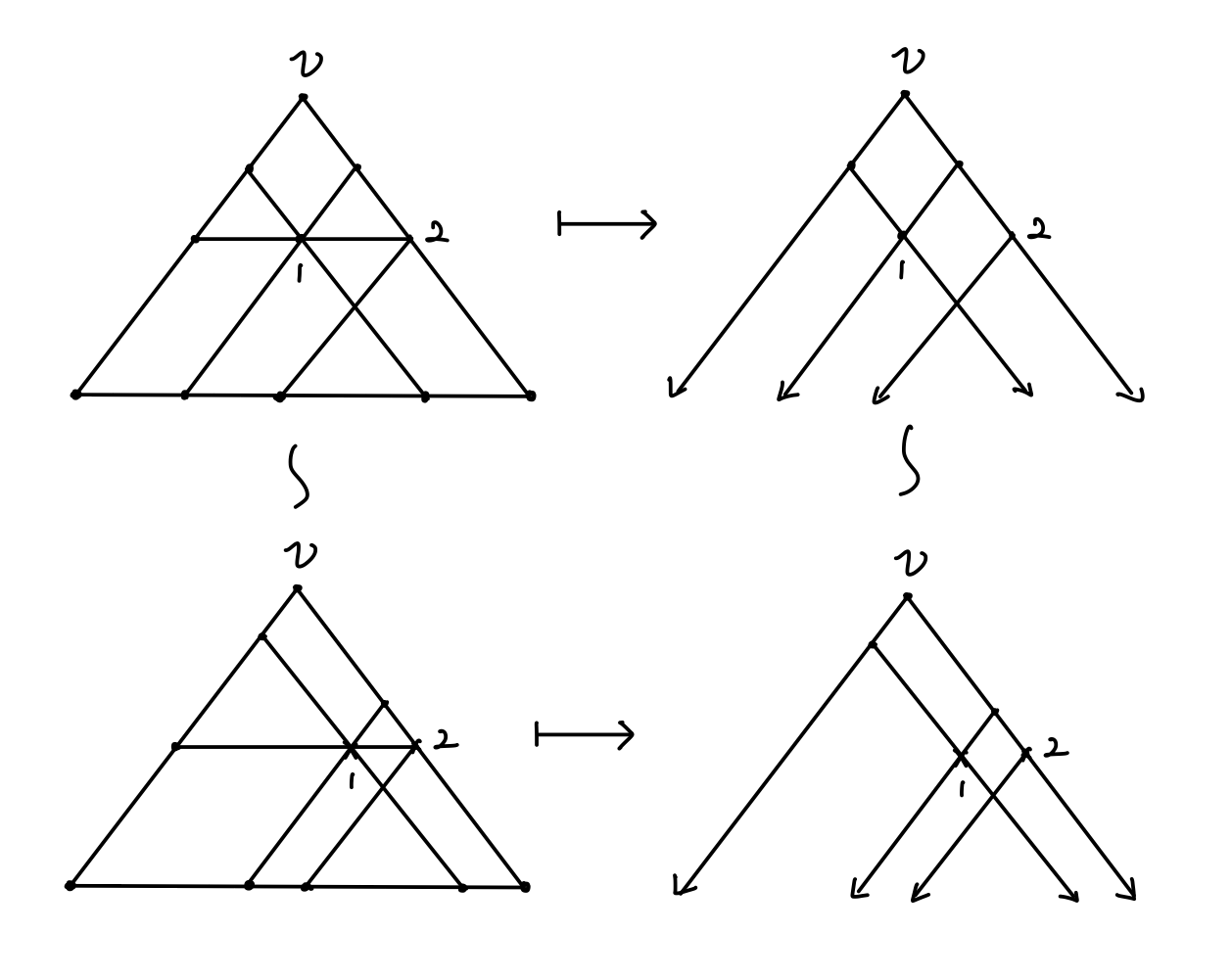}
	\caption{$\kappa_v$ is a map of cone complexes}
	\label{fig:cutting_map}
\end{figure}

In fact, the maps $\kappa_v$ are compatible with specialisation of combinatorial types, and therefore glue across cones to form a map of cone complexes
\[\tropmoddegen(\cone(\rho)) \xrightarrow{\kappa_v} \Pi_{I_v}(\Sigma_v).
\]

Combining these maps for all rays $r_i$ gives a \emph{combinatorial cutting map} 
\[\tropmoddegen(\cone(\rho)) \xrightarrow{\kappa} \prod_{v \in V(\cal{S}_{\rho})} \Pi_{I_v}(\Sigma_v).
\]

We have the following claim.
\begin{lemma}\label{lem:cutting_map}
	The combinatorial cutting map is a subdivision.
\end{lemma}

Following \cite[Section 10.4]{log_enum}, we first give an equivalent description of the map $\kappa_v$ in terms of metric geometry, then prove the lemma using this new description.

\subsubsection{The map $\kappa_v$, via metric geometry.} \label{subsubsection:metric}
We first recall that given any two points in the relative interior of $\cone(\tau)$, the fibres of the family $\tropfamdegen \to \tropmoddegen$ are polyhedral complexes with the same combinatorial type. We denote the resulting combinatorial polyhedral complex by $\mathsf{G}_\tau$, and we denote its 0-skeleton by $\mathsf{G}_\tau^0$.

Let $\mathcal{S}_{\rho,t}$ denote the preimage in $\tropfamdegen$ of the height $t$ point on the ray $\cone(\rho)$. Given $\cone(\tau)$ with a ray $\cone(\rho)$ as one of its faces, we subdivide the preimage in $\tropfamdegen$ of $\cone(\tau)$ such that for every point $q \in \cone(\tau)$ of height $t \in \Rgeq$, the slice $\mathcal{S}_q$ is a polyhedral subdivision of $\mathcal{S}_{\rho,t}$.

We recall the retraction $\psi\colon \mathsf{G}_\tau^0 \to \mathsf{G}_\rho^0$ defined in \cite[Section 10.4]{log_enum}, as follows. Fix a marked vertex $u_j$ of $\mathsf{G}_\tau^0$, then the image of the corresponding section $s_j\colon \cone(\tau) \to \cal{G}_{\cone(\tau)}$ is a cone in $\cal{G}_{\cone(\tau)}$ by combinatorial semistable reduction, mapping isomorphically by $p$ onto $\cone(\tau)$. Hence, there is a unique ray $u_\rho$ in $s_j$ which lies over $\cone(\rho)$, and $u_\rho$ is necessarily contained in $\cal{G}_{\cone(\rho)}$. This specifies a unique vertex $v$ in $\mathsf{G}_\rho^0$, and we let $\psi(u_j)$ be $v$. Note that the retraction $\psi$ is determined by the choice of the ray $\cone(\rho)$, and we see that $I_v$ is the set of labels of vertices in $\psi^{-1}(v)$. 

We next recall the construction of \emph{slice neighbourhoods} with respect to $\Sigma_{W_{\rho}} \to \Rgeq$. For a fixed value of $t\in \Rgeq$, we consider the slice $\cal{S}_{\rho,t}$, whose vertices biject with rays of $\Sigma_{W_{\rho}}$ which surject onto $\Rgeq$. Let $v$ be one of the vertices of $\cal{S}_{\rho,t}$. Consider the subspace $\cal{S}^o_{\rho,t}(v)$, which is the complement of faces of $\cal{S}_{\rho,t}$ which do not contain $v$. It is contained in the overstar of $\cone(v)$ (defined in Section~\ref{subsubsection:setup}), and therefore projects isomorphically onto its image in the star $\Sigma_v$. The set $\cal{S}^o_{\rho,t}(v)$ or its image in $\Sigma_v$ is called the \emph{slice neighbourhood} of 0 in $\Sigma_v$ at height $t$. 

For a point $\overline{q} \in \cone(\tau)(\cone(\rho))$, we describe the map $\overline{q}\mapsto \kappa_v(\overline{q})$. Choose a lift $q$ in $\cone(\tau)$ of height $t\in \Rgeq$. We note that $\mathcal{G}_q$ is a polyhedral subdivision of $\cal{S}_{\rho,t}$. We consider the points $s_1(q), \dots, s_n(q)$ in $\mathcal{G}_q$ contained in $\psi^{-1}(v)$. Such points lie in $\cal{S}^o_{\rho,t}(v)$, consequently they project to points in the slice neighbourhood in $\Sigma_v$. The data of such $I_v$-labelled points in $\Sigma_v$ defines a point in $\Pi_{I_v}(\Sigma_v)$, which coincides with $\kappa_v(\overline{q})$. We note that this construction is independent of the choice of the lift $q$, and sends a lattice point $\overline{q}$ to a lattice point $\kappa_v(\overline{q})$ .

We now prove Lemma~\ref{lem:cutting_map} by emulating the proof of \cite[Theorem 10.6.1]{log_enum}.

\begin{proof}[Proof of Lemma~\ref{lem:cutting_map}] 
	
	We prove that $\kappa$ is bijective on the supports, and on the lattices.
	
	\smallskip
	\noindent \textbf{Injectivity.} Consider two points $\overline{q}_1, \overline{q}_2$ in the support of $\cone(\tau)(\cone(\rho))$ such that $\kappa(\overline{q}_1)= \kappa(\overline{q}_2)$. Then $\kappa_v(\overline{q}_1)= \kappa_v(\overline{q}_2)$ for all $v \in V(\cal{S}_{\rho})$. We take $t \in \Rgeq$ sufficiently large such that for each $v$, the slice neighbourhood of height $t$ in $\Sigma_v$ contains the $I_v$-labelled points on $\Sigma_v$ corresponding to $\kappa_v(\overline{q}_1)= \kappa_v(\overline{q}_2) \in \Pi_{I_v}(\Sigma_v)$. 
	
	Choose corresponding lifts $q_1, q_2 \in \cone(\tau)$ over $t$. By the preceding argument, $\kappa_v(\overline{q}_1)$ is the data of $I_v$-labelled points on the slice neighbourhood of height $t$ in $\Sigma_v$, whose preimages are $I_v$-labelled points on \[\mathcal{G}_{q_1} \cap \cal{S}^o_{\rho,t}(v).\] Repeating the procedure for all $v$, we recover all the points $s_1(q_1), \dots, s_n(q_1)$ in $\mathcal{G}_{q_1}$, as the slice neighbourhoods $\cal{S}^o_{\rho,t}(v)$ form an open cover of $\cal{S}_{\rho,t}$. But as $\kappa_v(\overline{q}_1)= \kappa_v(\overline{q}_2)$ for all $v$, therefore $q_1 = q_2$, and $\overline{q}_1 = \overline{q}_2$.
	
	\smallskip
	\noindent \textbf{Surjectivity.} Given a collection of $I_v$-labelled points on $\Sigma_v$ over all $v$, say \[\big((\overline{u}_{v,j})_{j\in I_v}\big)_{v \in V(\cal{S}_{\rho}) } \in  \prod_{v \in V(\cal{S}_{\rho})} \Pi_{I_v}(\Sigma_v),\] as before, take $t\in \Rgeq$ sufficiently large such that for all $v$, the slice neighbourhood at height $t$ in $\Sigma_v$ contains the respective $I_v$-labelled points. These give rise to points $({u}_{v,j})_{j\in I_v}$ on $\cal{S}^o_{\rho,t}(v)$, and by taking union over $v$, we obtain $n$ points on $\Delta_t$, corresponding to a point $q \in \tropmoddegen$. By convex geometry, the point $q$ lies in the overstar of $\cone(\rho)$ for sufficiently large $t$. Therefore, $q$ descends to a point $\overline{q}$ in $\tropmoddegen(\cone(\rho))$, and $\kappa(\overline{q})$ recovers the collection $(\overline{u}_{ij})_{j\in I_i}$.
	
	\smallskip
	\noindent \textbf{Bijection on lattice points.} If $((\overline{u}_{v,j})_{j\in I_v})_v$ are lattice points, then in the proof of surjectivity we choose $t \in h_{\text{tot}}\Z$ (c.f. Proposition~\ref{prop: semistable_degen}) so that $\overline{q}$ is a lattice point.
\end{proof}

\begin{remark}
	To see that $\kappa$ is in general not an isomorphism of cone complexes, consider the following example showing a refinement of combinatorial types:
	\begin{figure}[h]
		\centering
		\includegraphics[width=\textwidth]{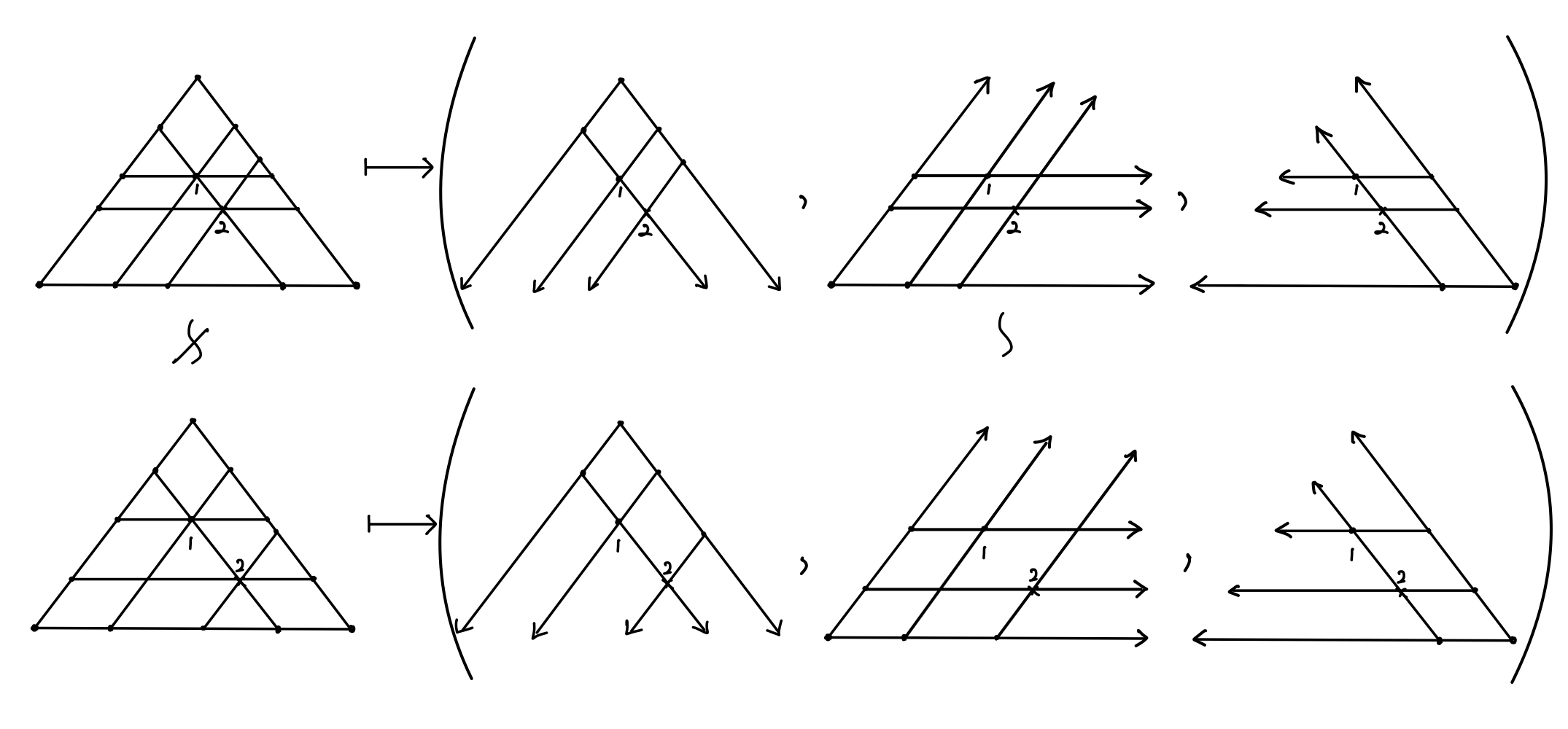}
		\caption{A refinement of combinatorial types}
		\label{fig:cutting_refinement}
	\end{figure}
\end{remark}

\subsubsection{The logarithmic modification}

\begin{theorem}\label{thm:points_degen_formula}
	The combinatorial cutting map induces a logarithmic modification \[\pointsdegen(\rho) \to \prod_{v \in V(\cal{S}_{\rho})} \logProdexample{Y_v}{D_v}{I_v}.\]
\end{theorem}
\begin{proof}
	The bulk of the proof lies in Lemma~\ref{lem:cutting_map}, and it only remains to lift the combinatorial subdivision to a logarithmic modification. 
	We denote the image of $\kappa$ as \[\prod_{v \in V(\cal{S}_{\rho})} \Pi_{I_v}^{\text{aux}}(\Sigma_v),\] where each $\Pi_{I_v}^{\text{aux}}(\Sigma_v)$ is a subdivision of $\Pi_{I_v}(\Sigma_v)$. As $\kappa$ is an isomorphism onto its image, this induces an isomorphism of Artin fans \[\simpexp(\rho) \cong \prod_{v \in V(\cal{S}_{\rho})} \cal{E}_{I_v}^{\text{aux}}(Y_v|D_v),\] and the subdivisions $\Pi_{I_v}^{\text{aux}}(\Sigma_v) \to \Pi_{I_v}(\Sigma_v)$ induce a logarithmic modification \[\prod_{v \in V(\cal{S}_{\rho})} \logProdexample{Y_v}{D_v}{I_v}_{\text{aux}} \to \prod_{v \in V(\cal{S}_{\rho})} \logProdexample{Y_v}{D_v}{I_v}.\] 
	
	We note that there is an isomorphism between the functors of points of $\pointsdegen(\rho)$ and $\prod_v \logProdexample{Y_v}{D_v}{I_v}_{\text{aux}}$: a family over $S$ of $I_v$-pointed ``auxiliary" expansions of $(Y_v|D_v)$, taken for each $v$, glue together to form a family over $S$ of special fibres of simplex-lattice expanded degenerations of type $\rho$. This gives the required logarithmic modification \[\pointsdegen(\rho) \to \prod_{v \in V(\cal{S}_{\rho})} \logProdexample{Y_v}{D_v}{I_v}.\]

\end{proof}
\begin{corollary}[Degeneration formula] \label{thm:degen_formula}
	The modification in Theorem~\ref{thm:points_degen_formula} lifts to a modification \[\FMdegen(\rho) \to \prod_{v \in V(\cal{S}_{\rho})} \logFMexample{{I_v}}{Y_v}{D_v}.\]
\end{corollary}
\begin{proof}
	We have a composition of modifications \[\FMdegen(\rho) \to \pointsdegen(\rho) \to \prod_{v \in V(\cal{S}_{\rho})} \logProdexample{Y_v}{D_v}{I_v},\] and we note that $\prod_{v \in V(\cal{S}_{\rho})} \logFMexample{{I_v}}{Y_v}{D_v}$ is itself an iterated blow-up of $\prod_{v \in V(\cal{S}_{\rho})} \logProdexample{Y_v}{D_v}{I_v}$. We observe that for each $w$, for any subset $J_w \subset I_w$, the preimage of the diagonals \[\delta(J_w) \times \prod_{w \neq v} (Y_v | D_v)^{[I_v]}\] under the modification $\FMdegen(\rho) \to \prod_{v \in V(\cal{S}_{\rho})} \logProdexample{Y_v}{D_v}{I_v}$ is a Cartier divisor. This is because each diagonal is one of the centres of the iterated blow-up $\FMdegen(\rho) \to \pointsdegen(\rho)$. Therefore, the morphism $\FMdegen(\rho) \to \prod_{v \in V(\cal{S}_{\rho})} \logProdexample{Y_v}{D_v}{I_v}$ factors through the morphism $\FMdegen(\rho) \to  \prod_{v \in V(\cal{S}_{\rho})} \logFMexample{{I_v}}{Y_v}{D_v}$.
	
	The morphism is proper by \cite[Corollary II.4.8(e)]{hartshorne}. It is birational, as it restricts to the identity on the dense open locus where all $n$ points on $W_0$ are distinct and lie away from $D$. Therefore the morphism is a modification. 
\end{proof}

\bibliographystyle{alpha}
\bibliography{references.bib}
\end{document}